\documentclass[english,a4paper,12pt,twoside]{report}
\usepackage[latin1]{inputenc}
\usepackage{amssymb}
\usepackage{amsthm}
\usepackage[all]{xy}
\usepackage{amsmath}
\usepackage{indentfirst}
\usepackage{dsfont}
\usepackage{amsfonts}
\usepackage{textcomp}
\usepackage{graphicx}
\usepackage{enumerate}
\usepackage{verbatim}
\usepackage{hyperref}
\usepackage{stmaryrd}
\usepackage{pifont}
\usepackage{eucal}

\newtheorem{lemma}{Lemma}[section]
\newtheorem{teo}[lemma]{Theorem}
\newtheorem{pro}[lemma]{Proposition}
\newtheorem{cor}[lemma]{Corollary}

\newtheorem{defi}[lemma]{Definition}

\newtheorem{ese}[lemma]{\bf Example}

\DeclareMathSymbol{\nmid}{\mathrel}{AMSb}{"2D}
\def\P{{\mathbb P}}
\def\N{{\mathbb N}}
\def\Z{{\mathbb Z}}
\def\Q{{\mathbb Q}}

\def\H{{\mathcal H}}
\def\A{{\mathcal A}}
\def\C{{\mathcal C}}
\def\L{{\mathcal L}}
\def\l{{\ell}}
\def\={\stackrel{\rm def}{=}}

\begin{document}

\begin{titlepage}

\begin{center}
\Large{UNIVERSITÀ DEGLI STUDI DI ROMA \\ ``TOR VERGATA''}\\
\Large{Facoltà di Scienze Matematiche, Fisiche e Naturali}\\
\Large{Dipartimento di Matematica}

\vspace{0.6cm}
\includegraphics{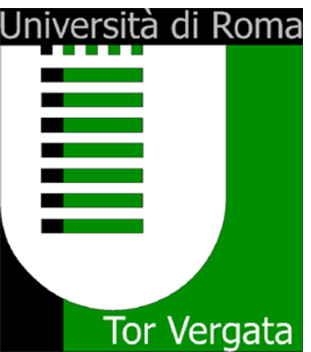}
\vskip 0,3cm
\Large{Dottorato di Ricerca in Matematica - Ciclo XXIV}\\
\vspace{0.4cm}
\huge{\textbf{Combinatorial properties of \\ Temperley--Lieb algebras}}
\vspace{0.4cm}

\huge{Alfonso Pesiri}\\

\vspace{2cm}

\Large{A.A. 2010/2011} 

\end{center}

\vspace{0.3cm}

\Large{Docente Guida: Prof. Francesco Brenti}

\vspace{1.2cm}

\Large{\textit{Coordinatore: Prof. Filippo Bracci}}





\end{titlepage}

\newpage

\pagenumbering{Roman}
\tableofcontents
\pagebreak
\pagenumbering{arabic}

\addcontentsline{toc}{section}{Introduction}
\vskip 2cm
{\bf\LARGE Introduction} 

\vskip 0,7cm

The Temperley--Lieb algebra $TL(X)$ is a quotient of the Hecke algebra $\H(X)$ associated to a Coxeter group $W(X)$, $X$ being an arbitrary Coxeter graph. It first appeared in \cite{tl-pcp}, in the context of statistical mechanics (see, \textrm{e.g.}, \cite{jo-pik}). The case $X=A$ was studied by Jones (see \cite{jo-har}) in connection to knot theory. For an arbitrary Coxeter graph, the Temperley--Lieb algebra was studied by Graham. More precisely, in \cite{gr-phd} Graham showed that $TL(X)$ is finite dimensional whenever $X$ is of type $A, B, D, E, F, H$ and $I$. If $X \not = A$ then $TL(X)$ is usually referred to as the generalized Temperley--Lieb algebra.\\ 
The algebra $TL(X)$ has many properties similar to the Hecke algebra $\H(X)$. In particular, it is shown in \cite{gl-cbh} that $TL(X)$ inherits an involution from $\H(X)$ and that it always has a basis, indexed by the fully commutative elements of $W(X)$, with some remarkable properties,  called an IC basis (see \cite{du-icb} and \cite{gl-cbh}). Thus, one has two families of polynomials, indexed by pairs of fully commutative elements of $W(X)$, which are analogous to the Kazhdan--Lusztig and $R$--polynomials of $\H(X)$. Although the Kazhdan--Lusztig and $R$--polynomials have been extensively studied (see, \textrm{e.g.}, \cite{jb-klp}, \cite{g-lck}, \cite{wa-pkl}, \cite{jb-lck},\cite{bf-pkl}, \cite{if-mci},\cite{de-cik}, \cite{if-cik},\cite{bi-lp}, \cite{bcm-sm},\cite{xi-lcc}, \cite{bf-pkic},\cite{bf-klp}, \cite{dc-ckl},\cite{wg-klp}, \cite{bf-ark}, \cite{dc-sop}, \cite{pp-cak}, \cite{bf-ulb}, \cite{bf-cek}, \cite{dc-sac}, \cite{th-nnf}, \cite{la-pkl}, \cite{bf-klr}, \cite{bf-lpk}, \cite{bb-klp}, \cite{BW01}, \cite{BW02}, \cite{Boe88}, \cite{Bre98}, \cite{BreSim}, \cite{Bre02}, \cite{Car94}, \cite{Cas}, \cite{Deo2}, \cite{Deo3}, \cite{Dy}, \cite{Dye97}, \cite{Ki-La}, \cite{L-S}, \cite{Le-Th}, \cite{Mar}, \cite{Pol99}, \cite{SSV}, \cite{Tag94a}, \cite{Zel}) and $TL(X)$ plays an important role in several areas and has also been extensively studied (see, \textrm{e.g.}, \cite{gr-phd}, \cite{jo-har}, \cite{fg-mtl}, \cite{gr-gtld}, \cite{fan-haq}, \cite{gl-ppk}, \cite{we-rttl}, \cite{fan-phd}, \cite{map}, \cite{levy}, \cite{mapp}, \cite{degu}, \cite{libr}, \cite{nere}, \cite{link}, \cite{caut}, \cite{kauff}, \cite{clai}, \cite{nich}, \cite{yong}, \cite{abra}, \cite{loui}, \cite{fend}, \cite{zhan}), these polynomials have not been investigated very much. Our purpose in this work is to begin the study of these polynomials from a combinatorial point of view. More precisely we obtain recursions, non--recursive formulas, symmetry properties, and expressions for the constant terms, of these polynomials. To do this, we need to study some auxiliary polynomials (which have no analogue in $\H(X)$, and which in some sense express the relationship between $\H(X)$ and $TL(X)$) which were first defined in \cite{gl-cbh}. Most of our results hold for every finite irreducible or affine non--branching Coxeter graph different from $\widetilde{F_4}$, although some hold in full generality. Our results show that there is a close relationship between Kazhdan--Lusztig and $R$--polynomials and their analogues in $TL(X)$.

The organization of the thesis is as follows.\\   
In Chapter \ref{ch nap} we fix the notation and recall the standard definitions and results needed in the sequel. In particular, we introduce the notion of Coxeter group, Bruhat order and present some basic examples. Finally we introduce the concept of pattern--avoiding element. In Chapter \ref{ch klt} we briefly reproduce part of theory developed in the celebrated work \cite{kl-rcg} concerning the Hecke algebra associated to a Coxeter group. In Chapter \ref{ch gtla} we define the generalized Temperley--Lieb algebra $TL(X)$ associated to a Coxeter group $W(X)$ and introduce the $D$--polynomials, which are a central tool in the study of the combinatorial properties of the Temperley--Lieb algebra. Then we define two other families of polynomials, namely $\{a_{x,w}\}$ and $\{L_{x,w}\}$, that arise naturally in $TL(X)$. In fact, they are the analogous of the well--known $R$--polynomials and Kazhdan--Lusztig polynomials of $\H(X)$. Both these polynomials can be studied in terms of $D$--polynomials. In Chapter \ref{comb prop} we prove a recurrence relation for the $D$--polynomials, which holds in general, and we derive some closed formulas in type $A$. Next, we prove our main result on $D$--polynomials, which holds for every finite irreducible or affine non--branching Coxeter graph $X$, except $\widetilde{F_4}$. We do the same for $\{a_{x,w}\}$ and $\{L_{x,w}\}$. For each family of polynomials we give recursive formulas. The rest of the chapter deals with lots of consequences of the main result for these polynomials, including symmetry properties, and some combinatorial non--recursive formulas.
   
\vspace{0.3cm}

{\bf\LARGE Acknowledgements}

\vspace{0.3cm}
I am grateful to my Ph.D. advisor, Francesco Brenti. He introduced me to algebraic combinatorics. I wish to thank him for many fruitful and interesting discussions on research topics. Without his constant encouragement and generous guidance this thesis would not exist. 

\setcounter{chapter}{0}
\setcounter{section}{0}
\setcounter{lemma}{0}



\chapter{Coxeter groups} \label{ch nap} 

\section{Notation}
We collect here some notation that is adhered to throughout the book.

\begin{equation}
\begin{array}{ll} 
\P                      & \mbox{the positive integers}                        \nonumber \\
\N                      & \mbox{the non--negative integers}                   \nonumber \\
\Z                      & \mbox{the integers}                                 \nonumber \\
\left[n\right]          & \mbox{the set $\{1, \,2,\cdots ,n\}$}               \nonumber \\
\left[\pm n \right]     & \mbox{the set $\{\pm1,\,\pm2,\cdots ,\pm n\}$}      \nonumber \\ 
|A|                     & \mbox{the cardinality of a set $A$}                 \nonumber \\
R[q]                    & \mbox{ring of polynomials with coefficients in $R$} \nonumber \\ 
\left[q^i\right]P  & \mbox{the coefficient of $q^i$ in $P$}                   \nonumber \\   
\delta_{i,j}            & \mbox{the Kronecker delta: } 
\delta_{i,j}\stackrel{\rm def}{=} 
\begin{cases}
1 & \mbox{if } i=j;  \\
0 & \mbox{if } i \neq j.
\end{cases}   \nonumber

\end{array}
\end{equation}
We write $\stackrel{\rm def}{=}$ if we are defining the left hand side by the right hand side. The symbol $\qed$ denotes the end of a proof or an example. Moreover, we put a $\qed$ at the end of the statement of a result if the result is obvious at that stage of reading.

\newpage

\section{Definition of Coxeter group} \label{sec cg}

Let $S$ be a finite set and let $m$ be a matrix $m:S \times S \rightarrow \P \cup \{\infty\}$ such that 
\begin{enumerate}
\item $m(s_i,s_j)=m(s_j,s_i)$, for every $s_i,s_j \in S$;
\item $m(s_i,s_j) \geq 2$ if $i \neq j$;
\item $m(s_i,s_i)=1$, for every $s_i \in S$.
\end{enumerate}
Then $m$ is called a \emph{Coxeter matrix}. Equivalently, $m$ can be represented by a \emph{Coxeter graph} whose node set is $S$ and whose edges are the unordered pairs $\{s_i, s_j\}$ such that $m(s_i, s_j) \geq  3$. The edges with $m(s_i, s_j) \geq 4$ are labeled by the number $m(s_i,s_j)$. 

\begin{ese}\label{ese h3}
Let $S=\{s_1,s_2,s_3\}$ and denote by $m$ the Coxeter matrix 
\begin{equation} \nonumber
\begin{bmatrix}
1 & 4 & 2  \\
4 & 1 & 3  \\
2 & 3 & 1  \\
\end{bmatrix}
\end{equation}
Then, $m$ corresponds to the Coxeter graph $B_3$ (cf. Appendix \ref{fiacg}).
\begin{figure}[!hbtp]
\[
\xymatrix @C=3pc { *=0{\bullet} \ar@{-}[r]^{4}_<{s_1} & *=0{\bullet} \ar@{-}[r]^{}_<{s_2}_>{s_3}  & *=0{\bullet} }
\]
\caption{Coxeter graph $B_3$.} \label{cgb3}
\end{figure}
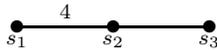
 \qed
\end{ese}

A \emph{Coxeter system} is a pair $(W,S)$ where $W$ is a group with set of generators $S=\{s_1, \cdots , s_n\}$ subject to the relations $$(s_is_j)^{m(s_i,s_j)}=e$$ where $e$ denotes the identity element of $W$. Equvalently, $W$ can be viewed as the quotient $F_S/N$ where $F_S$ is the free group generated by $S$ and $N$ is the normal subgroup generated by $\{(s_is_j)^{m(s_i,s_j)}:\, s_i,s_j \in S,\, m(s_i,s_j)<\infty\}$. The cardinality of $S$ is called the \emph{rank} of $(W,S)$.\\
In the previous example, the group $W$ determined by the Coxeter matrix $m$ has the presentation

\begin{equation}
\left\{ \begin{array}{llrl} 
s_1^2=s_2^2=s_3^2=e,        & \leftrightarrow & m(s_i,s_i)  &=1   \nonumber \\
s_1s_2s_1s_2=s_2s_1s_2s_1,  & \leftrightarrow & m(s_1,s_2)  &=4   \nonumber \\
s_3s_2s_3=s_2s_3s_2,        & \leftrightarrow & m(s_2,s_3)  &=3   \nonumber \\
s_1s_3=s_3s_1,              & \leftrightarrow & m(s_1,s_3)  &=2   \nonumber
\end{array} \right.
\end{equation}

There exists a bijection between Coxeter system $(W,S)$ and the Coxeter graph $X$ (or, equvalently, the Coxeter matrix $m$) associated to $W$ (see, \textrm{e.g.}, \cite[Theorem 1.1.2]{bb-ccg}). Therefore we say that $(W,S)$ is of type $X$. A Coxeter system is said to be \emph{finite} if $|W|< \infty$ and \emph{irreducible} if its Coxeter graph is connected. In Appendix \ref{fiacg} we list all the finite irreducible and affine Coxeter systems. A standard reference for these results is \cite{Hum}.

\section{Length function and Bruhat order} \label{lf bo}

Let $(W,S)$ be a Coxeter system. Each element $v \in W$ can be written as product of generators $s_i \in S$. We denote by $\l(v)$ the minimal integer $k$ such that $v$ can be written as a product of $k$ generators. If $v=s_{i_1}\cdots s_{i_k}$ and $\l(v)=k$ then $k$ is called the \emph{length} of $v$ and $s_{i_1}\cdots s_{i_k}$ is called a reduced expression of $v$.  We list some basic properties on the length function (see, \textrm{e.g.}, \cite[Proposition 1.4.2]{bb-ccg}).
\begin{pro}
Let $u,v \in W$. Then
\begin{itemize}
\item[{\rm (i)}] $\l(e)\stackrel{\rm def}{=}0$;
\item[{\rm (ii)}] $\l(u)=1 \iff u \in S$;
\item[{\rm (iii)}] $\l(u)=\l(u^{-1})$;
\item[{\rm (iv)}] $\l(uv) \equiv l(u)+\l(v) \pmod{2}$;
\item[{\rm (v)}] $\l(us)=\l(u)\pm 1$ for all $s \in S$.  
\end{itemize} 
\end{pro}

Denote by $T$ the set $\{vsv^{-1},\,v \in W,\,s \in S\}$. An element $t \in T$ is called a \emph{reflection}. Write $u \rightarrow v$ if there exists $t \in T$ such that $ut=v$, with $\l(u)<\l(v)$. The \emph{Bruhat graph} of $(W,S)$ is the directed graph whose nodes are the elements of $W$. An ordered pair $(u,v)$ is an edge if and only if $u \rightarrow v$. 
It is possible to define a partial order relation on $W$.
\begin{defi}
Let $u,v \in W$. Then $u \leq v$ in the Bruhat order if there exist $u_i \in W$ such that $$u=u_0\rightarrow u_1 \rightarrow \cdots \rightarrow u_k=v.$$
\end{defi} 

Observe that the Bruhat order relation is the transitive closure of $\rightarrow$. The set $\{u \leq y \leq v :\, y \in W\}$ is customarily  denoted by $[u,v]$.

\begin{ese}
Denote by $A_2$ the Coxeter graph with node set $S=\{s_1,\,s_2\}$ and unlabelled edge $\{s_1,s_2\}$. Then the group $W(A_2)$ has the diagram shown in Figure \ref{fig a2} under the Bruhat order.
\begin{figure}[!hbtp]
\[
\xymatrix{
                                             & *=0{\bullet}\ar@{-}[dr]\ar@{-}[dl]_<{\; s_1s_2s_1}&                                           \\
*=0{\bullet} \ar@{-}[drr]\ar@{-}[d]_<{s_1s_2}&                                                & *=0{\bullet} \ar@{-}[dll]\ar@{-}[d]^<{s_2s_1}\\
*=0{\bullet} \ar@{-}[dr]_<{s_1}_>{e}         &                                                & *=0{\bullet} \ar@{-}[dl]^<{s_2}               \\
                                             & *=0{\bullet}                                   &                                           }  
\]        

\caption{Bruhat order of $A_2$.} \label{fig a2}
\end{figure}
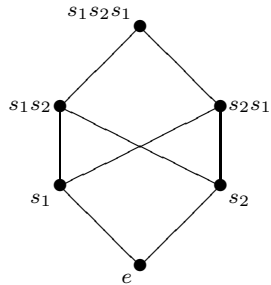
To obtain the Bruhat graph of $A_2$, direct all edges of Figure \ref{fig a2} upward and add the edge $e \longrightarrow s_1s_2s_1$. \qed
\end{ese}

The following fundamental property characterizes the Bruhat order relation (see, \textrm{e.g.}, \cite[Theorem 2.2.2]{bb-ccg}).
\begin{pro}[\textbf{Subword Property}] \label{sub prop}
Let $v=s_1\cdots s_r$ be a reduced expression. Then $u\leq v$ if and only if there exists a reduced expression $s_{i_1}\cdots s_{i_k}$ of $u$ such that $1\leq i_1 \leq \cdots \leq i_k \leq r$. 
\end{pro}
Another property which characterizes the Bruhat order is the following, which will be very useful in the sequel (see, \textrm{e.g.}, \cite[Proposition 2.2.7]{bb-ccg}).  
\begin{lemma}[\textbf{Lifting Property}] \label{lif lemma}
Let $u,v \in W$ be such that $u<v$ and suppose $s \in S$ such that $vs<v$ and $us>u$. Then $u \leq vs$ and $us \leq v$. 
\end{lemma} 
\begin{figure}[!hbtp]
\[
\xymatrix@dr@C=5pc{
*=0{\bullet} \ar@{--}[r] \ar@{-}[d]_<{v}   & *=0{\bullet} \ar@{-}[d]^<{us} \\
*=0{\bullet} \ar@{--}[r]_<{vs}             & *=0{\bullet} \ar@{-}[ul]^<{u} }
\]
\caption{The lifting property.}
\end{figure}
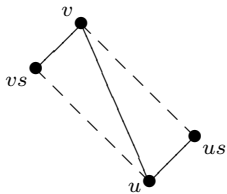
The Bruhat order relation gives $W$ the structure of graded \emph{partially ordered set}, with the length fuction $\l$ as \emph{rank function} (see, \textrm{e.g.}, \cite[\S 2.2]{bb-ccg}).\\
A proof of the next result can be found in \cite[Corollary 2.2.5 and Proposition 2.3.4]{bb-ccg}.
\begin{pro} \label{pro automorph}
Let $W$ be a finite Coxeter group with maximum element $w_0$. Then the following are equivalent: 
\begin{itemize}
\item[{\rm (i)}] $u\leq v$;
\item[{\rm (ii)}] $u^{-1}\leq v^{-1}$;
\item[{\rm (iii)}] $w_0uw_0 \leq w_0vw_0$.
\end{itemize} 
\end{pro}
In other words, Proposition \ref{pro automorph} states that the mappings $\phi : u \mapsto u^{-1}$ and $\psi: u \mapsto w_0uw_0$ are automorphism of the Bruhat order.

\section{The word property} \label{word prop}

Let $(W,S)$ be a Coxeter system with associated Coxeter matrix $[m(s_i,s_j)]$. Define $$\alpha_{s_i,s_j}\stackrel{\rm def}{=}\underbrace{s_is_js_is_j\cdots}_{m(s_i,s_j) \mbox{ \small{factors}}}.$$
Let $u,v \in W$. We say that $u$ and $v$ are linked by a \emph{braid--move} if there exist $s_i,s_j \in S$ such that $u$ can be obtained from $v$ by substituting $\alpha_{s_i,s_j}$ for $\alpha_{s_j,s_i}$. Moreover we say that $u$ and $v$ are linked by a \emph{nil--move} if $u$ can be obtained from $v$ by deleting a factor of the form $s_is_i$. In both cases we write $u \sim v$.
The problem of deciding whether two expressions in the alphabet $S$ represents the same element in $W$ is not trivial at all. A complete answer to this question is given by the following general result (see, \textrm{e.g.}, \cite[Theorem 3.3.1]{bb-ccg}).  

\begin{teo}[\textbf{Word Property}]
Let $v \in W$. Then any reduced expression for $v$ can be obtained from any other by applying a finite sequence of braid--moves and nil--moves.
\end{teo}

\begin{cor}
Let $s_1 \cdots s_r$ be an expression of $v \in W$. Then any reduced expression for $v$ can be obtained from $s_1 \cdots s_r$ by applying a finite sequence of braid--moves and nil--moves. \qed
\end{cor}

\begin{ese} \label{ese dpol}
Let us consider the Coxeter graph drawn in Example \ref{ese h3}. Then, we have nil--relations $$s_1^2=s_2^2=s_3^2=e$$ and braid--relations $$s_1s_2s_1s_2=s_2s_1s_2s_1, \, s_3s_2s_3=s_2s_3s_2, \, s_1s_3=s_3s_1.$$
Hence, a reduced expression for $v=s_2s_3s_2s_1s_1s_3s_2s_1$ can be obtained by a sequence of three nil-moves and one braid-move, as shown below:
$$s_2s_3s_2\underline{s_1s_1}s_3s_2s_1 \sim \underline{s_2s_3s_2}s_3s_2s_1 \sim s_3s_2\underline{s_3s_3}s_2s_1 \sim s_3\underline{s_2s_2}s_1 \sim s_3s_1.$$ \qed
\end{ese}

\section{The symmetric group} \label{sec sym}
For a proof of the results given in this section we refer to \cite[\S 1.5]{bb-ccg}.\\
The symmetric group $S_n$ is the group of bijections of $[n]$ to itself. An element $\tau \in S_n$ is called a permutation. We consider different standard notations for a permutation $\tau$. The first one is the \emph{matrix notation}
\begin{equation} \nonumber
\begin{pmatrix}
1      & 2      & 3      & \cdots & n      \\
\tau_1 & \tau_2 & \tau_3 & \cdots & \tau_n \\
\end{pmatrix}
\end{equation}
and it means that $\tau: i \mapsto \tau_i$.
By taking only the second row of the above matrix we get the \emph{one--line notation} and we write $\tau=\tau_1\cdots \tau_n$.
Another notation is the \emph{disjoint--cylces notation}. For instance, if $\tau=432516$ then we write $\tau=(1,4,5)(2,3)$ and omit the $1$--cycles of $\tau$. The product $\sigma \tau$ is defined as the function composition $\sigma \circ \tau$. For example, $(1,3,2)(3,2)=(1,3)$.\\ 
Consider as set of generators for $S_n$ the set of the adjacent transpositions $S=\{(i,i+1),\,i\in [n-1]\}$, and set $s_i\=(i,i+1)$, for every $i \in [n-1]$. 
Denote by $A_{n-1}$ the Coxeter graph having nodes $\{s_1, \cdots , s_{n-1}\}$ and unlabelled edges $\{s_i,\,s_{i+1}\}$, for all $i \in [n-1]$ (cf. Appendix \ref{fiacg}). 

\begin{figure}[!hbtp]
\[
\xymatrix @C=3pc { *=0{\bullet} \ar@{-}[r]^{}_<{s_1} & *=0{\bullet} \ar@{-}[r]^{}_<{s_2}& *=0{\bullet} \ar@{--}[r]^{}_<{s_3} & *=0{\bullet} \ar@{-}[r]^{}_<{s_{n-2}}_>{s_{n-1}}  & *=0{\bullet} }
\]
\caption{Coxeter graph $A_{n-1}$.} 
\end{figure}
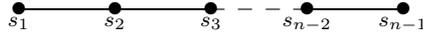

\begin{pro}
$(S_n,S)$ is a Coxeter system of type $A_{n-1}$. 
\end{pro}
In fact, it is straightforward to verify that the generators $s_1,\,s_2, \cdots, s_{n-1}$ satisfy the Coxeter relations 
\begin{equation}
\left\{ \begin{array}{ll} 
s_i^2=e,              &                            \nonumber \\
s_is_js_i=s_js_is_j   & \mbox{if } | i-j | =1,     \nonumber \\
s_is_j=s_js_i         & \mbox{if } | i-j | \geq 2. \nonumber
\end{array} \right.
\end{equation}
Let $\tau \in S_n$ and define $$inv(\tau)\stackrel{\rm def}{=} | \{(i,j) \in [n]^2: \, i<j,\, \tau_i > \tau_j\} |,$$ called the number of inversions of $\tau$. The length and the number of inversions of a permutation in the symmetric group are closely related, as explained by the following result. 
\begin{pro} \label{linv}
Let $\tau \in S_n$. Then $\l(\tau)=inv(\tau)$. 
\end{pro}
Hence, the longest element, with respect to the length function, of $S_n$ is $w_0=n n-1 \cdots 3 2 1$.

\section{The hyperoctahedral group}
For a proof of the results given in this section we refer to \cite[\S 8.1]{bb-ccg}.\\
The \emph{hyperoctahedral group} $S_n^B$ is the group of bijections $\sigma$ of $[\pm n]$ to itself such that $\sigma(-i)=- \sigma(i)$, for all $i \in [n]$. An element in $S_n^B$ is called a \emph{signed permutation}. We consider different standard notations for an element $\sigma \in S_n^B$. The first one is the \emph{window notation}: $\sigma=[\sigma_1,\sigma_2,\cdots ,\sigma_n]$ means that $\sigma(i)=\sigma_i$ for every $i \in [n]$. A more compact notation is often used: we write $\sigma=|\sigma_1| \, |\sigma_2| \, \cdots |\sigma_n|$ and put bar over an element with a negative sign. For instance, $3\overline{5}41\overline{2}$ corresponds to the signed permutation $[3,-5,4,1,-2]$.
Another notation is the \emph{disjoint--cylces notation}. For instance, if $\tau=[4,-3,-2,5,1,6]$ then we write $\tau=(1,4,5)(2,-3)$, and we omit the $1$--cycles of $\tau$. The product $\sigma \tau$ is defined as the function composition $\sigma \circ \tau$. Consider as set of generators for $S_n^B$ the set $S=\{(1,-1)\} \cup \{(i,i+1)(-i,-i-1),\, i \in [n-1]\}$, and set $s_0\=(1,-1)$, $s_i\=(i,i+1)(-i,-i-1)$, for every $i \in [n-1]$. Denote by $B_n$ the Coxeter graph having nodes $\{s_0, \cdots , s_{n-1}\}$, unlabelled edges $\{s_i,\,s_{i+1}\}$ for all $i \in [n-2]$ and the edge $\{s_0,s_1\}$ labelled by $4$ (cf. Appendix \ref{fiacg}).
\begin{figure}[!hbtp]
\[
\xymatrix @C=3pc { *=0{\bullet} \ar@{-}[r]^{4}_<{s_0} & *=0{\bullet} \ar@{-}[r]^{}_<{s_1}& *=0{\bullet} \ar@{--}[r]^{}_<{s_2} & *=0{\bullet} \ar@{-}[r]^{}_<{s_{n-2}}_>{s_{n-1}}  & *=0{\bullet} }
\]
\caption{Coxeter graph $B_n$, $n \geq 2$.}
\end{figure}
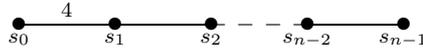

\begin{pro}
$(S_n^B,S)$ is a Coxeter system of type $B_n$.
\end{pro}
In fact, it can be easily checked that the generators $s_1,\,s_2, \cdots, s_{n-1}$ satisfy the Coxeter relations 
\begin{equation}
\left\{ \begin{array}{ll} 
s_i^2=e,                  &                                                  \nonumber \\
s_0s_1s_0s_1=s_1s_0s_1s_0,&                                                  \nonumber \\
s_is_js_i=s_js_is_j       & \mbox{if } | i-j | =1 \mbox{ and }i,j \in [n-1], \nonumber \\
s_is_j=s_js_i             & \mbox{if } | i-j | \geq 2.                       \nonumber
\end{array} \right.
\end{equation}
Let $\sigma \in S_n^B$ and define $$inv(\sigma)\stackrel{\rm def}{=} | \{(i,j) \in [\pm n]^2: \, i<j,\, \sigma_i > \sigma_j\} |,$$ called the number of inversions of $\sigma$. The length function $\l_B$ of $S_n^B$ may be expressed in many ways. A simple combinatorial description of $\l_B$ is given in \cite[Proposition 3.1]{bre-qep}.
\begin{pro} 
Let $\sigma \in S_n^B$. Then 
$$\l_B(\sigma)=inv(\sigma)-\sum_{j \in [n]:\,\sigma(j)<0} \sigma(j).$$
\end{pro}
Hence, the longest element of $S_n^B$ is $w_0=\overline{1\,2 \cdots \,n}$.

\section{Restricted permutations} \label{restr perm}

Two numerical sequences $\alpha=(\alpha_1,\cdots , \alpha_m)$ and $\beta=(\beta_1,\cdots , \beta_m)$ are said to be order--isomorphic if $\alpha_i < \alpha_j \iff \beta_i < \beta_j$, for all $i, j \in [m]$. 
Let $\pi \in S_n$ and $\tau \in S_k$. An occurrence of $\tau$ in $\pi$ is a subsequence $1 \leq i_1 \leq \cdots \leq i_k \leq n$ such that $(\pi_{i_1}, \cdots , \pi_{i_k})$ is order--isomorphic to $(\tau_1,\cdots , \tau_k)$. In this context $\tau$ is usually called a pattern and we say that $\pi$ is $\tau$--avoiding if there is no occurrence of $\tau$ in $\pi$. Denote by $S_n(\tau)$ the set $\{\sigma \in S_n:\, \sigma \mbox{ is $\tau$--avoiding}\}$. For an arbitrary finite collection of patterns $T$, we say that $\pi$ avoids $T$ if $\pi$ avoids any $\tau \in T$ and denote by $S_n(T)$ the corresponding subset of $S_n$.\\
Throughout this work the $321$--avoiding permutations of $S_n$ will be of particular interest.   
\begin{defi}
An element $w \in W$ is fully commutative if any reduced expression for $w$ can be obtained from any other by applying braid--moves that involve only commuting generators. We let $$W_c \stackrel{\rm def}{=} \{w \in W:\, w \, \mbox{ is a fully commutative element}\}.$$  
\end{defi}
We will denote by $[x,w]_c$ the set $\{y \in [x,w]:\, y \in W_c\}$.\\
As we have seen in Section \ref{sec sym}, if  $X=A_{n-1}$ then $W=W(X) \simeq S_n$. In this case $W_c(A_{n-1})$ may be described as the set of elements of $W(A_{n-1})$ all of whose reduced expressions avoid substrings of the form $s_is_{i \pm 1}s_i$, for all $s_i \in S$ (see \cite[Proposition 1.1]{ste-fce}). The notion of fully commutative element in type $A$ can be reformulated in terms of pattern--avoidance. A proof of the next result can be found in \cite[Theorem 2.1]{bjs-cps}).
\begin{pro} \label{321wc}
The sets $S_n(321)$ and $W_c(A_{n-1})$ coincide. Moreover $\vert W_c(A_{n-1}) \vert = C_n$, where $C_n\stackrel{\rm def}{=}\frac{1}{n+1}\binom{2n}{n}$ denotes the $n$--th Catalan number.
\end{pro}

A similar result holds for $W_c(B_n)$. Since the general definition of signed pattern--avoidance in type $B$ is quite complicated, we prefer to explain it by examples. For instance, we say that an element $w=w_1w_2 \cdots w_n \in S_n^B$ avoids the pattern $2 \overline{31}$ if there is no triple $i < j < k$ such that $-w_j > w_i > -w_k > 0$. For a proof of the following result we refer to \cite[Theorem 5.1 and Proposition 5.9]{ste-sca}.
\begin{pro}
The sets $S_n^B(\{ \overline{12};\, 321;\, \overline{3}21;\, \overline{23}1;\, 2\overline{3}1 \})$ and $W_c(B_n)$ coincide. Moreover $\vert W_c(B_n) \vert = (n+2)C_n-1$, where $C_n\stackrel{\rm def}{=}\frac{1}{n+1}\binom{2n}{n}$ denotes the $n$--th Catalan number.  
\end{pro}

Let us now look at the maps $\phi : u \mapsto u^{-1}$ and $\psi: u \mapsto w_0uw_0$ defined in Proposition \ref{pro automorph}.
\begin{lemma} \label{maps}
Let $W(X)$ be a finite Coxeter group. Then $$u \in W_c(X) \Longleftrightarrow \phi(u)\in W_c(X) \Longleftrightarrow \psi(u) \in W_c(X),$$ for all $u \in W(X)$.
\end{lemma}
\begin{proof}
The maps $u \mapsto w_0uw_0$ and $u \mapsto u^{-1}$ are Bruhat order automorphisms (see \cite[Proposition 2.3.4 and Corollary 2.3.6]{bb-ccg}). Moreover, both these maps send Coxeter generators to Coxeter generators, since $\ell(u)=\ell(u^{-1})=\ell(w_0uw_0)$ (see \cite[Corollary 2.3.3]{bb-ccg}).
\end{proof}




\chapter{Hecke algebras} \label{ch klt}

\section{Definition of Hecke algebra}
In this section we recall some basic facts about Hecke algebras $\H(X)$, $X$ being any Coxeter graph. A standard reference for this topic is \cite{Hum}. Let $W(X)$ be the Coxeter group having $X$ as Coxeter graph and $S(X)$ as set of generators. Throughout this work we will denote by $\varepsilon_w$ the constant $(-1)^{\l(w)}$, for every $w \in W(X)$. Let $\A$ be the ring of Laurent polynomials $\Z[q^{\frac{1}{2}},q^{-\frac{1}{2}}]$. The Hecke algebra $\H(X)$ associated to $W(X)$ is an $\A$--algebra with linear basis $\{T_w:\,w\in W(X)\}$. For all $w \in W(X)$ and $s \in S(X)$ the multiplication law is determined by
\begin{equation}\label{prod}
T_{w}T_{s} = 
\left\{ \begin{array}{ll} 
T_{ws}             & \mbox{if $\l(ws)>\l(w)$,} \\
q T_{ws}+(q-1)T_w  & \mbox{if $\l(ws)<\l(w)$,}
\end{array} \right.
\end{equation}
where $\l$ denotes the usual length function of $W(X)$ (see Section \ref{lf bo}). We refer to $\{T_w:\,w\in W(X)\}$ as the $T$--basis for $\H(X)$.\\ 

One easily checks that $T_s^2=(q-1)T_s+qT_e$, being $T_e$ the identity element, and so 
\begin{equation} \label{invTs}
T_s^{-1}=q^{-1}(T_s-(q-1)T_e). 
\end{equation}
It follows that all the elements $T_w$ are invertible, since, if $w=s_1 \cdots s_r$ and $\l(w)=r$, then $T_w=T_{s_1} \cdots T_{s_r}$.

Define a map $\iota:\H\rightarrow \H$ such that $\iota(T_w)=(T_{w^{-1}})^{-1}$, $\iota(q)=q^{-1}$ and extend by linearity. 
\begin{pro} \label{mapiota}
The map $\iota$ is a ring homomorphism of order $2$ on $\H(X)$.
\end{pro}
\begin{proof}
It is straightforward to prove that $\iota^2(T_s)=T_s$, for all $s \in S(X)$. We will show that $\iota$ is a ring homomorphism. First, we show that 
\begin{equation} \label{invols}
\iota(T_sT_w)=\iota(T_s)\iota(T_w), \mbox{ for all } w \in W(X) \mbox{ and } s \in S(X). 
\end{equation}
There are two cases to study: if $\l(sw)>\l(w)$ then 
$$\iota(T_sT_w)= \iota(T_{sw})=\left(T_{w^{-1}s} \right)^{-1}=\left(T_{w^{-1}}T_s \right)^{-1}=(T_s)^{-1} \left(T_{w^{-1}}\right)^{-1}=\iota(T_s)\iota(T_w).$$
If $\l(sw)<\l(w)$ then let $v\= sw$, so that $w=sv$. Therefore $T_w=T_sT_v$. By (\ref{prod}), we obtain
$$\iota(T_sT_w)=\iota(qT_{sw}+(q-1)T_w)=q^{-1}\left(T_{v^{-1}}\right)^{-1}+(q^{-1}-1)\left(T_{w^{-1}} \right)^{-1}.$$
On the other hand
\begin{eqnarray} 
\iota(T_s)\iota(T_w) &=& (T_s)^{-1}\left(T_{w^{-1}} \right)^{-1}=q^{-1}(T_s+1-q)\left(T_{w^{-1}}\right)^{-1} \nonumber \\
                     &=& q^{-1}T_s \left(T_{w^{-1}} \right)^{-1}+(q^{-1}-1)\left(T_{w^{-1}} \right)^{-1}. \nonumber
\end{eqnarray}
But $\left(T_{v^{-1}} \right)^{-1}=T_s\left(T_{w^{-1}} \right)^{-1} \Longleftrightarrow T_{w^{-1}}=(T_{v^{-1}})T_s$, which is obviousely true by definition. Hence $\iota(T_sT_w)=\iota(T_s)\iota(T_w)$. Now we ready to prove that $\iota(T_xT_w)=\iota(T_x)\iota(T_w)$ for all $x,w \in W(X)$. Proceed by induction on $\l(x)$. If $\l(x)=1$ then $x=s' \in S(X)$. If $\l(x)>1$ then there exists $s \in S(X)$ such that $\l(xs)<\l(x)$ and, by induction, $$\iota(T_xT_w)=\iota(T_{xs}T_sT_w)=\iota(T_{xs})\iota(T_sT_w).$$ Finally, by applying (\ref{invols}) and the induction hypotesis we get
$$\iota(T_{xs})\iota(T_sT_w)=\iota(T_{xs})\iota(T_s)\iota(T_w)=\iota(T_x)\iota(T_w),$$ as desired. 
\end{proof}

\section{$R$--polynomials}
In Proposition \ref{mapiota} we stated the existence of the involution $\iota$ in $\H(X)$. To express the image of $T_w$ under $\iota$ as a linear combination of elements in the $T$--basis, one defines the so--called \emph{$R$--polynomials}.  
\begin{teo}\label{r-pol}
Let $w \in W(X)$. Then there exists a unique family of polynomials $\{ R_{x,w} \}_{x \in W(X)} \subseteq \Z[q]$ such that 
\begin{equation} \label{rpolinv}
(T_{w^{-1}})^{-1}=\varepsilon_wq^{-\l(w)}\sum_{x\leq w}\varepsilon_x R_{x,w}T_x, 
\end{equation}
where $R_{w,w}=1$ and $R_{x,w}=0$ if $x \not \leq w$. 
\end{teo}
\begin{proof}
First, we prove the existence of the $R$--polynomials. 
If $w=s \in S(X)$, then the statement follows by (\ref{invTs}) and so we set $R_{e,s}\=q-1$. Now proceed by induction on $\l(w)$. Set $w=sv$, so that $\varepsilon_w=-\varepsilon_v$ and $q^{-\l(w)}=q^{-\l(v)-1}$.
\begin{eqnarray}
(T_{w^{-1}})^{-1} & = & (T_s)^{-1} (T_{v^{-1}})^{-1}  \nonumber \\
                  & = & \frac{1}{q} (T_s-(q-1)T_e)(\varepsilon_v q^{-\l(v)} \sum_{y \leq v} \varepsilon_y R_{y,v} T_y)  \nonumber \\
                  & = & -\varepsilon_v q^{-\l(v)-1} \left((q-1)\sum_{y \leq v} \varepsilon_y R_{y,v}T_y-\sum_{y \leq v} \varepsilon_y R_{y,v}T_sT_y  \right)  \nonumber \\
                  & = & \varepsilon_w q^{-\l(w)} \left((q-1) \sum_{y \leq v} \varepsilon_y R_{y,v} T_y - \sum_{\substack{y \leq v \\ sy>y}} \varepsilon_y R_{y,v}T_{sy} \right) +   \nonumber \\
                  &   & -\: \varepsilon_w q^{-\l(w)} \left(\sum_{\substack{y \leq v \\ sy<y}} \left(q \varepsilon_y R_{y,v} T_{sy}+(q-1)\varepsilon_y R_{y,v} T_y  \right) \right) \nonumber \\
                  & = & \varepsilon_w q^{-\l(w)} \left((q-1) \sum_{\substack{y \leq v \\ sy>y}} \varepsilon_y R_{y,v} T_y -q \sum_{\substack{y \leq v \\ sy<y}} \varepsilon_y R_{y,v}T_{sy}\right)+  \nonumber \\
                  &   & -\varepsilon_w q^{-\l(w)}\: \left(\sum_{\substack{y \leq v \\ sy>y}} \varepsilon_y R_{y,v} T_{sy}\right) \nonumber \\
                  & = & \varepsilon_w q^{-\l(w)} \left(  \sum_{\substack{x \leq v \\ x<sx}} \varepsilon_x ((q-1)R_{x,v}+qR_{sx,v})T_x + \sum_{\substack{x \leq v \\ sx<x}} \varepsilon_x R_{sx,v}T_{x} \right), \nonumber
\end{eqnarray}   
where we set $sx=y$. If $sx<x$, then we can define $R_{x,w}\=R_{sx,sw}$. Otherwise, we set $R_{x,w}\=qR_{sx,sw}+(q-1)R_{x,sw}$. In both cases, $R_{x,w}$ satisfies the condition of the theorem.\\
The uniqueness of the $R$--polynomials is trivial. 
\end{proof}
We remark that in the proof of Theorem \ref{r-pol} is given an inductive method to compute the $R$--polynomials.
\begin{cor} \label{cor recr}
Let $x,w \in W(X)$ such that $x \leq w$ and $s \in S(X)$ such that $\l(sw)<\l(w)$. Then 
\begin{equation} \label{rec foR}
R_{x,w} = \left\{ 
\begin{array}{ll} 
R_{sx,sw},                   & \mbox{if $\l(sx)<\l(x)$,} \\ 
qR_{sx,sw}+(q-1)R_{x,sw}, & \mbox{otherwise.} 
\end{array} \right.  
\end{equation}  
\qed
\end{cor}
Obviously, there exists a right version of Corollary \ref{cor recr}, with $s$ occurring on the right instead of the left.\\  
Next, we state some basic facts about the $R$--polynomials.
\begin{cor} \label{val0}
Let $x, w \in W(X)$, $x \leq w$. Then, $R_{x,w}$ is a monic polynomial of degree $\l(w)-\l(x)$ such that $R_{x,w}(0)=\varepsilon_x \varepsilon_w$ and $R_{x,w}(1)=\delta_{x,w}$.
\end{cor}
\begin{proof}
The statement follows from (\ref{rec foR}), by induction on $\l(w)$. 
\end{proof}
\begin{pro} \label{rsymm}
Let $x,w \in W(X)$. Then $$R_{x,w}(q^{-1})=\varepsilon_x\varepsilon_wq^{\l(x)-\l(w)}R_{x,w}(q).$$
\end{pro}
\begin{proof}
We proceed by induction on $\l(w)$. Let $\l(w)>0$ and let $s \in S(X)$ be such that $sw<w$. By (\ref{rec foR}), $sx<x$ implies $R_{x,w}(q)=R_{sx,sw}(q)$ and, by induction, we have $$R_{sx,sw}(q^{-1})=\varepsilon_{sx}\varepsilon_{sw}q^{\l(sx)-\l(sw)}R_{sx,sw}(q)=\varepsilon_x\varepsilon_wq^{\l(x)-\l(w)}R_{x,w}(q).$$  
The case $sx>x$ implies $R_{x,w}(q)=qR_{sx,sw}(q)+(q-1)R_{x,sw}(q)$. We show that $$R_{x,w}(q^{-1})=\varepsilon_x\varepsilon_w q^{\l(x)-\l(w)}(qR_{sx,sw}(q)+(q-1) R_{x,sw}(q)).$$
By induction, we achieve 
\begin{eqnarray}
R_{x,w}(q^{-1}) & = & q^{-1}R_{sx,sw}(q^{-1})+(q^{-1}-1)R_{x,sw}(q^{-1}) \nonumber \\
                & = & q^{-1} \varepsilon_{sx} \varepsilon_{sw} q^{\l(sx)-\l(sw)} R_{sx,sw}(q) + (q^{-1}-1)\varepsilon_{x} \varepsilon_{sw} q^{\l(x)-\l(sw)} R_{x,sw}(q)   \nonumber \\
                & = & q^{-1} \varepsilon_{x} \varepsilon_{w} q^{\l(x)-\l(w)+2} R_{sx,sw}(q) - (1-q)\varepsilon_{x} \varepsilon_{w} q^{\l(x)-\l(w)} R_{x,sw}(q), \nonumber \\
                & = & \varepsilon_x\varepsilon_w q^{\l(x)-\l(w)}(qR_{sx,sw}(q)+(q-1) R_{x,sw}(q)), \nonumber
\end{eqnarray}
as desired.
\end{proof}

\begin{pro} \label{issi}
Let $\sigma: \H \rightarrow \H$ be such that $\sigma(T_w)=\varepsilon_w q^{-\l(w)}T_w$ and $\sigma(q)=q^{-1}$. Then $\sigma$ is an involution and $\iota\sigma=\sigma\iota$.
\end{pro}
\begin{proof}
The proof that $\sigma$ is a ring homomorphism of order $2$ is similar to the one given in Proposition \ref{invols} for $\iota$.\\
Let us now prove that $\iota\sigma = \sigma\iota$. By definition, $\sigma$ coincides with $\iota$ on $\Z[q^{\frac{1}{2}},q^{-\frac{1}{2}}]$. Moreover $$\iota \sigma(T_w)=\iota(\varepsilon_w q^{-\l(w)}T_w)=\varepsilon_w q^{\l(w)}(T_{w^{-1}})^{-1}=\sum_{x \leq w} \varepsilon_x R_{x,w}(q)T_x.$$
On the other hand, $$\sigma \iota(T_w)= \sigma(\varepsilon_wq^{-\l(w)}\sum_{x \leq w}\varepsilon_x R_{x,w}(q)T_x)=\varepsilon_w q^{\l(w)}\sum_{x \leq w} q^{-\l(x)}R_{x,w}(q^{-1})T_x,$$ and the statement follows by applying Proposition \ref{rsymm}.  
\end{proof}

The next property is needed in Proposition \ref{apollo}.
\begin{pro} \label{rpol prop}
Let $w \in W(X)$. Then $\sum_{x \leq w} R_{x,w}=q^{\l(w)}$.
\end{pro}
\begin{proof}
We proceed by induction on $\l(w)$. The case $\l(w) \leq 1 $ is trivial. Hence, suppose $\l(w)>1$ and $s \in S(X)$ such that $sw<w$. Then, by Corollary \ref{cor recr} and Lemma \ref{lif lemma} we obtain 
\begin{eqnarray}
\sum_{x \leq w} R_{x,w} &=& \sum_{\substack{x \leq w \\ sx<x}} R_{x,w}+\sum_{\substack{x \leq w \\ sx>x}} R_{x,w}  \nonumber \\
                           &=& \sum_{\substack{x \leq w \\ sx<x}} R_{sx,sw}+ \sum_{\substack{x \leq w \\ sx>x}} qR_{sx,sw}+(q-1)R_{x,sw} \nonumber \\
                           &=& \sum_{\substack{x \leq w \\ sx<x}} R_{sx,sw}+ \sum_{\substack{y \leq w \\ y>sy}} qR_{y,sw}+(q-1)R_{sy,sw} \nonumber \\
                           &=& q\sum_{\substack{x \leq w \\ sx<x}} R_{sx,sw}+ q\sum_{\substack{y \leq w \\ y>sy}} R_{y,sw} \nonumber \\
                           &=& q\sum_{\substack{y \leq w \\ sy>y}} R_{y,sw}+q\sum_{\substack{y \leq w \\ y>sy}} R_{y,sw}   \nonumber\\ 
                           &=& q\sum_{y \leq w}R_{y,sw}   \nonumber \\ 
                           &=& q \cdot q^{\l(sw)}=q^{\l(w)}. \nonumber
\end{eqnarray}
\end{proof}

\section{Kazhdan--Lusztig polynomials}
In this section we prove the existence and uniqueness of an $\iota$--invariant basis for $\H(X)$. 
\begin{teo} \label{kl-pol}
There exists a unique basis $\{C_w:\, w \in W(X)\}$ for $\H(X)$ such that the following properties hold:
\begin{itemize}
\item[{\rm (i)}] $\iota(C_w)=C_w$,
\item[{\rm (ii)}] $C_w=\varepsilon_w q^{\frac{\l(w)}{2}}\sum_{x\leq w}\varepsilon_x q^{-\l(x)}P_{x,w}(q^{-1})T_x$,
\end{itemize}
where $\{P_{x,w}(q)\}\subseteq \Z[q]$, $P_{w,w}(q)=1$ and ${\rm deg} (P_{x,w}(q)) \leq \frac{1}{2}(\l(w)-\l(x)-1)$ if $x<w$. 
\end{teo}
\begin{proof}
First, we prove the existence of $C_w$ by induction on $\l(w)$. Assume $w \neq e$ and suppose that $C_x$ has already been constructed, for every $x < w$. Obviously, $C_e=T_e$. Let $s \in S(X)$ be such that $sw<w$ and set $w=sv$. We define 
\begin{equation} \label{prodcscv}
C_w\=C_sC_v- \sum_{sz<z}\mu(z,v)C_z,
\end{equation}
where $\mu(z,v)\=[q^{\frac{\l(v)-\l(z)-1}{2}}]P_{z,v}(q)$. Observe that (\ref{prodcscv}) implies $\iota(C_w)=C_w$. In fact, $$C_s=q^{-\frac{1}{2}}T_s-q^{\frac{1}{2}}T_e.$$ On the other hand,  
\begin{eqnarray}
\iota(C_s) & = & \iota(q^{-\frac{1}{2}}T_s-q^{\frac{1}{2}}T_e)  \nonumber \\ 
           & = & q^{\frac{1}{2}} (T_s)^{-1}-q^{-\frac{1}{2}}T_e  \nonumber \\
           & = & q^{\frac{1}{2}} (q^{-1}(T_s-(q-1)T_e))-q^{-\frac{1}{2}}T_e \nonumber \\
           & = & q^{-\frac{1}{2}}T_s-q^{\frac{1}{2}}T_e \nonumber 
\end{eqnarray}
and the invariance follows by induction on $\l(w)$. Next, extract and equate the coefficient of $T_x$ on both sides of (\ref{prodcscv}).
A careful analysis of the left hand side of (\ref{prodcscv}) shows that the coefficient of $T_x$ in $C_sC_v$ is 
\begin{equation} \nonumber
[T_x](C_sC_v) = \left\{ 
\begin{array}{ll} 
\varepsilon_x \varepsilon_w q^{-\l(x)}q^{\frac{\l(w)}{2}} (P_{sx,w}(q^{-1})+q^{-1}P_{x,w}(q^{-1})),  & \mbox{if $\l(sx)<\l(x)$,} \\ 
\varepsilon_x \varepsilon_w q^{-\l(x)}q^{\frac{\l(w)}{2}} (q^{-1}P_{sx,w}(q^{-1})+P_{x,w}(q^{-1})),  & \mbox{otherwise,} 
\end{array} \right.  
\end{equation} 
and that $$[T_x]\left(\sum_{sz<z}\mu(z,v)C_z\right)=\varepsilon_x \varepsilon_w q^{-\l(x)}q^{\frac{\l(w)}{2}} \sum_{sz<z}\mu(z,v) q^{\frac{-\l(z)}{2}}q^{-\frac{\l(w)}{2}}P_{x,z}(q^{-1}).$$
Combining these information we can express a basis element in the form 
\begin{equation} \label{Cwteo}
C_w=\varepsilon_w q^{\frac{\l(w)}{2}}\sum_{x\leq w}\varepsilon_x q^{-\l(x)}P_{x,w}(q^{-1})T_x,
\end{equation}
where we set
\begin{equation} \label{pxwc}
P_{x,w}(q) \= q^{1-c}P_{sx,sw}(q) + q^cP_{x,sw}(q)- \sum_{\{z:\,sz<z\}}q^{\frac{\l(w)-\l(z)}{2}}\mu(z,sw)P_{x,z}(q),
\end{equation}    
with $c=1$ if $sx<x$ and $c=0$ otherwise.
This is routine to check that the polynomials defined by (\ref{pxwc}) satisfy the degree bound stated in the theorem.\\
Now we deal with the uniqueness part, assuming the existence and the invariance of $C_w$. The uniqueness of the element in (\ref{Cwteo}) is equivalent to saying that there is a unique choice for polynomials $P_{x,w}$. We proceed by induction on $\l(w)-\l(x)$, where $w \in W(X)$ is fixed. Assume that the $P_{y,w}$ can be chosen in a unique way, for all $x<y \leq w$. We will show that $P_{x,w}$ is uniquely determined. Combining (\ref{Cwteo}) with (\ref{rpolinv}), we achieve
\begin{eqnarray}
\iota(C_w) & = & \varepsilon_w q^{\frac{-\l(w)}{2}}\sum_{y\leq w}\varepsilon_y q^{\l(y)}P_{y,w}(q)(T_{y^{-1}})^{-1} \nonumber \\
           & = & \varepsilon_w q^{\frac{-\l(w)}{2}}\sum_{y\leq w}\varepsilon_y q^{\l(y)}P_{y,w}(q)\left(\varepsilon_y q^{-\l(y)}\sum_{x \leq y} \varepsilon_x R_{x,y}T_x \right)  \nonumber \\
           & = & \varepsilon_w q^{\frac{-\l(w)}{2}}\sum_{x \leq w} \left(\sum_{x \leq y \leq w}\varepsilon_x R_{x,y}P_{y,w} \right) T_x.   \label{eqco}
\end{eqnarray}
On the other hand, $\iota(C_w)=C_w$. Therefore, equating the coefficient of $T_x$ in (\ref{eqco}) with the one in (\ref{Cwteo}) we obtain
\begin{equation} \label{mbs}
\varepsilon_x \varepsilon_w q^{-\l(x)}q^{\frac{\l(w)}{2}}P_{x,w}(q^{-1})= \varepsilon_w q^{-\frac{\l(w)}{2}} \sum_{x \leq y \leq w}\varepsilon_x R_{x,y}P_{y,w}.
\end{equation}
Multiplying both sides of (\ref{mbs}) by $q^{\frac{\l(x)}{2}}$ and moving the term for $y=x$ to the left, we get 
\begin{equation} \label{recdaC}
q^{\frac{-\l(x)}{2}}q^{\frac{\l(w)}{2}}P_{x,w}(q^{-1})-q^{\frac{\l(x)}{2}}q^{-\frac{\l(w)}{2}}P_{x,w}(q)=q^{\frac{\l(x)}{2}}q^{-\frac{\l(w)}{2}} \sum_{x < y \leq w} R_{x,y}P_{y,w}.
\end{equation}
The bound on the degree of $P_{x,w}$ implies that no cancellation occurs on the left hand side of (\ref{recdaC}). Hence, we conclude that the polynomial $P_{x,w}$ satisfying (\ref{recdaC}) is unique.
\end{proof}

The polynomials $\{P_{x,w}\}_{x,w \in W(X)}$ are the so--called \emph{Kazhdan--Lusztig polynomials} of $W(X)$. In \cite[\S 7.9]{Hum} it is shown that one can substitute the basis $\{C_w:\, w \in W(X)\}$ with the equivalent basis $\{C_w':\,w \in W(X)\}$, where 
\begin{equation}\label{C'w}
C_w'\=q^{-\frac{\l(w)}{2}}\sum_{x \leq w}P_{x,w}(q)T_x.
\end{equation}
We will refer to the latter basis as the \emph{Kazhdan--Lusztig basis} for $\H(X)$.\\
Observe that $C_w'=\varepsilon_w \sigma(C_w)$, so that, by Proposition \ref{issi}, we achieve 
\begin{equation} \label{icec}
\iota(C_w')=\varepsilon_w\iota\sigma(C_w)= \varepsilon_w \sigma \iota(C_w)=\varepsilon_w \sigma(C_w)=C_w'.  
\end{equation}
The Kazhdan--Lusztig polynomials can be computed inductively by appealing to the following result. 
\begin{pro} \label{indkl}
Let $x,w \in W(X)$ be such that $x \leq w$. Then $$q^{\l(w)-\l(x)}P_{x,w}(q^{-1})= \sum_{y \in [x,w]}R_{x,y}(q)P_{y,w}(q).$$
\end{pro} 
\begin{proof}
\begin{eqnarray}
\iota(C_w') & = & q^{\frac{\l(w)}{2}}\sum_{x \leq w}P_{x,w}(q^{-1})\left(T_{x^{-1}}\right)^{-1}    \nonumber \\
            & = & q^{\frac{\l(w)}{2}}\sum_{x \leq w}P_{x,w}(q^{-1})\left(\varepsilon_x q^{-\l(x)}\sum_{y \leq x} \varepsilon_y R_{y,x}(q)T_y   \right)    \nonumber \\
            & = & q^{\frac{\l(w)}{2}}\sum_{y \leq w} \left(\sum_{x \in [y,w]} \varepsilon_x q^{-\l(x)} R_{y,x}(q)P_{x,w}(q^{-1}) \right) T_y.    \label{prokl}
\end{eqnarray}
The relation (\ref{icec}) implies that the coefficient of $T_y$ in (\ref{prokl}) and that of $T_y$ in (\ref{C'w}) are the same, that is 
$$q^{-\l(w)}P_{y,w}(q)=\sum_{x \in [y,w]} \varepsilon_x q^{-\l(x)} R_{y,x}(q)P_{x,w}(q^{-1}).$$
Therefore, $$\iota(q^{-\l(w)}P_{y,w}(q))=q^{\l(w)}P_{y,w}(q^{-1})=\sum_{x \in [y,w]} \varepsilon_x q^{\l(x)} R_{y,x}(q^{-1})P_{x,w}(q),$$
and the statement follows by applying Proposition \ref{rsymm}. 
\end{proof}

It is a routine exercise to prove the following properties (see, \textrm{e.g.}, \cite[\S 5, Exercise 3 and Exercise 7(a)]{bb-ccg}).
\begin{lemma} \label{rp2}
Let $x,w \in W(X)$ be such that $x \leq w$. If $\ell(w)-\ell(x) \leq 2$ then 
\begin{itemize}
\item[{\rm (i)}] $R_{x,w}=(q-1)^{\ell(w)-\ell(x)}$;
\item[{\rm (ii)}] $P_{x,w}=1$.
\end{itemize} 
\end{lemma}

\begin{defi} \label{topcoeff}
Let $x,w \in W(X)$ such that $x < w$. Define $\mu(x,w) \in \Z$ to be the top coefficient of $P_{x,w}$, namely 
$$ \mu(x,w)\stackrel{\rm def}{=}[q^{\frac{\l(w)-\l(x)-1}{2}}]P_{x,w}.$$ 
If $\mu(x,w) \neq 0$ then we write $x \prec w$. 
\end{defi}
In the symmetric group $S_4$ there are exactly two pairs of element $(x,w)$ such that $x \prec w$. They are $([1,3,2,4],[3,4,1,2])$ and $([2,1,4,3],[4,2,3,1])$ (see \cite[\S 7.12]{Hum}). In both cases, by means of Proposition \ref{indkl}, one easily checks that $P_{x,w}=q+1$. 

The product of two Kazhdan--Lusztig basis elements may be computed by means of the following well--known formula, which is implicit in the proof of Theorem \ref{kl-pol}. 
\begin{pro}\label{pro klmolt}
Let $s,w \in W(X)$, with $s \in S(X)$. Then 
\[
C'_sC'_w= 
\begin{cases}
C'_{sw}+\sum_{\substack{x \prec w  \\ sx<x}}\mu(x,w)C'_x  & \mbox{if } \l(sw)>\l(w);  \\
(q^{\frac{1}{2}}+q^{-\frac{1}{2}})C'_w                    & \mbox{otherwise}.
\end{cases}   \nonumber
\]
\qed
\end{pro}
Moreover, the proof of Theorem \ref{kl-pol} results in the following interesting relation. 
\begin{teo} \label{teo klqc}
Let $x,w \in W(X)$ be such that $x \leq w$, and let $s \in S(X)$ be such that $sw < w$. Then,
\begin{equation} \label{recpxw}
P_{x,w} = q^{1-c}P_{sx,sw} + q^cP_{x,sw} - \sum_{\{z:\,sz<z\}}q^{\frac{\l(w)-\l(z)}{2}}\mu(z,sw)P_{x,z},
\end{equation}
with $c=1$ if $sx<x$ and $c=0$ otherwise. \qed
\end{teo}
An important consequence of Theorem \ref{teo klqc} follows. 
\begin{cor} \label{kl property}
Let $x,w \in W(X)$ be such that $x \leq w$, and let $s \in S(X)$ be such that $sw<w$. Then $P_{x,w}=P_{sx,w}$.
\end{cor} 
\begin{proof}
We proceed by induction on $\l(w)$. If $\l(w) \leq 1$ then $P_{x,w}=1$ by Lemma \ref{rp2}. Let $\l(w)>1$ and suppose that $sx<x$. Then, by (\ref{recpxw}), we get 
\begin{eqnarray}
P_{sx,w}    & = & qP_{x,sw} + P_{sx,sw}- \sum_{\{z:\,sz<z\}}q^{\frac{\l(w)-\l(z)}{2}}\mu(z,sw)P_{sx,z} \nonumber \\
            & = & qP_{x,sw} + P_{sx,sw}- \sum_{\{z:\,sz<z\}}q^{\frac{\l(w)-\l(z)}{2}}\mu(z,sw)P_{x,z}, \nonumber \\
            & = & P_{x,w} \nonumber 
\end{eqnarray}
since $sz<z \leq sw$. The case $x<sx$ follows similarly.  
\end{proof}

The following is needed in Proposition \ref{pro dsgn} (see \cite[\S 5, Exercise 17]{bb-ccg}).
\begin{pro} \label{pro pdpol}
Let $x, w \in W(X)$. Then $$\sum_{x \leq w} \varepsilon_x P_{x,w}=\delta_{e,w}.$$
\end{pro}
\begin{proof}
The case $w=e$ is trivial. Suppose $w \neq e$. Let $s \in S(X)$ such that $sw<w$. Combining Proposition \ref{kl property} and Lemma \ref{lif lemma}, we get
\begin{eqnarray}
\sum_{x \leq w} \varepsilon_x P_{x,w} &=& \sum_{\substack{x \leq w \\ sx<w}} \varepsilon_x P_{x,w} + \sum_{\substack{x \leq w \\ sx>x}} \varepsilon_x P_{x,w}  \nonumber \\
                                      &=& \sum_{\substack{x \leq w \\ sx<x}} \varepsilon_x P_{x,w} + \sum_{\substack{y \leq w \\ sy<y}} \varepsilon_{sy} P_{sy,w}  \nonumber \\
                                      &=& \sum_{\substack{x \leq w \\ sx<x}} \varepsilon_x P_{x,w} + \sum_{\substack{x \leq w \\ sy<y}} \varepsilon_{sy} P_{y,w}   \nonumber \\
                                      &=& \sum_{\substack{x \leq w \\ sx<x}} \varepsilon_x P_{x,w} - \sum_{\substack{y \leq w \\ sy<y}} \varepsilon_y P_{y,w}=0. \nonumber
\end{eqnarray}
\end{proof}


\chapter{The generalized Temperley--Lieb algebra $TL(X)$} \label{ch gtla}

\section{Definition of generalized Temperley--Lieb algebra}

Throughout this section we will denote by $X$ an arbitrary Coxeter graph. Let $s_i, s_j \in S(X)$ and denote by $\langle s_i,s_j \rangle$ the subgroup of $W(X)$ generated by $s_i$ and $s_j$. Following \cite{gr-phd}, we consider the two--sided ideal $J(X)$ generated by all elements of $\H(X)$ of the form $$\sum_{w \in \langle s_i,s_j \rangle}T_w,$$ where $(s_i,s_j)$ runs over all pairs of non--commuting elements in $S(X)$ such that the order of $s_is_j$ is finite.   
\begin{defi}
The generalized Temperley--Lieb algebra is $$TL(X)\=\H(X)/J(X).$$
\end{defi}
When $X$ is of type $A$, we refer to $TL(X)$ as the Temperley--Lieb algebra.\\ 
If we project the $T$--basis of $\H(X)$ to the quotient $\H(X)/J(X)$ we obtain a basis for $TL(X)$. Let $t_w=\sigma(T_w)$, where $\sigma: \H \rightarrow \H/J$ is the canonical projection. A proof of the following can be found in \cite{gl-cbh}.
\begin{teo}
The Temperley--Lieb algebra admits an $\A$--basis of the form $\{t_w:\, w \in W_c(X)\}$.
\end{teo}
We call $\{t_w:\, w \in W_c(X)\}$ the \emph{$t$--basis} of $TL(X)$. By (\ref{prod}), it satisfies 
\begin{equation} \label{mult-b}
t_{w}t_{s} = \left\{ \begin{array}{ll} 
t_{ws}             & \mbox{if $\l(ws)>\l(w)$,} \\
q t_{ws}+(q-1)t_w  & \mbox{if $\l(ws)<\l(w)$.}
\end{array} \right.
\end{equation} 
Observe that it may be the case that $ws \not \in W_c(X)$. We will see in Proposition \ref{d-pol} how to handle this case.\\
The next result appears in \cite[Lemma 1.4]{gl-cbh}.
\begin{lemma} \label{ijj}
The involution $\iota$ fixes the ideal $J(X)$. 
\end{lemma}
\begin{proof}
Let $w \in \langle s_i, s_j \rangle$, with $s_i, s_j \in S(X)$ such that $2 < m(s_i,s_j) < +\infty$. Therefore $w^{-1} \leq w_0(s_i,s_j)$, where $w_0(s_i, s_j)$ denotes the longest element in $\langle s_i, s_j \rangle$. Hence, there exists $u \in \langle s_i, s_j \rangle$ such that $w_0(s_i,s_j)=w^{-1}u$ and $\l(w_0(s_i,s_j))=\l(w^{-1})+\l(u)$. Observe that the map $w^{-1} \mapsto u$ is a bijection of $\langle s_i, s_j \rangle$. We have $T_{w_0(s_i,s_j)}=T_{w^{-1}}T_u$, that is $\left(T_{w^{-1}} \right)^{-1}=T_u \left(T_{w_0(s_i,s_j)}\right)^{-1}$. Then 
$$\iota\left(\sum_{w \in \langle s_i,s_j \rangle}T_w \right)=\sum_{w \in \langle s_i,s_j \rangle} (T_{w^{-1}})^{-1}=\left(\sum_{u \in \langle s_i,s_j \rangle}T_u \right) (T_{w_0}(s_i,s_j))^{-1} \in J(X).$$ 
Therefore, we conclude that $\iota(J(X)) \subseteq J(X)$ and the statement follows.
\end{proof}

From Lemma \ref{ijj} it follows that $\iota$ induces an involution on $TL(X)$, which we still denote by $\iota$, if there is no danger of confusion. 
\begin{pro}
The map $\iota$ is a ring homomorphism of order $2$ such that $\iota(t_w)=(t_{w^{-1}})^{-1}$ and $\iota(q)=q^{-1}$. \qed
\end{pro}

\section{The polynomials $D_{x,w}$}
For the theory developed in this section we refer to \cite{gl-cbh} and to \cite{bre-algcomb}. \\

\begin{pro}\label{d-pol}
Let $w \in W(X)$. Then there exists a unique family of polynomials $\{D_{x,w}\}_{x \in W_c(X)} \subseteq \mathbb{Z}[q]$ such that 
$$ t_w=\sum_{\substack{x \in W_c(X) \\ x \leq w}} D_{x,w}t_x, $$
where $D_{w,w}=1$ if $w \in W_c(X)$. Furthermore, $D_{x,w}=0$ if $x \not \leq w$.
\end{pro} 
\begin{proof}
We proceed by induction on $\l(w)$. First, observe that if $x,w \in W_c(X)$ then the statement is trivially true and this covers the case $\l(w) \leq 1$. Now, denote by $w_0(s_i,s_j)$ the longest element in $\langle s_i,s_j \rangle$, for all $s_i,s_j \in S(X)$ such that $m(s_i,s_j)< \infty$. 
Let $\l(w) \geq 2$ such that $w \not \in W_c(X)$. Then there exist $s_i,s_j \in S(X)$ and $u,v \in W(X)$ such that $m(s_i,s_j)< \infty$ and $w=uw_0(s_i,s_j)v$, with $\l(w)=\l(u)+\l(w_0(s_i,s_j))+\l(v)$. Hence 
\[
t_w = t_u t_{w_0} t_v =t_u \left( -\sum_{x<w_0}t_x \right) t_v = -\sum_{x<w_0} t_ut_xt_v.
\]
Hence, $t_w$ is a linear combination of element $t_y=t_ut_xt_v$, where $y<w$. By the induction hypotesis, each term $t_y$ can be expressed as a $\Z[q]$--linear combination of elements $t_z$, with $z<y,\, z \in W_c(X)$ and the statement follows. 
\end{proof}
To get a better feeling for how $D$--polynomials were defined, we compute $D_{x,w}$ step by step, in the case $w=s_1s_2s_3s_2s_1 \in W(A_3)$.  
\begin{ese}
Let $w=s_1s_2s_3s_2s_1=1,2,3,2,1 \in W(A_3)$. To compute $D_{x,w}(q)$ we need to combine (\ref{mult-b}) with the relation $$t_{s_is_{i+1}s_i}=-t_e-t_{s_i}-t_{s_{i+1}}-t_{s_i,s_{i+1}}-t_{s_{i+1},s_i}.$$
\begin{eqnarray}
t_{1,2,3,2,1} &=& t_1 \cdot t_{2,3,2} \cdot t_1                           \nonumber \\
              &=& t_1 \cdot (-t_e-t_{2}-t_{3}-t_{2,3}-t_{3,2}) \cdot t_1  \nonumber \\
              &=& -t_1 \cdot t_1-t_{1,2,1}-t_{1,3} \cdot t_1-t_{1,2,3,1}-t_{1,3,2,1}     \nonumber   \\
              &=& -(q t_e +(q-1) t_1)-t_{1,2,1}-(q t_{3} +(q-1) t_{1,3})-t_{1,2,1,3}-t_{3,1,2,1}  \nonumber  \\
              &=& -q t_e - (q-1)t_1+t_e+t_1+t_2+t_{1,2}+t_{2,1}-qt_3-(q-1)t_{1,3}+                \nonumber \\
              & & +\: t_3+t_{1,3}+t_{2,3}+t_{1,2,3}+t_{2,1,3}+t_3+t_{1,3}+t_{3,2}+t_{3,2,1}+t_{3,1,2}  \nonumber  \\
              &=& (1-q) t_e +(2-q) t_1 +t_2+(2-q)t_3+(3-q)t_{1,3}+t_{1,2}+  \nonumber \\
              & & +\: t_{2,1}+t_{2,3}+t_{3,2}+t_{1,2,3}+t_{3,2,1}+t_{1,3,2}+t_{2,1,3}.\nonumber 
\end{eqnarray}
Therefore we get $D_{e,w}=1-q$, $D_{s_1,w}=D_{s_3,w}=2-q$ and $D_{s_1s_3,w}=3-q$. Moreover, $D_{x,w}=1$ for the rest of the elements $x \leq w$. \qed 
\end{ese}
The following proposition mirrors Theorem \ref{r-pol} in the context of the Temperley--Lieb algebra.
\begin{pro}\label{a-pol}
Let $w \in W_c(X)$. Then there exists a unique family of polynomials $\{a_{y,w}\}_{y \in W_c(X)} \subseteq \Z[q]$ such that 
\[ (t_{w^{-1}})^{-1}=q^{-\l(w)}\sum_{\substack{y \in W_c(X) \\ y\leq w}}a_{y,w}t_y, \]
where $a_{w,w}=1$ and $a_{y,w}=0$ if $y \not \leq w$.
\end{pro}
\begin{proof}
By Lemma \ref{ijj} we get $$ (t_{w^{-1}})^{-1} = \iota(t_w)= \iota(\sigma(T_w)) = \iota(T_w+J) = \iota(T_w)+J = T_{w^{-1}}^{-1}+J.$$
Therefore, from Theorem \ref{r-pol}, we achieve
\begin{eqnarray}
T_{w^{-1}}^{-1}+J & = &\varepsilon_wq^{-\l(w)}\sum_{x\leq w}\varepsilon_x R_{x,w}T_x+J \nonumber \\
                  & = &\varepsilon_wq^{-\l(w)}\sum_{x\leq w}\varepsilon_x R_{x,w}t_x \nonumber \\
                  & = &\varepsilon_wq^{-\l(w)}\sum_{x\leq w}\varepsilon_x R_{x,w}\sum_{\substack{y \in W_c(X) \\ y \leq x}}D_{y,x}t_y \nonumber \\
                  & = &\sum_{\substack{y\in W_c(X) \\ y\leq w}} q^{-\l(w)}\left(\sum_{y\leq x \leq w} \varepsilon_x \varepsilon_w R_{x,w}D_{y,x} \right)t_y, \nonumber \\
                  & = & q^{-\l(w)}\sum_{\substack{y\in W_c(X) \\ y \leq w}} a_{y,w} t_y, \nonumber
\end{eqnarray}
since the expression in the round brackets is a polynomial with integer coefficients, depending only on the elements $y$ and $w$. 
\end{proof}

The polynomials $\{a_{x,w}\}$ associated to $TL(X)$ play the same role as the polynomials $\{R_{x,w}\}$ associated to $\H(X)$. They both represent the coordinates of elements of the form $\iota(t_{w})$ (respectively $\iota(T_{w})$) with respect to the $t$--basis (respectively $T$--basis). The next result is the analogue of a well--known result for the $R$--polynomials (see \cite[\S 7.8]{Hum}). We follow the proof given in \cite{bre-algcomb}.

\begin{pro}\label{p-app}
Let $y,w \in W_c(X)$ be such that $y \leq w$. Then 
\[ \sum_{x \in [y,w]_c} q^{\l(w)-\l(x)}a_{y,x}(q)a_{x,w}(q^{-1})=\delta_{y,w}. \]
\end{pro}
\begin{proof}
From Proposition \ref{a-pol} we get
\begin{eqnarray}
t_w & = & \iota(\iota(t_w)) = \iota \left(t_{w^{-1}}^{-1} \right)=\iota \left(q^{-\l(w)}\sum_{\substack{x \in W_c(X) \\ x\leq w}}a_{x,w}(q)t_x \right) \nonumber \\
    & = & q^{\l(w)}\sum_{\substack{x \in W_c(X) \\ x\leq w}}a_{x,w}(q^{-1})\iota(t_x) \nonumber \\
    & = & q^{\l(w)}\sum_{\substack{x \in W_c(X) \\ x\leq w}}a_{x,w}(q^{-1}) \cdot q^{-\l(x)} \sum_{\substack{y \in W_c(X) \\ y \leq x}} a_{y,x}(q)t_y \nonumber \\
    & = & \sum_{\substack{y \in W_c(X) \\ y \leq w}} \left(\sum_{x \in [y,w]_c}q^{\l(w)-\l(x)}a_{y,x}(q) a_{x,w}(q^{-1}) \right) t_y.\label{eqn apol}
\end{eqnarray}
Hence, the expression in the round brackets in (\ref{eqn apol}) is equal to $1$ if $y=w$ and $0$ otherwise. 
\end{proof}

The generalized Temperley--Lieb algebra admits a \emph{canonical basis} $\{c_w:\, w \in W_c(X)\}$ that is analogous to the Kazhdan--Lusztig basis $\{C_w':\, w \in W(X)\}$ of $\H(X)$. To introduce this new basis we need some definitions. Let $\L$ be the free $\Z[q^{-\frac{1}{2}}]$--module of $TL(X)$ with basis $\{q^{-\frac{\l(w)}{2}}t_w:\, w \in W_c(X)\}$ and let $\pi:\L\rightarrow \L/q^{-\frac{1}{2}}\L$ be the canonical projection. 
\begin{defi} \label{defici}
If there exists a unique basis $\{c_w:\,w \in W_c(X)\}$ for $\L$ such that $c_w$ is $\iota$--invariant and $\pi(c_w)=\pi(q^{-\frac{\l(w)}{2}}t_w)$, then the basis $\{c_w:\,w \in W_c(X)\}$ is called an IC basis for $TL(X)$ with respect to the triple $(\{q^{-\frac{\l(w)}{2}}t_w\},\, \iota,\, \L)$.
\end{defi}

The next results will enable us to state the existence of the IC basis for $TL(X)$, $X$ being any Coxeter graph. Observe that the existence of such a basis was established in \cite{gl-cbh}.\\
Here we introduce a family of polynomials in a purely combinatorial way, as explained in \cite{bre-algcomb}.
\begin{teo}\label{l-pol}
Let $w \in W_c(X)$. Then there exists a unique family of polynomials $\{L_{x,w}\}_{x \in W_c(X)} \subseteq q^{-\frac{1}{2}}\Z[q^{-\frac{1}{2}}]$ such that
\begin{itemize}
\item[{\rm (i)}] $L_{x,w}=0 \mbox{ if } x \not \leq w$;   
\item[{\rm (ii)}] $L_{x,x}=1$;
\item[{\rm (iii)}] $L_{x,w} \in q^{-\frac{1}{2}}\Z[q^{-\frac{1}{2}}] \mbox{ if } x<w$;
\item[{\rm (iv)}] $L_{x,w}(q^{-\frac{1}{2}})=\sum_{y \in [x,w]_c}q^{\frac{\l(x)-\l(y)}{2}}a_{x,y}(q)L_{y,w}(q^{\frac{1}{2}})$.  
\end{itemize}
\end{teo}
\begin{proof}
For the existence part, we proceed by induction on $\l(w)-\l(x)$. Let $\l(w)-\l(x)>0$ and define $p\=\iota\left(L_{x,w}-\iota(L_{x,w}) \right)$. Then
\begin{eqnarray}
p &=& \iota\left(\sum_{y \in (x,w]_c}q^{\frac{\l(x)-\l(y)}{2}}a_{x,y}(q)\iota(L_{y,w})\right) \nonumber \\
  &=& \sum_{y \in (x,w]_c}q^{\frac{\l(y)-\l(x)}{2}}a_{x,y}(q^{-1}) L_{y,w}  \nonumber \\
  &=& \sum_{y \in (x,w]_c}q^{\frac{\l(y)-\l(x)}{2}}a_{x,y}(q^{-1})\left( \sum_{z \in [y,w]_c} q^{\frac{\l(y)-\l(z)}{2}}a_{y,z}(q) \iota(L_{z,w}) \right) \nonumber \\
  &=& \sum_{z \in (x,w]_c} \iota(L_{z,w}) q^{\frac{\l(z)-\l(x)}{2}} \left( \sum_{y \in (x,z]_c}q^{\l(y)-\l(z)}a_{x,y}(q^{-1})a_{y,z}(q) \right)   \nonumber \\
  &=& \sum_{z \in (x,w]_c} \iota(L_{z,w}) q^{\frac{\l(z)-\l(x)}{2}} \left( -q^{\l(x)-\l(z)}a_{x,z}(q) \right)  \nonumber \\
  &=& -\sum_{z \in (x,w]_c}\iota(L_{z,w}q^{\frac{\l(x)-\l(z)}{2}}) a_{x,z}(q)=-\iota(p).  \nonumber
\end{eqnarray}
Hence, $p$ is antisymmetric and the existence statement follows.\\
Next we prove the uniqueness part, assuming the existence of $L_{x,w}$. We proceed by induction on $\l(w)-\l(x)$, where $w \in W(X)$ is fixed. Assume that the $L_{y,w}$ can be chosen in a unique way, for all $x<y \leq w$. We will show that $L_{x,w}$ is uniquely determined. Consider equation \textrm{(iv)} and move the term for $y=x$ to the left, so that
$$L_{x,w}(q^{-\frac{1}{2}})-L_{x,w}(q^{\frac{1}{2}})=\sum_{y \in (x,w]_c}q^{\frac{\l(x)-\l(y)}{2}}a_{x,y}(q)L_{y,w}(q^{\frac{1}{2}}).$$
Observe that no cancellation occurs on the left hand side of (\ref{recdaC}), since $L_{x,w}(q^{-\frac{1}{2}}) \in q^{-\frac{1}{2}}\Z[q^{-\frac{1}{2}}]$. Hence, we conclude that the polynomial $L_{x,w}$ satisfying \textrm{(iv)} is unique.  
\end{proof}

\begin{teo} \label{defic}
Let $X$ be an arbitrary Coxeter graph. Let $w \in W_c(X)$ and define 
$$c_w\=\sum_{\substack{x \in W_c(X) \\ x \leq w}}q^{-\frac{\l(x)}{2}}L_{x,w}(q^{-\frac{1}{2}})t_x.$$
Then $\{c_w:\, w \in W_c(X)\}$ is an IC basis for $TL(X)$. 
\end{teo}
\begin{proof}
We have to show that every basis element $c_w$ satisfies the conditions of Definition \ref{defici}. First,  $\pi(c_w)=\pi(q^{-\frac{\l(w)}{2}}t_w)$ since $L_{x,w}=1 \Leftrightarrow x=w$ and $L_{x,w}(q^{-\frac{1}{2}}) \in q^{-\frac{1}{2}}\Z[q^{-\frac{1}{2}}]$ if $x<w$ (see Theorem \ref{l-pol}). Second, we check that $c_w$ is $\iota$--invariant.   
\begin{eqnarray}
\iota(c_w) & = & \sum_{\substack{x \in W_c(X) \\ x \leq w}}q^{\frac{\l(x)}{2}} (L_{x,w}(q^{\frac{1}{2}})) \iota(t_x)  \nonumber \\
           & = & \sum_{\substack{x \in W_c(X) \\ x \leq w}}q^{\frac{\l(x)}{2}} \left(\sum_{y \in [x,w]_c} q^{\frac{\l(y)-\l(x)}{2}} a_{x,y}(q^{-1}) L_{y,w}(q^{-\frac{1}{2}})\right) \iota(t_x) \nonumber \\
           & = & \sum_{\substack{x \in W_c(X) \\ x \leq w}} \left( \sum_{y \in [x,w]_c} q^{\frac{\l(y)}{2}} a_{x,y}(q^{-1}) L_{y,w}(q^{-\frac{1}{2}})\right) q^{-\l(x)} \sum_{\substack{z \in W_c(X) \\ z \leq x}} a_{z,x}(q)t_z \nonumber \\
           & = & \sum_{z \in W_c(X)} \left(\sum_{\substack{y \in W_c(X) \\ y \leq w}} q^{-\frac{\l(y)}{2}} \left(\sum_{x \in [z,y]_c} q^{\l(y)-\l(x)} a_{z,x}(q) a_{x,y}(q^{-1})\right) L_{y,w}(q^{-\frac{1}{2}})  \right)t_z.  \nonumber
\end{eqnarray}
Finally, by applying Proposition \ref{p-app} we get 
$$ \sum_{z \in W_c(X)} \left(\sum_{\substack{y \in W_c(X) \\ y \leq w}} q^{-\frac{\l(y)}{2}} \delta_{z,y} L_{y,w}(q^{-\frac{1}{2}}) \right)t_z = \sum_{\substack{z \in W_c(X) \\ z \leq w}}q^{-\frac{\l(z)}{2}}L_{z,w}(q^{-\frac{1}{2}})t_z $$
as desired.
\end{proof}

\begin{cor}\label{ic basis}
There exists an IC basis for $TL(X)$ with respect to the triple $(\{q^{-\frac{\l(w)}{2}}t_w\},\, \iota,\, \L)$. \qed
\end{cor}
Comparing the definition of $c_w$ (see Theorem \ref{defic}) with that of $C_w'$ (see relation \ref{C'w}), we notice that the polynomials $L_{x,w}$ play the same role as $q^{\frac{\l(x)-\l(w)}{2}}P_{x,w}$, where $P_{x,w}$ are the Kazhdan--Lusztig polynomials defined in Theorem \ref{kl-pol}.

\section{The projection property}

Since the Kazhdan--Lusztig basis and the canonical basis are both $\iota$--invariant and since $\iota(J)=J$, it is natural to ask to what extent $\{\sigma(C_w'):\, w \in W(X) \}$ coincides with $\{c_w:\, w \in W_c(X)\}$. Denote by $\C$ the set $\{C_w':\, w \in W_c(X)\}$.
\begin{defi}
We say that a Coxeter graph $X$ satisfies the projection property if $$\sigma(\C)=\{c_w:\, w \in W_c(X)\}.$$ 
\end{defi}
A sufficient condition for a Coxeter graph to have the projection property is given in \cite[Proposition 1.2.3]{gl-ppk}.
\begin{pro} \label{p-pp}
Let $\sigma: \H(X) \rightarrow \H(X)/J(X)=TL(X)$ be the canonical projection. If $Ker(\sigma)$ is spanned by the basis element $C_w'$ that it contains, then $X$ satisfies the projection property.
\end{pro}
The kernel of the canonical projection $\sigma: \H(A_n)\rightarrow \H(A_n)/J(A_n)$ is spanned by all elements $C_w'$ such that $w \notin W_c(A_n)$ (see \cite[Proposition 3.1.1]{fg-mtl}).
Therefore, type $A$ has the projection property and the same argument holds for types $B$ and $I_2(m)$ (see \cite[Theorem 3.1.1]{gl-fck} and \cite[Proposition 6.14]{gre-gjt}). This fact was also verified for types $H_3$, $H_4$, and $F_4$, by means of computer calculations (see \cite[\S 3]{gl-fck}). The converse of Proposition \ref{p-pp} is not true in general. The following counterexample is given in \cite[Example 2.5]{los-klb}, where Losonczy shows that type $D_n$, with $n \geq 4$, has the projection property, but Proposition \ref{p-pp} does not apply. 

\begin{ese} \label{d4klcell}
Let $W=W(D_4)$ with set of generators $\{s_1,\,s_2,\,s_3,\,s_4\}$, where $s_3$ corresponds to the branch node, as shown in Figure \ref{fig d4}. On the one hand, if $s_2s_3s_4s_3s_1s_2s_3$ is a reduced expression of $w \not \in W_c(D_4)$ then $C'_w \in J$. Hence $C'_{s_1}C'_w \in J$. On the other hand, by Proposition \ref{pro klmolt} we get
$$C'_{s_1}C'_w=C'_{s_1s_2s_3s_4s_3s_1s_2s_3}+ C'_{s_1s_2s_4s_3}.$$ Therefore there exists a nonzero $\A$--linear combination of elements of $\C$ that belongs to $J$. But $\sigma(\C)=\{\sigma(C_w'):\, w \in W_c(X)\}$ is a basis for $TL(D_4)$, so we conclude that $J(D_4)$ is not spanned by the Kazhdan--Lusztig basis elements that it contains. \qed    
\end{ese}

\begin{figure}[!hbtp]
\[
\xymatrix @C=3pc {                & *=0{\bullet} \ar@{-}[d]^<{s_1}              &                                       \\
*=0{\bullet} \ar@{-}[r]^{}_<{s_2} & *=0{\bullet} \ar@{-}[r]^{}_<{s_3}_>{s_4}  & *=0{\bullet} }
\]
\caption{Coxeter graph $D_4$.} \label{fig d4}
\end{figure}
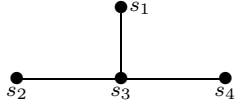

Similar problems arise whenever $X$ is a Coxeter graph that contains a vertex connected to at least three other vertices. Graphs having this property are sometimes called \emph{branching graphs}. Some examples of branching graphs are types $D$, $E_6$, $E_7$, $E_8$, and in these cases Proposition \ref{p-pp} does not apply (see \cite[Corollary 3.1.3]{gl-fck}). However, no example of a Coxeter group that fails to satisfy the projection property is known and even in type $E$ it is an open problem (see \cite[\S 2]{gl-fck}).\\
We remark that Example \ref{d4klcell} can be translated in the ``language'' of Kazhdan--Lusztig cells. We stick to the definitions given in \cite[Definition 1.2.2]{gl-fck}. 
\begin{defi} \label{klcell}
Let $x,w \in W(X)$. If there exists a chain $x=x_0,x_1,\cdots,x_k=w$, $k\geq 0$, such that for every $i<k$, $C_{x_i}'$ occurs with non--zero coefficient in the linear expansion of $C_s'C_{x_{i+1}}'$ for some $s \in S(X)$ such that $sx_{i+1}>x_{i+1}$, then we write $x \stackrel{L}{\leq} w$.  
\end{defi}
Define the equivalence relation $\stackrel{L}{\sim}$ as follows: $x \stackrel{L}{\sim} w \Leftrightarrow x \stackrel{L}{\leq} w \mbox{ and } w \stackrel{L}{\leq} x$.
The equivalence classes with respect to $\stackrel{L}{\sim}$ are called \emph{left cells} of $W(X)$. We write $x \stackrel{R}{\leq} w \Leftrightarrow x^{-1} \stackrel{L}{\leq} w^{-1}$. Finally, we set $x \stackrel{LR}{\leq} w \Leftrightarrow x \stackrel{L}{\sim} w \mbox{ and } x \stackrel{R}{\sim} w$. The equivalence classes with respect to the equivalence relation $\stackrel{R}{\sim}$ (respectively $\stackrel{LR}{\sim}$) are called \emph{right cells} (respectively \emph{two--sided cells}).\\
By Definition \ref{klcell}, we can restate Example \ref{d4klcell} as follows: $s_1s_2s_3s_4 \stackrel{L}{\leq} s_2s_3s_4s_3s_1s_2s_3$. Therefore, $W(D_4) \setminus W_c(D_4)$ is not closed under $\stackrel{L}{\leq}$.\\ 
Observe that Proposition \ref{p-pp} is equivalent to asking that $\sigma(C_w')=0$, for all elements $w\not \in W_c(X)$ (see \cite[Theorem 2.2.3]{gl-fck}). This is a key observation in order to study the $D$--polynomials introduced in Proposition \ref{d-pol}.
In particular, one may wonder whether the map $\sigma: \H(X)\rightarrow \H(X)/J(X)$ satisfies
\begin{equation} \label{Cw-0} 
\sigma(C_w') = \left\{ \begin{array}{ll} 
c_w  & \mbox{if $w \in W_c(X)$,} \\ 
0    & \mbox{if $w \not \in W_c(X)$.} 
\end{array} \right.
\end{equation}
The answer is affirmative in types $A$, $B$, $I_2(m)$, $F_4$, $H_3$ and $H_4$, and negative for types $D$, $E_6$, $E_7$ and $E_8$ (for a complete discussion of these results, see \cite{gre-gjt} and \cite{gl-fck}). More generally, if $X$ is a finite irreducible or affine Coxeter graph, relation (\ref{Cw-0}) holds if and only if $W_c(X)$ is closed under $\stackrel{LR}{\leq}$ or, equivalently, if and only if $W_c(X)$ is a union of two--sided Kazhdan--Lusztig cells (see \cite[Theorem 2.1]{shi-fceii} and \cite[Theorem 2.2.3]{gl-fck}). On the other hand, in \cite{shi-fce} it is shown that $W_c(X)$ is a union of two--sided Kazhdan--Lusztig cells if and only if $X$ is non--branching and $X \neq \widetilde{F_4}$. 

\begin{ese}
Let $W=W(\widetilde{F_4})$ with set of generators $\{s_0,\,s_1,\,s_2,\,s_3,\,s_4\}$, where $m(s_2,s_3)=4$, as shown in Figure \ref{fig f4t}. On the one hand, $s_0s_2s_4$ is a fully commutative element in $W(\widetilde{F_4})$ and $s_0s_1s_0 \not \in W_c(\widetilde{F_4})$. On the other hand, in \cite[\S 3.7]{shi-fce} Shi states that $s_0s_2s_4 \stackrel{LR}{\sim} s_0s_1s_0$. Therefore, $W_c(\widetilde{F_4})$ is not a union of two--sided Kazhdan--Lusztig cells. \qed   
\end{ese}

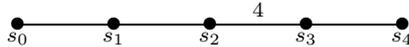
\begin{figure}[!hbtp]
\[
\xymatrix @C=3pc {*=0{\bullet} \ar@{-}[r]^{}_<{s_0}  & *=0{\bullet} \ar@{-}[r]^{}_<{s_1} & *=0{\bullet} \ar@{-}[r]^{4}_<{s_2} &*=0{\bullet} \ar@{-}[r]^{}_<{s_3}_>{s_4}  & *=0{\bullet} }
\]
\caption{Coxeter graph $\widetilde{F_4}$.} \label{fig f4t}
\end{figure}

\begin{teo} \label{teocw0}
Let $X$ be a finite irreducible or affine Coxeter graph. Then, relation (\ref{Cw-0}) holds if and only if $X$ is non--branching and $X \neq \widetilde{F_4}$. \qed
\end{teo}

\chapter{Combinatorial properties of $TL(X)$} \label{comb prop}
\section{Combinatorial properties of $D_{x,w}$} \label{s-dpol}

In the first part of this section we study the $D$--polynomials defined in Proposition \ref{d-pol}. More precisely, we obtain a recurrence relation for the polynomials $\{D_{x,w}\}_{x, \in W_c(X),\, w \in W(X)} \subseteq \Z[q]$, $X$ being an arbitrary Coxeter graph.
Then we will focus on the Coxeter graph satisfying equation (\ref{Cw-0}) and derive some results concerning symmetry properties, the value of the constant term and explicit formulas for $D_{x,w}$ when the Bruhat interval $[x,w]$ has a particular structure. Throughout this chapter $\l(x,w)$ will denote the difference $\l(w)-\l(x)$.\\

\begin{pro} \label{p-gc}
Let $X$ be an arbitrary Coxeter graph. Let $w \not \in W_c(X)$ and $s \in S(X)$ be such that $ws \not \in W_c(X)$, with $\ell(ws)<\ell(w)$. Then, for all $x \in W_c(X),\,x \leq w$, we have 
$$D_{x,w}=\widetilde{D_{x,w}} + \sum_{\substack{y \in W_c(X),\, ys \not \in W_c(X) \\ ys>y}}D_{x,ys}D_{y,ws}.$$
where
\[
\widetilde{D_{x,w}}=\left\{
\begin{array}{ll}
D_{xs,ws}+(q-1)D_{x,ws}       & \mbox{ if } xs<x,           \\
qD_{xs,ws}                    & \mbox{ if } x<xs \in W_c(X),   \\
0                             & \mbox{ if } x<xs \not \in W_c(X),
\end{array} \right.
\] 
\end{pro}
\begin{proof}
On the one hand, by Proposition \ref{d-pol}, we have
\[ t_w=\sum_{\substack{x \in W_c(X) \\ x \leq w}} D_{x,w}t_x. \]
On the other hand, letting $v \stackrel{\rm def}{=} ws$, 
\begin{eqnarray}
t_w & = & t_vt_s= \left(\sum_{\substack{y \in W_c(X) \\ y\leq v}} D_{y,v}t_y \right) t_s \nonumber \\
    & = & \sum_{\substack{y \in W_c(X) \\ y\leq v,\,ys>y}} D_{y,v}t_{ys}+\sum_{\substack{y \in W_c(X) \\ y\leq v,\,ys<y}} D_{y,v}\left(qt_{ys}+(q-1)t_y \right)       \nonumber \\
    & = &\sum_{\substack{y \in W_c(X),\, ys \in W_c(X) \\ y\leq v,\,ys>y}} D_{y,v}t_{ys} + \sum_{\substack{y \in W_c(X),\, ys \not \in W_c(X) \\ y\leq v,\,ys>y}} D_{y,v}t_{ys} \nonumber \\ 
    && +\: \sum_{\substack{y \in W_c(X) \\ y\leq v,\,ys<y}} D_{y,v} qt_{ys} + \sum_{\substack{y \in W_c(X) \\ y\leq v,\,ys<y}} (q-1) D_{y,v} t_y      \nonumber \\
    & = &\sum_{\substack{y \in W_c(X),\, ys \in W_c(X) \\ y\leq v,\,ys>y}} D_{y,v}t_{ys} + \sum_{\substack{y \in W_c(X) \\ y\leq v,\,ys<y}} D_{y,v} qt_{ys} \nonumber \\ 
    && +\: \sum_{\substack{y \in W_c(X) \\ ys<y \leq v}} (q-1) D_{y,v} t_y + \sum_{\substack{y \in W_c(X),\, ys \not \in W_c(X) \\ y\leq v,\,ys>y}} D_{y,v} \left(\sum_{\substack{x \in W_c(X) \\ x \leq ys}}D_{x,ys} t_x  \right)    \nonumber \\
    & = & \sum_{\substack{x \in W_c(X) \\ xs<x \leq w}} D_{xs,ws} t_x + \sum_{\substack{xs \in W_c(X) \\ x \leq w,\,xs>x}}q D_{xs,ws} t_x \nonumber \\ 
    && +\: \sum_{\substack{x \in W_c(X) \\ xs<x \leq w}} (q-1) D_{x,ws} t_x + \sum_{x \in W_c(X)} \left(\sum_{\substack{y \in W_c(X),\, ys \not \in W_c(X) \\ y<ys}} D_{x,ys} D_{y,ws} \right) t_x     \nonumber
\end{eqnarray}
(note that $xs \in W_c(X),\, x<xs \Rightarrow x \in W_c(X)$). Extracting the coefficient of $t_x$ we get 
\[
D_{x,w}=\left\{
\begin{array}{ll}
D_{xs,ws}+(q-1)D_{x,ws}+ b(x,w)       & \mbox{ if } xs<x,           \\
qD_{xs,ws}+ b(x,w)                    & \mbox{ if } x<xs \in W_c(X),   \\
b(x,w)                                & \mbox{ if } x<xs \not \in W_c(X),
\end{array} \right.
\] 
where $$b(x,w)=\sum_{\substack{y \in W_c(X),\, ys \not \in W_c(X) \\ y<ys}}D_{x,ys}D_{y,ws},$$
as desired.
\end{proof}

It is interesting to note that the recursion in Proposition \ref{p-gc} is similar to the one for the parabolic Kazhdan--Lusztig polynomials (see \cite{deo-sga}).\\ 
The preceding recursion can sometimes be solved explicitly. In the proof of the next result we need the notion of Grassmannian and bi--Grassmannian elements (see, \textrm{e.g.}, \cite[\S 3]{ls-tbg} and \cite[\S 5, Exercise 39]{bb-ccg}).
\begin{defi}
Let $w \in W(A_{n-1})$ and define $D_R(w) \= |\{s \in S(X):\, ws<w\}|$. The permutation $w$ is called Grassmannian if $|D_R(w)| \leq 1$ and bi--Grassmannian if $|D_R(w)|=|D_R(w^{-1})|=1$. 
\end{defi}

As a consequence of \cite[Theorem 2.1]{bjs-cps}, if $w \in W(A_{n-1})$ is Grassmannian then $w \in W_c(A_{n-1})$.

\begin{cor}
Let $X$ be of type $A$ and let $x_0 \in W_c(X)$ be a bi--Grassmannian element. If $x_0$ is a maximal element in the Bruhat order of $W_c(X)$, then $D_{x_0,w}=\varepsilon_{x_0} \varepsilon_{w}$, for all elements $w \geq x_0$. 
\end{cor}
\begin{proof}
If $w \in W_c(X)$ then the result is trivial. Suppose $w \not \in W_c(X)$. Observe that if $s \in S(X)$ is such that $x_0s>x_0$, then $x_0s \not \in W_c(X)$. Moreover, if $x_0s>x_0$ and $y \in W_c(X)$ is such that $ys>x_0$ then $y \geq x_0$ by Lemma \ref{lif lemma}, so $y=x_0$. Hence $$\{y \in W_c(X):\,ys>x_0 \}=\{x_0\},$$ for any $s \not \in D_R(x_0)$.
Choosing $s$ such that $x_0s>x_0,\,ws<w$ (there exists such an $s$ since $x_0$ is a bi-Grassmannian element, while $w \not \in W_c(X)$ is not Grassmannian), the third case of Theorem \ref{p-gc} applies, so 
$$D_{x_0,w}=D_{x_0,x_0s}D_{x_0,ws}.$$
Define $\ell(x,w) \= \ell(w)-\ell(x)$, and proceed by induction on $\ell(x_0,w)$. Suppose $\ell(x_0,w)=1$, with $w \not \in W_c(A_{n-1})$. By Proposition \ref{sub prop} it follows that $w$ admits a reduced expression of the form $x_1s_is_{i+1}s_ix_2$, with $x_1,x_2 \in W_c(A_{n-1})$,  $s_i,s_{i+1} \in S(A_{n-1})$, and $x_0$ admits a reduced expression of the form $x_1\widehat{s_i}s_{i+1}s_ix_2$ or $x_1s_is_{i+1}\widehat{s_i}x_2$, since $x_0 \in W_c(A_{n-1})$. Therefore $$t_w=t_{x_1}t_{s_is_{i+1}s_i} t_{x_2}=t_{x_1} (-t_{s_is_{i+1}}-t_{s_{i+1}s_i}-t_{s_{i+1}}-t_{s_i}-t_e) t_{x_2},$$
and the statement follows by applying (\ref{mult-b}).
If $\ell(x_0,w)>1$, then $$D_{x_0,w}=D_{x_0,x_0s}D_{x_0,ws}=-D_{x_0,ws}=-\varepsilon_{x_0} \varepsilon_{ws}=\varepsilon_{x_0} \varepsilon_{w}.$$     
\end{proof}

Observe that a maximal element in the Bruhat order of $W_c(A_{n-1})$ is an element whose one--line notation is of the form $[k+1,k+2, \cdots,n,1,2,\cdots,k]$, with $k \in [n-1]$. 

From here to the end of this section we will denote by $X$ a Coxeter graph satisfying (\ref{Cw-0}). Observe that $D_{x,w}=\delta_{x,w}$ if $x,w \in W_c(X)$.
\begin{lemma} \label{lemain}
For all $x \in W_c(X)$ and $w \not \in W_c(X)$, we have $$\sum_{x \leq y \leq w} D_{x,y}P_{y,w}=0.$$
\end{lemma}
\begin{proof}
Let $w \in W(X)$. Then, by Proposition \ref{d-pol},
\begin{eqnarray}
\sigma(C_w') & = & q^{-\frac{\l(w)}{2}}\sum_{y \leq w}P_{y,w}\sigma(T_y) \nonumber \\
             & = & q^{-\frac{\l(w)}{2}}\sum_{y \leq w}P_{y,w} \left( \sum_{\substack{x \in W_c(X) \\ x \leq y}} D_{x,y}t_x \right) \nonumber \\
             & = & q^{-\frac{\l(w)}{2}}\sum_{\substack{x \in W_c(X) \\ x \leq w}}\left(\sum_{x \leq y \leq w} D_{x,y}P_{y,w}\right)t_x. \nonumber 
\end{eqnarray}
When $w \not \in W_c(X)$ we get $\sigma(C_w')=0$, so the expression in round brackets must vanish and the statement follows. 
\end{proof}
The next result will be useful in \S \ref{sec: a} and \S \ref{sec: l}.
\begin{lemma} \label{ddelta}
Let $x \in W_c(X)$ be such that $xs \not \in W_c(X)$ and let $w \not \in W_c(X)$ be such that $w>ws \in W_c(X)$. Then $$D_{x,w}=-\delta_{x,ws}.$$
\end{lemma}
\begin{proof}
We proceed by induction on $\l(x,w)$. If $\l(x,w)=1$, then $D_{x,w}=D_{x,xs}=-1=-\delta_{x,ws}$. Suppose $\l(x,w)>1$. From Lemma \ref{lemain} and Corollary \ref{kl property} we get 
\begin{eqnarray}
D_{x,w}   & = & -P_{x,w} -\sum_{\substack{t \not \in W_c(X) \\ x < t < w }}D_{x,t}P_{t,w}  \nonumber \\
          & = & -P_{x,w} -D_{x,xs}P_{xs,w}-\sum_{\substack{t \not \in W_c(X), t \neq xs \\ x < t < w }}D_{x,t}P_{t,w}  \nonumber \\
          & = & -P_{x,w} -\underbrace{D_{x,xs}}_{-1}P_{x,w}-\sum_{\substack{t \not \in W_c(X), t \neq xs \\ x < t < w }}D_{x,t}P_{t,w}  \nonumber \\
          & = & -\sum_{\substack{t \not \in W_c(X), t \neq xs \\ x < t < w }}D_{x,t}P_{t,w}  \nonumber \\
          & = & -\sum_{\substack{t>ts \in W_c(X), t \neq xs \\ x < t < w}}D_{x,t}P_{t,w}-\sum_{\substack{t>ts \not \in W_c(X) \\ x < t < w }}D_{x,t}P_{t,w}-\sum_{\substack{t<ts \\ x < t < w }}D_{x,t}P_{t,w}.  \nonumber
\end{eqnarray}
By induction hypothesis, the term $D_{x,t}$ in the first sum is equal to $-\delta_{x,ts}$, since $\l(x,t)<\l(x,w)$. Therefore, the first sum is zero. On the other hand, the second and the third sum can be written as 
\begin{equation} \label{iszero}
-\sum_{\substack{z \not \in W_c(X) \\ x<z<zs<w}} P_{z,w} \left(D_{x,zs}+D_{x,z} \right),
\end{equation}
since $ts>t \not \in W_c(X)$ implies $ts \not \in W_c(X)$. To prove the statement we have to show that the term (\ref{iszero}) is zero. First, observe that $\l(x,z)<\l(x,w)$, since Lemma \ref{lif lemma} implies $z \leq w$, but $z \not \in W_c(X)$ and $w \in W_c(X)$. Moreover, by Proposition \ref{p-gc} and by induction hypothesis, we achieve
\begin{equation}\label{dxzs0}
D_{x,zs}=\sum_{\substack{u \in W_c(X) \\ u<us \not \in W_c(X)}} D_{x,us}D_{u,z}=\sum_{\substack{u \in W_c(X) \\ u<us \not \in W_c(X)}} (-\delta_{x,u})D_{u,z}=-D_{x,z}.
\end{equation}
We conclude that $D_{x,zs}+D_{x,z}=0$, for all $z \not \in W_c(X)$ such that $x<z<zs<ws$, so the sum in (\ref{iszero}) is zero.
\end{proof}
The following is the main result of this section.
\begin{teo} \label{te1}
Let $X$ be such that equation (\ref{Cw-0}) holds. For all $x \in W_c(X)$ and $w \not \in W_c(X)$ such that $x<w$, we have 
$$D_{x,w}=\sum \left((-1)^k \prod_{i=1}^{k}P_{x_{i-1},x_i}\right),$$
where the sum is taken over all the chains $x=x_0<x_1<\cdots <x_k=w$ such that $x_i \not \in W_c(X)$ if $i>0$, and $ 1\leq k \leq \l(x,w)$.
\end{teo}
\begin{proof}
We proceed by induction on $\l(x,w)$.\\
If $\l(x,w)=1$ then, from Lemma \ref{lemain}, we get $D_{x,w} =-P_{x,w}$, proving the claim in this case. If $\l(x,w)>1$ then, from Lemma \ref{lemain} and our induction hypothesis, we have
\begin{eqnarray}
D_{x,w}  & = & -P_{x,w} -\sum_{\substack{t \not \in W_c(X) \\ x < t < w }}D_{x,t}P_{t,w}  \nonumber \\
         & = & -P_{x,w} -\sum_{\substack{t \not \in W_c(X) \\ x < t < w }}P_{t,w} \sum_{k=1}^{\l(x,t)}\left(\sum_{x=x_0<\cdots<x_k=t} (-1)^k\prod_{i=1}^{k}P_{x_{i-1},x_i}  \right) \nonumber \\
           & = & -P_{x,w} +\sum_{\substack{t \not \in W_c(X) \\ x < t < w }} \left(\sum_{k=1}^{\l(x,t)}\left(\sum_{\substack{x=x_0<\cdots<x_{k+1}=w \\ x_k=t}} (-1)^{k+1}\prod_{i=1}^{k+1}P_{x_{i-1},x_i}  \right) \right) \nonumber \\
           & = & \sum_{\substack{t \not \in W_c(X) \\ x < t < w }} \left(\sum_{k \geq 0}\left(\sum_{\substack{x=x_0< \cdots <x_{k+1}=w,\\ x_k=t \mbox{ if } k\not = 0}}(-1)^{k+1}\prod_{i=1}^{k+1}P_{x_{i-1},x_i}   \right) \right) \nonumber \\
           & = & \sum_{k \geq 0} \left(\sum_{\substack{t \not \in W_c(X) \\ x < t < w }} \left(\sum_{\substack{x=x_0< \cdots <x_{k+1}=w,\\ x_k=t \mbox{ if } k\not = 0}}(-1)^{k+1}\prod_{i=1}^{k+1}P_{x_{i-1},x_i}   \right) \right) \nonumber \\
           & = & \sum_{k \geq 0} \left(\sum_{\substack{x=x_0<\cdots \\ \cdots <x_{k+1}=w}} (-1)^{k+1} \prod_{i=1}^{k+1}P_{x_{i-1},x_i}   \right), \nonumber 
\end{eqnarray}   
as desired.
\end{proof}
Theorem \ref{te1} shows that the $D$--polynomials are intimately related to the Kazhdan--Lusztig polynomials, which is not at all obvious from their  definition.\\ 
We now derive some consequences of Theorem \ref{te1}. First we obtain some symmetry properties of the polynomials $\{D_{x,w}\}$. 
\begin{cor} \label{c2-dpol}
Let $x \in W_c(X)$, $w \not \in W_c(X)$ and $x<w$. Then
\begin{itemize}
\item[{\rm (i)}]  
$D_{x,w}=D_{x^{-1},w^{-1}}$;  
\item[{\rm (ii)}] 
$D_{x,w}=D_{w_0xw_0,w_0ww_0}$.
\end{itemize}
\end{cor}
\begin{proof}
By Lemma \ref{maps}, $x^{-1} \in W_c(X)$ for every $x \in W_c(X)$. Therefore we get
\begin{eqnarray}
D_{x^{-1},w^{-1}} & = & \sum_{\substack{x^{-1}=x_0<x_1<\cdots \\ \cdots <x_k=w^{-1}}}\left((-1)^k \prod_{i=1}^{k}P_{x_{i-1},x_i} \right) \nonumber  \\
                     & = & \sum_{\substack{x=x_0^{-1}<x_1^{-1}<\cdots \\ \cdots <x_k^{-1}=w}}\left((-1)^k \prod_{i=1}^{k}P_{x_{i-1},x_i}\right) \nonumber \\
                     & = & \sum_{\substack{x=x_0^{-1}<x_1^{-1}<\cdots \\ \cdots <x_k^{-1}=w}}\left((-1)^k \prod_{i=1}^{k} P_{x_{i-1}^{-1},x_i^{-1}}\right) \nonumber \\
                     & = & \sum_{\substack{x=y_0<y_1<\cdots \\ \cdots <y_k=w}}\left((-1)^k \prod_{i=1}^{k}P_{y_{i-1},y_i} \right). \nonumber \\
                     & = & D_{x,w}, \nonumber
\end{eqnarray}
where we have used a well--known property of the Kazhdan--Lusztig polynomials (see, \textrm{e.g.}, \cite[\S 5, Exercise 12]{bb-ccg}). The same holds for $D_{w_0xw_0,w_0ww_0}$, using the properties in \cite[\S 5, Exercise 13(c)]{bb-ccg}.
\end{proof}

Next, we compute the constant term of the polynomials $D_{x,w}$.
\begin{cor} \label{c-d0}
For all $x \in W_c(X)$ and $w \not \in W_c(X)$ such that $x<w$, we have   
$$D_{x,w}(0)=\sum_{\substack{x=x_0<\cdots \\ \cdots <x_k=w}} (-1)^k,$$ where $x_i \not \in W_c(X)$ if $i>0$, and $ 1\leq k \leq \l(x,w)$. 
\end{cor}
\begin{proof}
The statement follows immediately from Theorem \ref{te1} and the well--known fact that $P_{x,w}(0)=1$ for all $x,w \in W(X)$ such that $x \leq w$ (see, \textrm{e.g.}, \cite[Proposition 5.1.5]{bb-ccg}).   
\end{proof}

By \cite[Proposition 3.8.5]{sta-ec1}, Corollary \ref{c-d0} asserts that $D_{x,w}(0)$ equals the M\"{o}bius function $\mu(\widehat{0},w)$ in the poset $\{y \in W(X) \setminus W_c(X):\,y \in [x,w]\} \cup \{\widehat{0}\}$. This suggests the study of the partial order induced on $W(X)\setminus W_c(X)$ by the Bruhat order.\\ 
We now derive an interesting property for $D$--polynomials.
\begin{pro} \label{pro dsgn}
Let $w \in W(X)$. Then 
$$\sum_{\substack{x \in W_c(X) \\ x \leq w}}\varepsilon_x D_{x,w}=\varepsilon_w.$$
\end{pro}
\begin{proof}
We proceed by induction on $\l(w)$. The proposition is trivial if $w \in W_c(X)$, which covers the case $\l(w) \leq 2$. Suppose that $w \not \in W_c(X)$. Then, by Lemma \ref{lemain} we have
\begin{eqnarray}
\sum_{\substack{x \in W_c(X) \\ x \leq w}}\varepsilon_x D_{x,w} &=& \sum_{\substack{x \in W_c(X) \\ x < w}}\varepsilon_x(-P_{x,w}) + \sum_{\substack{x \in W_c(X) \\ x < w}}\varepsilon_x \left( -\sum_{\substack{t \not \in W_c(X) \\ x<t<w}} D_{x,t}P_{t,w} \right)   \nonumber \\
      &=& -\sum_{\substack{x \in W_c(X) \\ x<w}} \varepsilon_x P_{x,w} - \sum_{\substack{t \not \in W_c(X) \\ t<w}} P_{t,w} \left( \sum_{\substack{x \in W_c(X) \\ x<t}} \varepsilon_x D_{x,t} \right) \nonumber \\  
      &=& -\sum_{\substack{x \in W_c(X) \\ x<w}} \varepsilon_x P_{x,w} - \sum_{\substack{t \not \in W_c(X) \\ t<w}} P_{t,w} \varepsilon_t \nonumber \\  
      &=& -\sum_{x<w} \varepsilon_x P_{x,w}, \nonumber
\end{eqnarray} 
and the statement follows from Proposition \ref{pro pdpol}.
\end{proof}

\begin{lemma} \label{dpol2}
Let $x \in W_c(X), w \not \in W_c(X)$ be such that $x < w$. If $\ell(x,w)=1$ then $D_{x,w}=-1$. If $\ell(x,w)=2$, then
\[
D_{x,w}=\left\{ \begin{array}{rl}
1  & \mbox{ if $k=2$,} \\ 
0  & \mbox{ if $k=1$,} \\
-1 & \mbox{ if $k=0$,}
\end{array} \right. 
\]
with $k \= |\{y \not \in W_c(X):\, x<y<w\}|$.
\end{lemma}
\begin{proof}
By Lemma \ref{rp2}, $P_{x,w}=1$ for all $x,w \in W(X)$ such that $\ell(x,w)\leq 2$. If $\ell(x,w)=1$, then Theorem \ref{te1} implies $D_{x,w}=-P_{x,w}=-1$. If $\ell(x,w)=2$, then the interval $[x,w]$ is isomorphic to the boolean lattice $B_2$ (see, \textrm{e.g.}, \cite[Lemma 2.7.3]{bb-ccg}), so $k \in \{0,1,2\}$. Moreover, $P_{x,y}=P_{y,w}=1$ for all $y \not \in W_c(X)$ such that $x<y<w$, since $\ell(x,y)=\ell(y,w)=1$. From Theorem \ref{te1} we get
\begin{eqnarray}
D_{x,w}    & = & \sum \left((-1)^k \prod_{i=1}^{k} P_{x_{i-1},x_i} \right) \nonumber \\
           & = & -P_{x,w} + \sum_{\substack{y \not \in W_c(X) \\ x<y<w}} P_{x,y}P_{y,w} \nonumber \\           
           & = & -1 + \sum_{\substack{y \not \in W_c(X) \\ x<y<w}} 1 \nonumber \\
           & = & -1+k, \nonumber
\end{eqnarray}
and the statement follows.
\end{proof}

Using Theorem \ref{te1} we obtain an upper bound for the degree of $D_{x,w}$ that is also used in \S \ref{sec: l}.
\begin{pro}\label{p-upb}
For all $x \in W_c(X)$, $w \not \in W_c(X)$ such that $x<w$, we have $deg(D_{x,w}) \leq \frac{1}{2}(\l(w)-\l(x)-1)$.
\end{pro} 
\begin{proof}
Recall from \cite[Theorem 1.1]{kl-rcg} that $deg(P_{x,w}) \leq \frac{1}{2}(\l(w)-\l(x)-1)$ if $x<w$. By Theorem \ref{te1} we know that 
$$D_{x,w}=\sum \left((-1)^k \prod_{i=1}^{k}P_{x_{i-1},x_i}\right),$$ where the sum runs over all the chains $x=x_0<x_1<\cdots <x_k=w$ such that $x_i \not \in W_c(X)$ if $i>0$, and $ 1\leq k \leq \l(x,w)$. Each term $\prod_{i=1}^{k}P_{x_{i-1},x_i}$ has degree
$$\sum_{i=1}^{k} deg(P_{x_{i-1},x_i}) \leq \sum_{i=1}^{k}\frac{1}{2}(\l(x_i)-\l(x_{i-1})-1)=\frac{1}{2}(\l(x_k)-\l(x_0)-k).$$
Since $k \geq 1$, the statement follows.   
\end{proof}

We end this section by deriving from Theorem \ref{te1} a closed formula for the polynomials $D_{x,w}$ indexed by elements $x \in W_c(X)$ and $w \not \in W_c(X)$ such that $([x,w]\cap (W(X)\setminus W_c(X))) \cup \{x\} = [x,w] \cong B_{l(x,w)}$. In type $A$ it is easy to realize this case.
Let $x \in W(A_n)$. Recall that $x$ is said to be a Coxeter element if $s_{\sigma(1)} \cdots s_{\sigma(n)}$ is a reduced expression for $x$, for some $\sigma \in S_n$. 
It is clear that a Coxeter element is always a fully commutative element.

\begin{teo} \label{l-bool}
Let $s_1 s_2 \cdots s_n \cdots s_2 s_1$ be a reduced expression for $w \in W(A_n)$ and let $x \in W(A_n)$ be a Coxeter element. Then the following hold:
\begin{itemize}
\item[{\rm (i)}]  
$x \leq w$;  
\item[{\rm (ii)}]
$[x,w] \cong B_{\ell(x,w)}$;
\item[{\rm (iii)}] 
$([x,w]\cap (W(A_n)\setminus W_c(A_n))) \cup \{x\} = [x,w]$.
\end{itemize}
\end{teo}
\begin{proof}
\textrm{(i)} We find a reduced expression for $x$ that is a subexpression of $$s_1 s_2 \cdots s_n \cdots s_2 s_1.$$ From \cite[Theorem 1.5]{shi-ece} there is a bijection between the set of Coxeter elements and the acyclic orientations of the Coxeter graph $A_n$. Let $A_n^x$ be the acyclic orientation of the graph $A_n$ associated to $x$. We say that $s_i$ is on the left (respectively on the right) of $s_{i+1}$ in $x=s_{\sigma(1)} \cdots s_{\sigma(n)}$ if $s_i \longrightarrow s_{i+1}$ (respectively $s_i \longleftarrow s_{i+1}$) in $A_n^x$. Therefore we are able to produce a reduced expression for $x$ from $A_n^x$ in the following way: set $x_n:=s_n$ and juxtapose $s_{n-1}$ to the left (respectively to the right) of $x_n$ if $s_{n-1} \longrightarrow s_n$ (respectively $s_{n-1} \longleftarrow s_n$). Set $x_{n-1}:= s_{n-1}x_n$ (respectively $x_n s_{n-1}$). Repeat the same process with $x_{n-1}$ and $s_{n-2}$, and so on. The process ends when we get $x_1$. In fact, $x_1$ is a reduced expression for $x$ and $x_1$ is, by construction, a subexpression of $w$. Hence \textrm{(i)} follows from Theorem \ref{sub prop}.

\textrm{(ii)} By Theorem \ref{sub prop}, every element $y \in [x,w]$ admits (at least) one reduced expression that is a subexpression of $s_1 s_2 \cdots s_n \cdots s_2 s_1$. Let $r(y)$ be one of these reduced expressions. Observe that the reduced expression $x_1$ obtained in \textrm{(i)} is a possible choice for $r(x)$. Consider the map $\phi : [x,w] \rightarrow \mathcal{A}$, with $\mathcal{A} = \{(\alpha_1,\cdots,\alpha_{n-1}):\, \alpha_i \in \{1,2\} \}$, such that $\phi(y)=(\alpha_1,\cdots,\alpha_{n-1})$ if and only if $r(y)$ has $\alpha_i$ occurences of the generator $s_i$. By \cite[Corollary 3.3]{m-bek}, the map $\phi$ is well-defined. We claim that $\phi$ is a bijection.\\ 
First, we prove the surjectivity. Fix $(\alpha_1,\cdots,\alpha_{n-1}) \in \mathcal{A}$. We describe an algorithm to construct (a reduced expression $r(y)$ for) an element $y \in W(A_n)$ such that $y \in [x,w]$ and $\phi(y)=(\alpha_1,\cdots,\alpha_{n-1})$ in the following way: set $y_n:=s_n$. If $\alpha_{n-1}=2$ then set $y_{n-1}:=s_{n-1}y_ns_{n-1}$. Otherwise, proceed as in the proof of point \textrm{(i)}, that is, juxtapose $s_{n-1}$ to the left (respectively, to the right) of $y_n$ if $s_{n-1} \longrightarrow s_n$ (respectively, $s_{n-1} \longleftarrow s_n$) and set $y_{n-1}:= s_{n-1}y_n$ (respectively, $y_n s_{n-1}$). Repeat the same process with $y_{n-1}$ and $\alpha_{n-2}$, and so on. The process ends when we get $y_1$. In fact, $r(x)=x_1$ is a subexpression of $y_1$ by construction. Next, we show that $y_1$ is a reduced expression. Observe that if $y_j$ is reduced then $y_js_{j-1}>y_j$ and $s_{j-1}y_j>y_j$, since there is no occurence of $s_{j-1}$ in $y_j$. Now, we proceed by contradiction to prove that $s_{j-1}y_js_{j-1}$ is reduced. Suppose $\ell(s_{j-1}y_js_{j-1}) < \ell(y_j)+2$, \textrm{i.e.}, $\ell(s_{j-1}y_js_{j-1}) \leq \ell(y_j)$. By applying Lemma \ref{lif lemma} with $x=s_{j-1}y_js_{j-1}$ and $w=y_js_{j-1}$ we get $s_{j-1}y_j = y_j s_{j-1}$, and $s_{j-1}y_j,\, y_j s_{j-1}$ are both reduced. Hence, \cite[Lemma 3.1]{m-bek} implies that $s_{j-1}$ commutes with each generator in $y_j$, which is absurd, since $s_j \leq y_j$ by construction. Therefore $\ell(s_{j-1}y_js_{j-1})=\ell(y_j)+2$. We conclude that $y_1$ is a reduced expression by induction on $n-i$, with $i=0 \cdots n-1$.\\
Denote by $y \in  W(A_n)$ the element that admits $y_1$ as a reduced expression. Then $y$ has the desired properties.\\
For the injectivity we proceed by contradiction. Suppose that $u,v \in [x,w]$ are such that $\phi(u)=\phi(v)=(\alpha_1,\cdots,\alpha_{n-1})$, with $u \neq v$. Denote by $r(u)$ (respectively, $r(v)$) the reduced expression of $u$ (respectively, $v$) obtained by applying the algorithm described above. Then $u \neq v$ implies $r(u) \neq r(v)$, that is, there exists an index $i \in [n-1]$ such that $\alpha_i=1$ and the position of the factor $s_{i}$ in $r(u)$ and $r(v)$ is different. Denote by $j$ be the minimum among these indices. Therefore, for every $h < j$ such that $\alpha_{h}=1$, $s_{h}$ appears on the same side in $r(u)$ as $r(v)$. Suppose that $\alpha_{j+1}=1$ and, for instance, that $$ r(u) = y_1 s_{j} \widehat{s_{j+1}} \cdots s_n \cdots s_{j+1}\widehat{s_{j}} y_2,$$ where $y_1 \leq s_1s_2 \cdots s_{j-1}$ and $y_2 \leq s_{j-1} s_{j-2} \cdots s_1$. Then $$r(v) = y_1 \widehat{s_{j}} s_{j+1} \cdots s_n \cdots \widehat{s_{j+1}} s_{j} y_2 \, \mbox{ or } \, r(v) = y_1 \widehat{s_{j}} \widehat{s_{j+1}} \cdots s_n \cdots s_{j+1} s_{j} y_2.$$ In both cases, $r(v)$ is a reduced expression such that $s_{j}$ is on the right of  $s_{j+1}$. On the other hand, $r(u)$ is a reduced expression of $u$ such that $s_{j}$ is on the left of $s_{j+1}$. Hence, Theorem \ref{sub prop} implies that $x$ admits a reduced expression that is a subexpression of $r(u)$ and a (possibly different) reduced expression that is a subexpression of $r(v)$. This is a contradiction, since $x$ is uniquely determined by the relations $s_i \longrightarrow s_{i+1}$ or $s_i \longleftarrow s_{i+1}$. The same conclusion holds if we consider different deletions of $s_j$ and $s_{j+1}$. In the case $\alpha_{j+1}=2$, we may assume that $$ r(u) = y_1 s_{j} s_{j+1} \cdots s_n \cdots s_{j+1}\widehat{s_{j}} y_2 \mbox{ and } r(v) = y_1 \widehat{s_{j}} s_{j+1} \cdots s_n \cdots s_{j+1} s_{j} y_2.$$ Observe that $r(v)$ (respectively, $r(u)$) is a reduced expression such that $s_{j}$ is on the right (respectively, on the left) of  $s_{j+1}$. Therefore, we reach the same contradiction that we obtained in the previous case.
               
\textrm{(iii)} Let $y \in (x,w]$ and $\phi(y)=(\alpha_1,\cdots,\alpha_{n-1})$. Let $j$ be the maximum of the  $i \in [n-1]$ such that $\alpha_i=2$. Then $r(y)$ contains the braid $s_js_{j+1}s_j$, so $y \not \in W_c(X)$.
\end{proof}

\begin{cor}\label{boolint}
Let $s_1 s_2 \cdots s_n \cdots s_2 s_1$ be a reduced expression for $w \in W(A_n)$ and let $x \in W(A_n)$ be a Coxeter element. Then $D_{x,w}=\varepsilon_x \varepsilon_w$.
\end{cor}
\begin{proof}
If $n=1$ then $x=w$ and the statement follows trivially. Suppose $n>1$. If $u,v \in W(A_n)$ are such that $[u,v] \simeq B_{\ell(u,v)}$, then $P_{u,v}=1$ (see \cite[Corollary 4.12]{bre-cpkl}). Therefore, Theorem \ref{l-bool} implies that $P_{u,v}=1$, for all $u,v \in [x,w]$. Hence, from Theorem \ref{te1} and Theorem \ref{l-bool} we achieve $$D_{x,w} = \sum \left((-1)^k \prod_{i=1}^{k}P_{x_{i-1},x_i}\right)= \sum (-1)^k,$$ where the sum runs over all the chains $x=x_0<x_1<\cdots <x_k=w$ such that $1\leq k \leq \ell(x,w)$. Therefore, $D_{x,w}$ equals the alternating sum $$\sum_{k=1}^{\ell(x,w)}(-1)^k c_k,$$ where $c_k$ denotes the number of chains $x=x_0<x_1<\cdots <x_k=w$. By \cite[Proposition 3.8.5]{sta-ec1}, we get $$\sum_{k=1}^{\ell(x,w)}(-1)^kc_k=\mu(x,w),$$ where $\mu$ denotes the M\"{o}bius function on the poset induced by the Bruhat order on $[x,w]$. On the other hand, if $\mathcal{P}=(\mathcal{P}, \leq)$ is a boolean poset and $U,V \in \mathcal{P}$ are such that $U \leq V$, then $\mu_{\mathcal{P}}(U,V)=(-1)^{|V-U|}$, where $|V-U|$ denotes the length of the interval $[U,V]$ (see, \textrm{e.g.}, \cite[Example 3.8.3]{sta-ec1}). Finally, observe that the length of a Bruhat interval $[x,w]$ is $\ell(x,w)$ (see, \textrm{e.g.}, \cite[Theorem 2.2.6]{bb-ccg}) and the statement follows.
\end{proof}

For instance, in Example \ref{ese dpol} we showed that $D_{x,w}=\varepsilon_x \varepsilon_w=1$, for every $x \in \{s_1s_2s_3,\, s_1s_3s_2,\,s_2s_1s_3,\,s_3s_2s_1\}$, with $w=s_1s_2s_3s_2s_1$.

\newpage

\section{Combinatorial properties of $a_{x,w}$}\label{sec: a}
In this section we study the family of polynomials $\{a_{x,w}\}_{x,w \in W_c(X)} \subseteq \Z[q]$, which express the involution $\iota$ in terms of the $t$--basis (see Proposition \ref{a-pol}). More precisely, we obtain a recurrence relation for $a_{x,w}$, $X$ being an arbitrary Coxeter graph. Then we will focus on the Coxeter graph satisfying equation (\ref{Cw-0}) and derive some results concerning symmetry properties, and the value of the constant term of $a_{x,w}$. 

\begin{pro} \label{proarec}
Let $X$ be an arbitrary Coxeter graph. Let $w \in W_c(X)$ and $s \in S(X)$ be such that $w>ws \in W_c(X)$. Then, for all $x \in W_c(X),\,x \leq w$, we have 
$$a_{x,w}=\widetilde{a_{x,w}} + \sum_{\substack{y \in W_c(X),\, ys \not \in W_c(X) \\ ys>y}}D_{x,ys}a_{y,ws},$$
where
\[
\widetilde{a_{x,w}}\=\left\{
\begin{array}{ll}
a_{xs,ws}                        & \mbox{ if } x>xs,           \\
qa_{xs,ws}+(1-q)a_{x,ws}         & \mbox{ if } x<xs \in W_c(X),   \\
(1-q)a_{x,ws}                    & \mbox{ if } x<xs \not \in W_c(X).
\end{array} \right.
\] 
\end{pro}
\begin{proof}
On the one hand, by Proposition \ref{a-pol}, we have 
\[ (t_{w^{-1}})^{-1}=q^{-\l(w)}\sum_{\substack{y \in W_c(X) \\ y\leq w}}a_{y,w} t_y. \]
On the other hand, letting $v \stackrel{\rm def}{=} ws$, we get
\begin{eqnarray}
(t_{w^{-1}})^{-1} & = & (t_{v^{-1}})^{-1}(t_s)^{-1} \nonumber \\
                & = & q^{-\ell(v)}\sum_{\substack{y \in W_c(X) \\ y\leq v}}a_{y,v}t_y \cdot q^{-1}(t_s-(q-1)t_e)    \nonumber \\
                & = & q^{-\ell(w)} \left(\sum_{\substack{y \in W_c(X) \\ y\leq v}}a_{y,v}t_yt_s -(q-1)\sum_{\substack{y \in W_c(X) \\ y\leq v}}a_{y,v}t_y  \right)   \nonumber \\
                & = & q^{-\ell(w)} \left(\sum_{\substack{y \in W_c(X),\, ys \in W_c(X) \\ y\leq v,\,ys>y}} a_{y,v} t_{ys} + \sum_{\substack{y \in W_c(X),\, ys \not \in W_c(X) \\ y\leq v,\,ys>y}} a_{y,v} t_{ys} \right)  \nonumber \\
                &   & +\: q^{-\ell(w)} \left(\sum_{\substack{y \in W_c(X) \\ y\leq v,\,ys<y}} a_{y,v} t_{ys}-(q-1)\sum_{\substack{y \in W_c(X) \\ y\leq v}}a_{y,v}t_y \right) \nonumber \\
                & = & q^{-\ell(w)} \left(\sum_{\substack{y \in W_c(X),\, ys \in W_c(X) \\ y\leq v,\,ys>y}} a_{y,v} t_{ys} + \sum_{\substack{y \in W_c(X),\, ys \not \in W_c(X) \\ y\leq v,\,ys>y}} a_{y,v} \left(\sum_{\substack{z \in W_c(X) \\ z<sy}}D_{z,ys}t_z \right) \right) \nonumber \\
                &   & +\: q^{-\ell(w)} \left(\sum_{\substack{y \in W_c(X) \\ y\leq v,\,ys<y}} a_{y,v} (qt_{ys}+(q-1)t_y)-(q-1)\sum_{\substack{y \in W_c(X) \\ y\leq v}}a_{y,v}t_y \right) \nonumber \\                 
                & = & q^{-\ell(w)} \left(\sum_{\substack{z \in W_c(X) \\ z\leq v,\,z>zs}} a_{zs,v} t_{z} + \sum_{\substack{z \in W_c(X) \\ z<vs}} \left(\sum_{\substack{y \in W_c(X),\, ys \not \in W_c(X) \\ y \leq v,\,ys>y}} D_{z,ys}a_{y,v} \right)t_z \right)\nonumber \\
                &   & +\: q^{-\ell(w)} \left(\sum_{\substack{y \in W_c(X) \\ y\leq v,\,ys<y}} a_{y,v} qt_{ys}+ (q-1)\sum_{\substack{y \in W_c(X) \\ y\leq v,\,ys<y}}  a_{y,v} t_{y}-(q-1)\sum_{\substack{y \in W_c(X) \\ y\leq v}}a_{y,v}t_y \right). \nonumber
\end{eqnarray}
Observe that $$(q-1)\sum_{\substack{y \in W_c(X) \\ y\leq v,\,ys<y}}  a_{y,v} t_{y}-(q-1)\sum_{\substack{y \in W_c(X) \\ y\leq v}}a_{y,v}t_y=(1-q)\sum_{\substack{y \in W_c(X) \\ y\leq v,\,y<ys}} a_{y,v} t_{y}.$$ To sum up,
\begin{eqnarray}
(t_{w^{-1}})^{-1} & = & q^{-\ell(w)} \left(\sum_{\substack{x \in W_c(X) \\ x\leq v,\,x>xs}} a_{xs,v} t_{x} + \sum_{\substack{x \in W_c(X) \\ x<vs}} \left(\sum_{\substack{y \in W_c(X),\, ys \not \in W_c(X) \\ y \leq v,\,ys>y}} D_{x,ys}a_{y,v} \right)t_x \right)\nonumber \\
                  &   & +\: q^{-\ell(w)} \left(q \sum_{\substack{xs \in W_c(X) \\ xs \leq v,\,x<xs}} a_{xs,v} t_{x} + (1-q)\sum_{\substack{x \in W_c(X) \\ x \leq v,\,x<xs}} a_{x,v} t_{x}  \right), \nonumber
\end{eqnarray}
and the statement follows by extracting the coefficient of $t_x$.
\end{proof}
From now on, we assume $X$ to be such that equation (\ref{Cw-0}) holds.
\begin{cor} \label{aqa}
Let $x,w \in W_c(X)$. If there exists $s \in S(X)$ such that $ws<w$ and $x<xs \not \in W_c(X)$, then $$a_{x,w}=-qa_{x,ws}.$$
\end{cor}
\begin{proof}
By applying Proposition \ref{proarec} and Lemma \ref{ddelta}, we have 
\begin{eqnarray}
a_{x,w} &=& (1-q)a_{x,ws}+\sum_{\substack{y \in W_c(X),\, ys \not \in W_c(X) \\ ys>y}}D_{x,ys}a_{y,ws} \nonumber \\
        &=& (1-q)a_{x,ws}+\sum_{\substack{y \in W_c(X),\, ys \not \in W_c(X) \\ ys>y}}(-\delta_{x,y})a_{y,ws} \nonumber \\
        &=& (1-q)a_{x,ws}-a_{x,ws},  \nonumber
\end{eqnarray}
and the statement follows.
\end{proof}
Next we obtain, using the results in \S \ref{s-dpol}, a non--recursive formula for polynomials $\{a_{x,w}\}$, an expression for their constant term, and symmetry properties.   
\begin{pro}\label{c-apol}
Let $x,w \in W_c(X)$ be such that $x \leq w$. Then
\[ a_{x,w}=\varepsilon_x \varepsilon_w R_{x,w}+\sum_{\substack{y \not \in W_c(X) \\ x < y < w}} \varepsilon_y \varepsilon_w R_{y,w} \left(\sum (-1)^k\prod_{i=1}^{k}P_{x_{i-1},x_i} \right), \]
where the second sum runs over all the chains $x=x_0< \cdots <x_k=y$ such that $x_i \not \in W_c(X)$ if $i>0$.
\end{pro}
\begin{proof}
From Proposition \ref{a-pol} we get 
\begin{equation}\label{apolforma}
a_{x,w}=\sum_{x\leq y \leq w} \varepsilon_y \varepsilon_w R_{y,w}D_{x,y},
\end{equation}
for all $x,w \in W_c(X)$ such that $x \leq w$. Since  $D_{x,w}=\delta_{x,w}$ if $x,w \in W_c(X)$, we have
\begin{equation} \label{apolform}
a_{x,w}=\varepsilon_x \varepsilon_w R_{x,w}+\sum_{\substack{y \not \in W_c(X) \\ x < y < w}} \varepsilon_y \varepsilon_w R_{y,w}D_{x,y},
\end{equation}
so the statement follows immediately from Theorem \ref{te1}. 
\end{proof}

The recursion given in Corollary \ref{aqa} can sometimes be solved explicitly.
\begin{pro} \label{explapol}
Let $w=s_is_{i+1}\cdots s_{i+k}s_{i-j} s_{i-j+1}\cdots s_i \cdots s_{i+k-1} \in W(A_n)$ and let $x=s_is_{i+1}\cdots s_{i+k} \in W(A_n)$, with $i \in [2,n],\,k \in [1,n-i],\,j \in [1,i-1]$. Then $$a_{x,w}=(-q)^k(1-q)^j.$$
\end{pro}
\begin{proof}
Note that the one--line notation for $w$, as a permutation in $S_{n+1}$, is $$\underline{w}=[w_1,w_2, \cdots, w_{i+k-1}, i-j, i, w_{i+k+2}, \cdots, w_{n+1}],$$ with $w_1<w_2<\cdots<w_{i+k-1}<w_{i+k+2}<\cdots<w_{n+1}$. Hence, $\underline{w}$ avoids the pattern $321$, that is $w \in W_c(A_n)$ (see Proposition \ref{321wc}). Moreover $inv(\underline{w})=2k+j+1$. On the other hand, $\ell(w)=inv(\underline{w})$ (see Proposition \ref{linv}), so $w$ is reduced.\\    
Observe that $x<xs_{i+h} \not \in W_c(X)$, for every $h \in [0,k-1]$ and that $ws_{i+k-1}<w$. By applying Corollary \ref{aqa} to the triple $(x,w,s_{i+k-1})$ we get $a_{x,w}=-qa_{x,ws_{i+k-1}}$. Repeat the same process with the triple $(x,ws_{i+k-1},s_{i+k-2})$, and so on. After $k$ iteration of the process we get $a_{x,w}(q)=(-q)^k a_{x,ws_{i+k-1} \cdots s_{i}}=(-q)^k a_{x,w'}(q)$, where we set $$w'=s_i s_{i+1} \cdots s_{i+k} s_{i-j} s_{i-j+1} \cdots s_{i-1}.$$ To conclude, we show that $a_{x,w'}=(1-q)^j$. Observe that $[x,w'] \simeq B_{\ell(w')-\ell(x)}$, so $R_{x,w'}=(q-1)^{\ell(w')-\ell(x)}$ (see \cite[Corollary 4.10]{bre-cpkl}). On the other hand, Proposition \ref{c-apol} implies  $a_{x,w'}=\varepsilon_x \varepsilon_{w'} R_{x,w'}$, since $\{y \in [x,w']:\,y \not \in W_c(X)\}=\emptyset$. Therefore $a_{x,w'}=\varepsilon_x \varepsilon_{w'}(q-1)^{\ell(w')-\ell(x)}=(1-q)^j$, as desired.  
\end{proof}

\begin{ese}
Let $w=s_6s_7s_8s_9s_4s_5s_6s_7s_8 \in W(A_{10})$ and $x=s_6s_7s_8s_9 \in W(A_{10})$. Then $n=10, i=6, k=3, j=2$ and  $\underline{w}=[1,2,3,5,7,8,9,10,4,6,11]$. By Proposition \ref{explapol}, $a_{x,w}=-q^3(q^2-2q+1)$. 
\end{ese}

Proposition \ref{c-apol} allows us to compute the constant term of the polynomials $\{a_{x,w}\}_{x,w \in W_c(X)}$.
\begin{cor} \label{c-apol1}
For all $x,w \in W_c(X)$ such that $x < w$ we have   
\begin{itemize}
\item[{\rm (i)}]  
$a_{x,w}(1)=0$;  
\item[{\rm (ii)}] 
$a_{x,w}(0)=\sum (-1)^k$,
\end{itemize} 
where the sum runs over all the chains $x=x_0<x_1<\cdots <x_{k+1}=w$ such that $x_i \not \in W_c(X)$ if $1 \leq i \leq k$, and $0\leq k \leq \l(x,w)-1$.
\end{cor}
\begin{proof}
The statement follows from $(\ref{apolform})$ by applying Corollary \ref{val0} and Corollary \ref{c-d0}.
\end{proof}
Again, we deduce from Proposition \ref{c-apol} the following symmetry properties of the polynomials $\{a_{x,w}\}_{x,w \in W_c(X)}$.

\begin{cor}
Let $x,w \in W_c(X)$. Then
\begin{itemize}
\item[{\rm (i)}]  
$a_{x,w}=a_{x^{-1},w^{-1}}$;  
\item[{\rm (ii)}] 
$a_{x,w}=a_{w_0xw_0,w_0ww_0}$.
\end{itemize}
\end{cor}
\begin{proof}
By Lemma \ref{maps} and by (\ref{apolforma}) we get
\begin{eqnarray}
a_{x^{-1},w^{-1}}    & = & \sum_{x^{-1} \leq y \leq w^{-1}} \varepsilon_y \varepsilon_{w^{-1}} R_{y,w^{-1}} D_{x^{-1},y} \nonumber \\
                     & = & \sum_{x^{-1} \leq z^{-1} \leq w^{-1}} \varepsilon_{z^{-1}} \varepsilon_{w^{-1}} R_{z^{-1},w^{-1}} D_{x^{-1},z^{-1}} \nonumber \\     
                     & = & \sum_{x\leq z \leq w} \varepsilon_z \varepsilon_w R_{z,w} D_{x,z} \nonumber \\
                     & = & a_{x,w}, \nonumber
\end{eqnarray}
where we used Corollary \ref{c2-dpol} \textrm{(i)} and the property $R_{x,w}=R_{x^{-1},w^{-1}}$, for all $x,w \in W(X)$ (see, \textrm{e.g.}, \cite[\S 5, Exercise 10(a)]{bb-ccg}). The same holds for $a_{w_0xw_0,w_0ww_0}$, using Corollary \ref{c2-dpol} \textrm{(ii)} and \cite[\S 5, Exercise 10(b)]{bb-ccg}.
\end{proof}

\begin{cor} \label{leadinga}
Let $x,w \in W_c(X)$ and $x \leq w$. Then $a_{x,w}$ has degree $\l(x,w)$ and leading term $\varepsilon_x \varepsilon_w$.
\end{cor}
\begin{proof}
The statement follows from Proposition \ref{c-apol}, combining Proposition \ref{p-upb} with Corollary \ref{val0}.
\end{proof}

Next, we obtain a property for polynomials $\{a_{x,w}\}$ that will be required in Section \ref{sec: l}.
\begin{pro} \label{apollo}
Let $w \in W_c(X)$. Then 
$$\sum_{\substack{x \in W_c(X) \\ x \leq w}} \varepsilon_x \varepsilon_w a_{x,w}=q^{\l(w)}.$$
\end{pro}
\begin{proof}
By combining equation (\ref{apolform}) with Proposition \ref{pro dsgn} we get 
\begin{eqnarray}
\sum_{\substack{x \in W_c(X) \\ x \leq w}} \varepsilon_x \varepsilon_w a_{x,w}  
      &=& \sum_{\substack{x \in W_c(X) \\ x \leq w}} \varepsilon_x \varepsilon_w \left( \varepsilon_x \varepsilon_w R_{x,w}+\sum_{\substack{y \not \in W_c(X) \\ x< y <w}} \varepsilon_y \varepsilon_w R_{y,w} D_{x,y}  \right)   \nonumber \\
      &=& \sum_{\substack{x \in W_c(X) \\ x \leq w}} R_{x,w}  + \sum_{\substack{x \in W_c(X) \\ x \leq w}} \varepsilon_x \left( \sum_{\substack{y \not \in W_c(X) \\ x< y <w}} \varepsilon_y  R_{y,w} D_{x,y} \right)   \nonumber \\
      &=& \sum_{\substack{x \in W_c(X) \\ x \leq w}} R_{x,w}  + \sum_{\substack{y \not \in W_c(X) \\ y \leq w}} \varepsilon_y R_{y,w} \left( \sum_{\substack{x \in W_c(X) \\ x \leq y}} \varepsilon_x  D_{x,y}  \right)   \nonumber \\
      &=& \sum_{\substack{x \in W_c(X) \\ x \leq w}} R_{x,w}  + \sum_{\substack{y \not \in W_c(X) \\ y \leq w}} \varepsilon_y R_{y,w} \varepsilon_y   \nonumber \\
      &=& \sum_{x \leq w} R_{x,w} \nonumber
\end{eqnarray}
and the statement follows from Proposition \ref{rpol prop}.
\end{proof}
We refer to $\{y \in [x,w] :\, \l(y)=\l(x)+1\}$ as the set of \emph{atoms} in $[x,w]$ and denote by $a(x,w)$ the number of atoms in $[x,w]$. 
\begin{cor}
Let $w \in W_c(X)$ be such that $\l(w)>1$. Then, $$[q^{\l(w)-1}]a_{e,w}=-\varepsilon_w a(e,w).$$
\end{cor}
\begin{proof}
By Proposition \ref{apollo}, we get $\varepsilon_w a_{e,w}=q^{\l(w)}-\sum_{e < y \leq w} \varepsilon_y \varepsilon_w a_{y,w}$. Therefore, by Corollary \ref{leadinga}, we achieve 
$$[q^{\l(w)-1}]a_{e,w}=-\varepsilon_w \sum_{\substack{y \in (e,w] \\ \l(y)=1}} 1=-\varepsilon_w a(e,w).$$ 
\end{proof}

\begin{lemma} \label{apol2}
Let $x,w \in W_c(X)$ be such that $x < w$. If $\ell(x,w)=1$ then $a_{x,w}=1-q$. If $\ell(x,w)=2$, then
\[
a_{x,w}=\left\{ \begin{array}{ll}
  q^2-1       & \mbox{ if $k=2$,} \\ 
  q^2-q       & \mbox{ if $k=1$,} \\
  q^2-2q+1    & \mbox{ if $k=0$,}
\end{array} \right.
\]
with $k = |\{y \not \in W_c(X): \, x < y < w \}|$.
\end{lemma}
\begin{proof}
By Lemma \ref{rp2}, $R_{x,w}=(q-1)^{\ell(x,w)}$ for all $x,w \in W(X),\, x \leq w$ such that $\ell(x,w) \leq 2$. If $\ell(x,w)=1$, then Equation (\ref{apolform}) implies $a_{x,w}=\varepsilon_x \varepsilon_w R_{x,w}=1-q$. If $\ell(x,w)=2$, then the interval $[x,w]$ is isomorphic to the boolean lattice $B_2$ (see, \textrm{e.g.}, \cite[Lemma 2.7.3]{bb-ccg}). Moreover $R_{y,w}=q-1$ and $D_{x,y}=-1$ for all $y \not \in W_c(X)$ such that $x < y <w$, since $\ell(x,y)=\ell(y,w)=1$ (see Lemma \ref{dpol2}). From Equation (\ref{apolform}) we get
\begin{eqnarray}
a_{x,w}    & = &\varepsilon_x \varepsilon_w R_{x,w}+\sum_{\substack{y \not \in W_c(X) \\ x < y < w}} \varepsilon_y \varepsilon_w R_{y,w} D_{x,y} \nonumber \\
           & = & (q-1)^2+ \sum_{\substack{y \not \in W_c(X) \\ x < y < w}} (q-1) \nonumber \\
           & = & (q-1)^2+k(q-1), \nonumber 
\end{eqnarray}
and the statement follows. 
\end{proof}

\section{Combinatorial properties of $L_{x,w}$}\label{sec: l}
In this section we study the polynomials $\{L_{x,w}\}_{x,w \in W_c(X)}$ which play the same role, in $TL(X)$, as the Kazhdan--Lusztig polynomials play in $\H(X)$. First, we derive a recursive formula for $L_{x,w}$ by means of some results in \cite{gre-gjt}. Then, using the results in Section \ref{s-dpol}, we obtain a non--recursive formula, symmetry properties, expressions for the constant term, and bounds for the degrees, for these polynomials. All the results stated in this section hold for every Coxeter graph $X$ satisfying (\ref{Cw-0}).

It is known that the terms of maximum possible degree in the $L$--polynomials and in the Kazhdan--Lusztig polynomials coincide (see \cite[Theorem 5.13]{gre-gjt}) . 
\begin{pro} \label{muu}
For $x,w \in W_c(X)$ let $M(x,w)$ be the coefficient of $q^{-\frac{1}{2}}$ in $L_{x,w}$ and let $\mu(x,w)$ be the coefficient of $q^{\frac{\l(w)-\l(x)-1}{2}}$ in $P_{x,w}$. Then $M(x,w) = \mu(x,w)$. 
\end{pro}
The product of two IC basis elements can be computed by means of the following formula (see \cite[Theorem 5.13]{gre-gjt}). Recall that if $\mu(x,w) \neq 0$ then we write $x \prec w$ (see Definition \ref{topcoeff}).
\begin{pro}\label{icprod}
Let $s \in S(X)$ and $w \in W_c(X)$. Then 
\[
c_sc_w= 
\begin{cases}
c_{sw}+\sum_{\substack{x \prec w  \\ sx<x}}\mu(x,w)c_x    & \mbox{if } \l(sw)>\l(w);  \\
(q^{\frac{1}{2}}+q^{-\frac{1}{2}})c_w                     & \mbox{otherwise},
\end{cases}   \nonumber
\]
where $c_x \= 0$ for every $x \not \in W_c(X)$.
\end{pro}

\begin{cor} \label{txc} 
Let $s \in S(X)$ and $w \in W_c(X)$. Then 
\[
t_sc_w= 
\begin{cases}
-c_w+q^{\frac{1}{2}}\left(c_{sw} + \sum_{\substack{x \prec w  \\ sx<x}} \mu(x,w) c_x \right)    & \mbox{if } \l(sw)>\l(w);  \\
qc_w                                                                                            & \mbox{otherwise}.
\end{cases}   \nonumber
\]
\end{cor}

\begin{proof}
Observe that $t_s=q^{\frac{1}{2}}c_s-c_e$. So $t_sc_w=q^{\frac{1}{2}}c_sc_w-c_w$ and the statement follows by applying Proposition \ref{icprod}.
\end{proof}

\begin{teo}
Let $x,w \in W_c(X)$ be such that $sx \in W_c(X)$ and $sw<w$. Then
\begin{eqnarray}
L_{x,w}                   & = & L_{sx,sw}+q^{c-\frac{1}{2}}L_{x,sw}-\sum_{\substack{sz<z \\ z \in [sx,sw]_c}} \mu(z,sw)L_{x,z} \nonumber \\
                          &   & +\: q^{-\frac{1}{2}}\sum_{\substack{sz \not \in W_c(X) \\z \in [x,w]_c}} q^{\frac{\l(x)-\l(z)}{2}}D_{x,sz}L_{z,sw}, \nonumber
\end{eqnarray}
where $c=1$ if $sx<x$ and $0$ otherwise. 
\end{teo}
\begin{proof}
Let $w=sv$. By Proposition \ref{icprod}, we have
\begin{equation} \label{equazio}
c_w=c_{sv}=c_sc_v-\sum_{sz<z} \mu(z,sw)c_z.
\end{equation}
Recall that $c_s=q^{-\frac{1}{2}}(t_s+t_e)$, hence we get
\begin{eqnarray}
c_sc_v & = & q^{-\frac{1}{2}} c_v + q^{-\frac{1}{2}} t_s c_v  \nonumber \\
       & = & q^{-\frac{1}{2}} c_v + \sum_{\substack{x \in W_c(X) \\ x \leq sw}}q^{-\frac{\l(x)}{2}}L_{x,sw} t_s t_x  \nonumber \\
       & = & q^{-\frac{1}{2}} \left(c_v + \sum_{\substack{sx \in W_c(X) \\ x<sx}} q^{-\frac{\l(x)}{2}}L_{x,sw} t_{sx}+\sum_{sx<x} q^{-\frac{\l(x)}{2}}L_{x,sw} (qt_{sx}+(q-1)t_x) \right) \nonumber\\
       &   & +\: q^{-\frac{1}{2}}\left( \sum_{\substack{sx \not \in W_c(X) \\ x<sx}} q^{-\frac{\l(x)}{2}}L_{x,sw} \left(\sum_{\substack{y \in W_c \\ y<sx}} D_{y,sx} t_y  \right)    \right) \nonumber \\
       & = & q^{-\frac{1}{2}} \left(c_v + \sum_{\substack{sx \in W_c(X) \\ x<sx}} q^{-\frac{\l(x)}{2}}L_{x,sw} t_{sx}+\sum_{sx<x} q^{-\frac{\l(x)}{2}}L_{x,sw} (qt_{sx}+(q-1)t_x)\right) \nonumber\\
       &   & +\: q^{-\frac{1}{2}} \left( \sum_{\substack{y \in W_c(X) \\ y \leq w}} \left(\sum_{\substack{sx \not \in W_c(X) \\ x<sx}} q^{-\frac{\l(x)}{2}}D_{y,sx} L_{x,sw} \right) t_y    \right). \nonumber     
\end{eqnarray} 
Suppose that $su>u$ and extract the coefficient of $t_{su}$ on both sides of (\ref{equazio}). It follows that 
$$L_{su,w}=L_{u,sw}+q^{\frac{1}{2}}L_{su,sw}+\sum_{\substack{sz \not \in W_c(X) \\ z<sz}} q^{\frac{\l(u)-\l(z)}{2}}D_{su,sz}L_{z,sw}-\sum_{\substack{z \in [u,w]_c \\ sz<z}}\mu(z,sw) L_{su,z}.$$
Otherwise, if $su<u$ then
$$L_{su,w}=L_{u,sw}+q^{-\frac{1}{2}}L_{su,sw}+q^{-1}\sum_{\substack{sz \not \in W_c(X) \\ z<sz}} q^{\frac{\l(u)-\l(z)}{2}}D_{su,sz}L_{z,sw}-\sum_{\substack{z \in [u,w]_c \\ sz<z}}\mu(z,sw) L_{su,z},$$
and the statement follows by applying the substitution $x=su$.
\end{proof}

\begin{teo} \label{lrecurs}
Let $x,w \in W_c(X)$ be such that $x<w$. If there exists $s \in S(X)$ such that $sw<w$ and $x<sx \in W_c(X)$, then 
$$L_{x,w}=q^{-\frac{1}{2}}L_{sx,w}-\sum_{\substack{sz \not \in W_c(X) \\ z \in (x,w)_c}}q^{\frac{\l(x)-\l(z)}{2}} D_{sx,sz} L_{z,w}.$$  
\end{teo}
\begin{proof}
By Corollary \ref{txc} we get $t_sc_w=qc_w$, since $\l(sw)<\l(w)$ by hypothesis. Furthermore, if $x<sx \in W_c(X)$ then $[t_{sx}](qc_w)=q\cdot q^{-\frac{\l(xs)}{2}}L_{sx,w}$. On the other hand, by Theorem \ref{defic}, we get
\begin{eqnarray}
t_sc_w & = & \sum_{\substack{x \in W_c(X) \\ x \leq w}}q^{-\frac{\l(x)}{2}}L_{x,w} t_s t_x              \nonumber \\
       & = & \sum_{\substack{sx \in W_c(X) \\ sx>x}}q^{-\frac{\l(x)}{2}}L_{x,w} t_{sx} + \sum_{\substack{sx \not \in W_c(X) \\ sx>x}}q^{-\frac{\l(x)}{2}}L_{x,w} t_{sx}+               \nonumber \\
       &   & +\: \sum_{\substack{x \in W_c(X) \\ sx<x}}q^{-\frac{\l(x)}{2}}L_{x,w} (qt_{sx} + (q-1) t_x)  \nonumber \\       
       & = & \sum_{\substack{sx \in W_c(X) \\ sx>x}}q^{-\frac{\l(x)}{2}}L_{x,w} t_{sx} + \sum_{\substack{sx \not \in W_c(X) \\ sx>x}}q^{-\frac{\l(x)}{2}}L_{x,w} \left( \sum_{\substack{y \in W_c(X) \\ y<sx}} D_{y,sx}t_y   \right) +               \nonumber \\
       &   & +\: q\sum_{\substack{sz \in W_c(X) \\ z<sz}}q^{-\frac{\l(sz)}{2}}L_{sz,w} t_z + (q-1)\sum_{\substack{sz \in W_c(X) \\ z<sz}}q^{-\frac{\l(sz)}{2}}L_{sz,w} t_{sz}  \nonumber \\
       & = & \sum_{\substack{sx \in W_c(X) \\ sx>x}}q^{-\frac{\l(x)}{2}}L_{x,w} t_{sx} +  q^{\frac{1}{2}} q^{-\frac{\l(x)}{2}}L_{sx,w} t_x + q^{\frac{1}{2}} q^{-\frac{\l(x)}{2}}L_{sx,w} t_{sx} +    \nonumber \\       
       &   & -\: q^{-\frac{1}{2}}q^{-\frac{\l(x)}{2}}L_{sx,w} t_{sx} + \sum_{\substack{x \in W_c(X) \\ x \leq w}} \left(   \sum_{\substack{sz \not \in W_c(X) \\ z \in (x,w)_c}} q^{-\frac{\l(z)}{2}}D_{x,sz} L_{z,w} \right)t_x.            \label{lxs}
\end{eqnarray} 
By extracting the coefficient of $t_{sx}$ in (\ref{lxs}) we obtain
\begin{eqnarray}
q^{\frac{1}{2}}q^{-\frac{\l(x)}{2}}L_{sx,w} & = & q^{-\frac{\l(x)}{2}}L_{x,w} + q^{\frac{1}{2}}q^{-\frac{\l(x)}{2}} L_{sx,w} - q^{-\frac{1}{2}} q^{-\frac{\l(x)}{2}}L_{sx,w} +          \nonumber \\
                                            &   & +\: \sum_{\substack{sz \not \in W_c(X) \\ z \in (x,w)_c}} q^{-\frac{\l(z)}{2}}D_{sx,sz} L_{z,w}.       \nonumber
\end{eqnarray}
and the statement follows.
\end{proof} 

The following result was inspired by a similar property for the Kazhdan--Lusztig polynomials (see, \textrm{e.g.}, \cite[\S5, Exercises 16]{bb-ccg}). 
\begin{pro} \label{prop-l}
Let $w \in W_c(X)$ and define $$F_w(q^{-\frac{1}{2}})\stackrel{\rm def}{=}\sum_{\substack{x \in W_c(X) \\ x \leq w}}\varepsilon_x q^{-\frac{\ell(x)}{2}}L_{x,w}(q^{-\frac{1}{2}}).$$  Then $F_w(q^{-\frac{1}{2}})=\delta_{e,w}$.
\end{pro}
\begin{proof}
The case $w=e$ is trivial. Suppose $w \not = e$. Combining Theorem \ref{l-pol} \textrm{(iv)} with Proposition \ref{apollo} we have
\begin{eqnarray}
F_w(q^{-\frac{1}{2}}) &=& \sum_{\substack{u \in W_c(X) \\ u \leq w}}\varepsilon_u q^{-\frac{\ell(u)}{2}} \left( \sum_{\substack{x \in W_c(X) \\ u \leq x \leq w}} q^{\frac{\ell(u)-\ell(x)}{2}} a_{u,x}(q) L_{x,w}(q^{\frac{1}{2}})  \right) \nonumber \\  
                      &=& \sum_{\substack{x \in W_c(X) \\ x \leq w}} \left( \sum_{\substack{u \in W_c(X) \\ u \leq x}} \varepsilon_u q^{-\frac{\ell(x)}{2}} a_{u,x}(q) L_{x,w}(q^{\frac{1}{2}})  \right) \nonumber \\
                      &=& \sum_{\substack{x \in W_c(X) \\ x \leq w}} \varepsilon_x q^{-\frac{\ell(x)}{2}} L_{x,w}(q^{\frac{1}{2}}) \left( \sum_{\substack{u \in W_c(X) \\ u \leq x}} \varepsilon_x \varepsilon_u a_{u,x}(q) \right) \nonumber \\
                      &=& \sum_{\substack{x \in W_c(X) \\ x \leq w}} \varepsilon_x q^{-\frac{\ell(x)}{2}} L_{x,w}(q^{\frac{1}{2}}) q^{\ell(x)} \nonumber \\
                      &=& \sum_{\substack{x \in W_c(X) \\ x \leq w}} \varepsilon_x q^{\frac{\ell(x)}{2}} L_{x,w}(q^{\frac{1}{2}})  \nonumber \\ 
                      &=& F_w(q^{\frac{1}{2}}). \nonumber 
\end{eqnarray}
This implies that $F_w(q^{-\frac{1}{2}})$ is constant. On the other hand, the constant term in $F_w(q^{-\frac{1}{2}})$ is zero since $L_{x,w} \in q^{-\frac{1}{2}}\Z[q^{-\frac{1}{2}}]$ by Theorem \ref{l-pol}, and the statement follows.
\end{proof}

The next result is a restatement of the well--known property $\mu(e,w)=0$, for every $w \in W(X)$ such that $\l(e,w)>1$ (see, \textrm{e.g.}, \cite[Proposition 5.1.9]{bb-ccg}).
\begin{cor} \label{coeffel}
Let $w \in W_c(X)$. Then 
\[
[q^{-\frac{1}{2}}]L_{e,w}=\left\{ \begin{array}{ll}
1   & \mbox{ if } \l(w)=1, \\
0   & \mbox{ if } \l(w) \neq 1.
\end{array} \right.
\]
\end{cor}

\begin{cor}
Let $w \in W_c(X)$. Then 
\[
[q^{-1}]L_{e,w}=\left\{ \begin{array}{ll}
0                                                                                     & \mbox{ if } \l(w)<2, \\
1                                                                                     & \mbox{ if } \l(w)=2, \\
\sum_{\substack{s \in S(X) \\ s \leq w}} \mu(s,w)                                     & \mbox{ if } \l(w)>2.
\end{array} \right.
\]
\end{cor}	
\begin{proof}
The case $\l(w) \leq 1$ is trivial. If $\l(w)=2$ then $L_{e,w}=q^{-1}$ as explained at the end of this section. Suppose $\l(w)>2$. Then 
\begin{eqnarray}
[q^{-1}](L_{e,w}) &=& -\sum_{\substack{x \in W_c(X) \\ x \leq w, \: \l(x)=1}} \varepsilon_x [q^{-\frac{1}{2}}]L_{x,w}-\sum_{\substack{x \in W_c(X) \\ x \leq w,  \: \l(x)=2}} \varepsilon_x \underbrace{[q^0]L_{x,w}}_{0} \nonumber \\
                  &=& \sum_{\substack{s \in S(X) \\ s \leq w}} [q^{-\frac{1}{2}}]L_{s,w}, \nonumber
\end{eqnarray}
an the statement follows by Proposition \ref{muu}.
\end{proof}
The previous results would suggest that $[q^k]L_{e,w} \geq 0$, for every $k \in \Q$. However, this is not true even in $W(A_3)$. For instance, computer calculations show that $L_{e,s_2s_1s_3s_2}=q^{-1}-q^{-2}$.
  
\begin{teo}\label{lpoch}
For all elements $x,w \in W_c(X)$ such that $x < w$ we have
$$L_{x,w}=q^{\frac{\l(x)-\l(w)}{2}}\sum \left((-1)^k \prod_{i=1}^{k+1}P_{x_{i-1},x_i}\right),$$
where the sum runs over all the chains $x=x_0<x_1<\cdots <x_{k+1}=w$ such that $x_i \not \in W_c(X)$ if $1\leq i \leq k$, and $ 0\leq k \leq \l(x,w)-1$.
\end{teo}
\begin{proof}
On the one hand, from Proposition \ref{d-pol} and the definition of the $t$--basis we get 
\begin{eqnarray}
\sigma(C_w') & = & q^{-\frac{\l(w)}{2}}\sum_{y \leq w}P_{y,w}\sigma(T_y) \nonumber \\
             & = & q^{-\frac{\l(w)}{2}}\sum_{y \leq w}P_{y,w} \left( \sum_{\substack{x \in W_c(X) \\ x \leq y}} D_{x,y}t_x \right) \nonumber \\
             & = & q^{-\frac{\l(w)}{2}}\sum_{\substack{x \in W_c(X) \\ x \leq w}}\left(\sum_{x \leq y \leq w} D_{x,y}P_{y,w}\right)t_x.\label{sigc} 
\end{eqnarray}
On the other hand, by Theorem \ref{defic}, 
\begin{equation} \label{ridefic}
c_w=\sum_{\substack{x \in W_c(X) \\ x \leq w}} q^{-\frac{\l(x)}{2}}L_{x,w}t_x.
\end{equation}
Therefore, equation (\ref{Cw-0}) implies that the coefficient of $t_x$ in (\ref{ridefic}) and in (\ref{sigc}) are equal, that is $$q^{-\frac{\l(x)}{2}} L_{x,w}=q^{-\frac{\l(w)}{2}}\sum_{x \leq y \leq w} D_{x,y}P_{y,w}.$$ 
Since $D_{x,y} = \delta_{x,y}$ if $y \in W_c(X)$ and $x \leq y$, we achieve
\begin{equation} \label{e-pol}
L_{x,w}=q^{\frac{\l(x)-\l(w)}{2}}\left(P_{x,w}+\sum_{\substack{y \not \in W_c(X) \\ x < y < w}} D_{x,y}P_{y,w}\right). 
\end{equation}
Combining (\ref{e-pol}) and Theorem \ref{te1} we get 
\begin{eqnarray}
L_{x,w} & = & q^{\frac{\l(x)-\l(w)}{2}} \left( P_{x,w} + \sum_{\substack{y \not \in W_c(X) \\ x < y < w}} \sum_{\substack{x=x_0<\cdots \\ \cdots<x_k=y}} \left((-1)^k \prod_{i=1}^{k}P_{x_{i-1},x_i}\right) P_{y,w} \right) \nonumber \\ 
        & = & q^{\frac{\l(x)-\l(w)}{2}} \left(\sum_{\substack{x=x_0<\cdots \\ \cdots <x_{k+1}=w}}(-1)^k \prod_{i=1}^{k+1}P_{x_{i-1},x_i} \right), \nonumber
\end{eqnarray} 
where $x_i \not \in W_c(X)$ if $1\leq i \leq k$.
\end{proof}
The previous theorem shows that the $L$--polynomials depend only on the Kazhdan--Lusztig polynomials and the poset structure induced by the Bruhat order on $\{x,w\} \cup ((x,w) \setminus(x,w)_c)$, where $(x,w)_c=\{y \in (x,w):\, y \in W_c(X) \}$.\\

\begin{lemma} \label{lzw0}
Let $x,w \in W_c(X)$. If there exists $s \in S(X)$ such that $sw<w$ and $x<sx \not \in W_c(X)$, then $L_{x,w}=0$.
\end{lemma}
\begin{proof}
By (\ref{e-pol}) we get 
\begin{eqnarray}
L_{x,w} &=& q^{\frac{\l(x)-\l(w)}{2}}\left(P_{x,w}+\sum_{\substack{y \not \in W_c(X) \\ x < y < w}} D_{x,y}P_{y,w}\right) \nonumber \\
        &=& q^{\frac{\l(x)-\l(w)}{2}}\left(P_{x,w}+D_{x,sx}P_{sx,w}+\sum_{\substack{y \not \in W_c(X), y \neq sx \\ x < y < w}} D_{x,y}P_{y,w}\right)\nonumber \\
        &=& q^{\frac{\l(x)-\l(w)}{2}} \left(\sum_{\substack{y \not \in W_c(X), y \neq sx \\ x < y < w}} D_{x,y}P_{y,w}\right). \label{star}
\end{eqnarray}
Denote by $(*)$ the expression in round brackets in (\ref{star}). We show that $(*)$ is zero and the statement follows. In fact, by applying  relation (\ref{dxzs0}) and Lemma \ref{ddelta}, we get
\begin{eqnarray}
    (*) &=& \sum_{\substack{y \not \in W_c(X) \\ y < sy }} D_{x,y}P_{y,w}+\sum_{\substack{y \not \in W_c(X), y \neq sx \\ y>sy}} D_{x,y}P_{y,w}\nonumber \\
        &=& \sum_{\substack{y \not \in W_c(X) \\ y < sy \not \in W_c(X)}} D_{x,y}P_{y,w} + \sum_{\substack{y \not \in W_c(X) \\ y>sy \not \in W_c(X)}} D_{x,y}P_{y,w} + \sum_{\substack{y \not \in W_c(X), y \neq sx \\ y>sy \in W_c(X)}} D_{x,y}P_{y,w}\nonumber \\
        &=& \sum_{\substack{y \not \in W_c(X) \\ y>sy \not \in W_c(X)}} (\underbrace{D_{x,sy}+D_{x,y}}_{0})P_{y,w}+\sum_{\substack{y \not \in W_c(X), y \neq sx \\ y>sy \in W_c(X)}} D_{x,y}P_{y,w} \nonumber \\
        &=& \sum_{\substack{y \not \in W_c(X), y \neq sx \\ y>sy \in W_c(X)}} (-\delta_{x,sy})P_{y,w}=0, \nonumber    
\end{eqnarray}
as desired.
\end{proof}

The next result mirrors a well--known property of the Kazhdan--Lusztig polynomials (see, \textrm{e.g.}, \cite[Proposition 5.1.8]{bb-ccg}).
\begin{cor}
Let $x,w \in W_c(X)$ be such that $x<w$. If there exists $s \in S(X)$ such that $sw<w$ and $x<sx \in W_c(X)$, then $L_{x,w}=q^{-\frac{1}{2}}L_{sx,w}$.
\end{cor}
\begin{proof}
The result follows by combining Theorem \ref{lrecurs} with Lemma \ref{lzw0}. 
\end{proof}

In the same way that Corollary \ref{c2-dpol} follows from Theorem \ref{te1} we deduce from Theorem \ref{lpoch} the following symmetry properties of the $L$--polynomials, whose proof we omit.  
\begin{cor}
Let $x,w \in W_c(X)$. Then   
\begin{itemize}
\item[{\rm (i)}]  
$L_{x,w}=L_{x^{-1},w^{-1}}$;  
\item[{\rm (ii)}] 
$L_{x,w}=L_{w_0xw_0,w_0ww_0}$.
\end{itemize}
\qed
\end{cor}

\begin{lemma}
Let $x,w \in W_c(X)$ be such that $x < w$. If $\ell(x,w)=1$ then $L_{x,w}=q^{-\frac{1}{2}}$. If $\ell(x,w)=2$, then

\[
L_{x,w}=\left\{ \begin{array}{rl}
-q^{-1}   & \mbox{ if $k=2$,} \\ 
 0        & \mbox{ if $k=1$,} \\
 q^{-1}   & \mbox{ if $k=0$,}
\end{array} \right.
\]
with $k = |\{ y \not \in W_c(X):\, x < y < w  \}|$.
\end{lemma}
\begin{proof}
If $\ell(x,w)=1$, then $P_{x,w}=1$ (see Lemma \ref{rp2}). Equation (\ref{e-pol}) then implies that $L_{x,w}=q^{-\frac{1}{2}}P_{x,w}=q^{-\frac{1}{2}}$. If $\ell(x,w)=2$, then from Equation (\ref{e-pol}) and Lemma \ref{dpol2} we get 
\begin{eqnarray}
L_{x,w} & = & q^{\frac{\ell(x)-\ell(w)}{2}}\left(P_{x,w}+\sum_{\substack{y \not \in W_c(X) \\ x < y < w}} D_{x,y}P_{y,w}\right) \nonumber \\
                          & = & q^{-1} \left(1 + \sum_{\substack{y \not \in W_c(X) \\ x < y < w}} (-1)  \right) \nonumber \\                          
                          & = & q^{-1} \left(1 - k \right), \nonumber 
\end{eqnarray}
and the statement follows.
\end{proof}


\appendix

\chapter{Coxeter systems}  \label{fiacg}

\section{Finite irreducible Coxeter systems}

\begin{figure}[!hbtp] 
\[
\xymatrix @C=3pc { *=0{\bullet} \ar@{-}[r]^{}_<{s_1} & *=0{\bullet} \ar@{-}[r]^{}_<{s_2}& *=0{\bullet} \ar@{--}[r]^{}_<{s_3} & *=0{\bullet} \ar@{-}[r]^{}_<{s_{n-1}}_>{s_{n}}  & *=0{\bullet} }
\]
$$A_n \quad (n \geq 1)$$

\[
\xymatrix @C=3pc { *=0{\bullet} \ar@{-}[r]^{4}_<{s_0} & *=0{\bullet} \ar@{-}[r]^{}_<{s_1}& *=0{\bullet} \ar@{--}[r]^{}_<{s_2} & *=0{\bullet} \ar@{-}[r]^{}_<{s_{n-2}}_>{s_{n-1}}  & *=0{\bullet} }
\]
$$B_n \quad (n \geq 2)$$

\[
\xymatrix @C=3pc {                & *=0{\bullet} \ar@{-}[d]^<{s_1}   &         &         & \\
*=0{\bullet} \ar@{-}[r]^{}_<{s_2} & *=0{\bullet} \ar@{-}[r]^{}_<{s_3}& *=0{\bullet} \ar@{--}[r]^{}_<{s_4} & *=0{\bullet} \ar@{-}[r]^{}_<{s_{n-1}}_>{s_{n}}  & *=0{\bullet} }
\]
$$D_n \quad (n \geq 4)$$

\caption{Finite irreducible Coxeter systems (part I).} \label{fig ficgI}
\end{figure}

\newpage

\begin{figure}[!hbtp]

\[
\xymatrix @C=3pc {              &                                & *=0{\bullet} \ar@{-}[d]^<{}    &               &               \\
*=0{\bullet} \ar@{-}[r]^{}_<{}  & *=0{\bullet} \ar@{-}[r]^{}_<{} & *=0{\bullet} \ar@{-}[r]^{}_<{} & *=0{\bullet} \ar@{-}[r]^{}_<{}_>{}  & *=0{\bullet} }
\]
$$E_6$$

\[
\xymatrix @C=3pc {              &                                & *=0{\bullet} \ar@{-}[d]^<{}    &     &     &               \\
*=0{\bullet} \ar@{-}[r]^{}_<{}  & *=0{\bullet} \ar@{-}[r]^{}_<{} & *=0{\bullet} \ar@{-}[r]^{}_<{} & *=0{\bullet} \ar@{-}[r]^{}_<{} & *=0{\bullet} \ar@{-}[r]^{}_<{}_>{}  & *=0{\bullet} }
\]
$$E_7$$

\[
\xymatrix @C=3pc {              &                                & *=0{\bullet} \ar@{-}[d]^<{}    &     &     &               \\
*=0{\bullet} \ar@{-}[r]^{}_<{}  & *=0{\bullet} \ar@{-}[r]^{}_<{} & *=0{\bullet} \ar@{-}[r]^{}_<{} & *=0{\bullet} \ar@{-}[r]^{}_<{} & *=0{\bullet} \ar@{-}[r]^{}_<{} & *=0{\bullet} \ar@{-}[r]^{}_<{}_>{}  & *=0{\bullet} }
\]
$$E_8$$

\[
\xymatrix @C=3pc { *=0{\bullet} \ar@{-}[r]^{}_<{} & *=0{\bullet} \ar@{-}[r]^{4}_<{} &*=0{\bullet} \ar@{-}[r]^{}_<{}_>{}  & *=0{\bullet} } 
\]
$$F_4$$

\[
\xymatrix @C=3pc { *=0{\bullet} \ar@{-}[r]^{5}_<{} & *=0{\bullet} \ar@{-}[r]^{}_<{}_>{}  & *=0{\bullet} }
\]
$$H_3$$

\[
\xymatrix @C=3pc { *=0{\bullet} \ar@{-}[r]^{5}_<{} & *=0{\bullet} \ar@{-}[r]^{}_<{} &*=0{\bullet} \ar@{-}[r]^{}_<{}_>{}  & *=0{\bullet} }
\]
$$H_4$$

\[ 
\xymatrix @C=3pc { *=0{\bullet} \ar@{-}[r]^{m}_<{} & *=0{\bullet} } 
\]
$$I_2(m) \quad (m \geq 3)$$
\caption{Finite irreducible Coxeter systems (part II).} \label{fig ficgII}
\end{figure}

\newpage

\section{Affine Coxeter systems}

\begin{figure}[!htbp]
\[ 
\xymatrix @C=3pc { *=0{\bullet} \ar@{-}[r]^{\infty}_<{} & *=0{\bullet} } 
\]
$$\widetilde{A_1}$$

\[
\xymatrix @C=3pc {                &                                  & *=0{\bullet} \ar@{-}[dll]^<{s_0}     &                  &    \\
*=0{\bullet} \ar@{-}[r]^{}_<{s_1} & *=0{\bullet} \ar@{-}[r]^{}_<{s_2}& *=0{\bullet} \ar@{--}[r]^{}_<{s_3} & *=0{\bullet} \ar@{-}[r]^{}_<{s_{n-2}}_>{s_{n-1}}  & *=0{\bullet} \ar@{-}[ull] }
\]
$$\widetilde{A_{n-1}} \quad (n \geq 3)$$

\[
\xymatrix @C=3pc {                 &                                  &                                    & *=0{\bullet} \ar@{-}[d]^<{s_n}                                                                                &    \\
*=0{\bullet} \ar@{-}[r]^{4}_<{s_0} & *=0{\bullet} \ar@{-}[r]^{}_<{s_1}& *=0{\bullet} \ar@{--}[r]^{}_<{s_2} & *=0{\bullet} \ar@{-}[r]^{}_<{s_{n-1}}_>{s_{n}}  & *=0{\bullet} }
\]
$$\widetilde{B_n} \quad (n \geq 3)$$

\[
\xymatrix @C=3pc { *=0{\bullet} \ar@{-}[r]^{4}_<{s_0} & *=0{\bullet} \ar@{-}[r]^{}_<{s_1}& *=0{\bullet} \ar@{--}[r]^{}_<{s_2} & *=0{\bullet} \ar@{-}[r]^{4}_<{s_{n-1}}_>{s_{n}}  & *=0{\bullet} }
\]
$$\widetilde{C_n} \quad (n \geq 2)$$

\[
\xymatrix @C=3pc {                 & *=0{\bullet} \ar@{-}[d]^<{s_0}   &                                    & *=0{\bullet} \ar@{-}[d]^<{s_n}                                                                                &    \\
*=0{\bullet} \ar@{-}[r]^{}_<{s_1} & *=0{\bullet} \ar@{-}[r]^{}_<{s_2}& *=0{\bullet} \ar@{--}[r]^{}_<{s_3} & *=0{\bullet} \ar@{-}[r]^{}_<{s_{n-2}}_>{s_{n-1}}  & *=0{\bullet} }
\]
$$\widetilde{D_n} \quad (n \geq 4)$$

\caption{Affine Coxeter systems (part I).}

\end{figure}
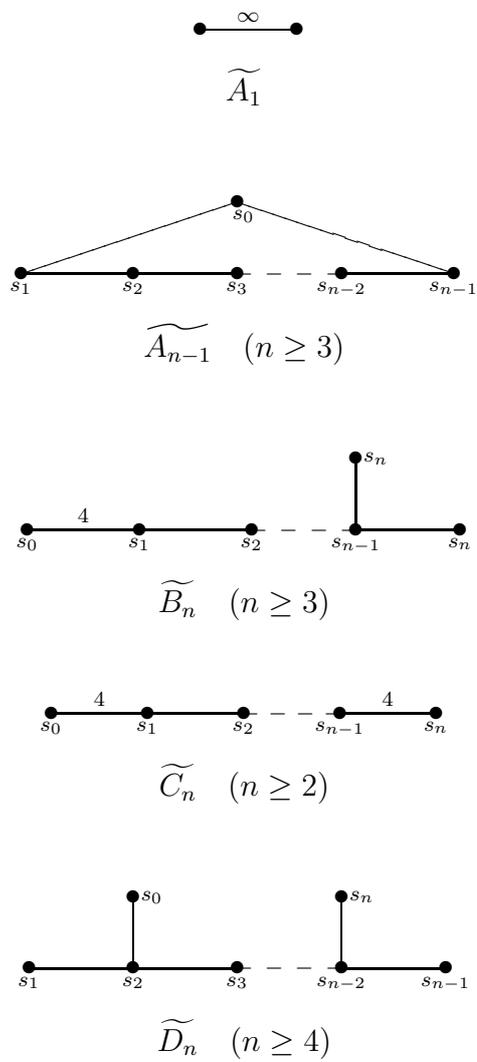

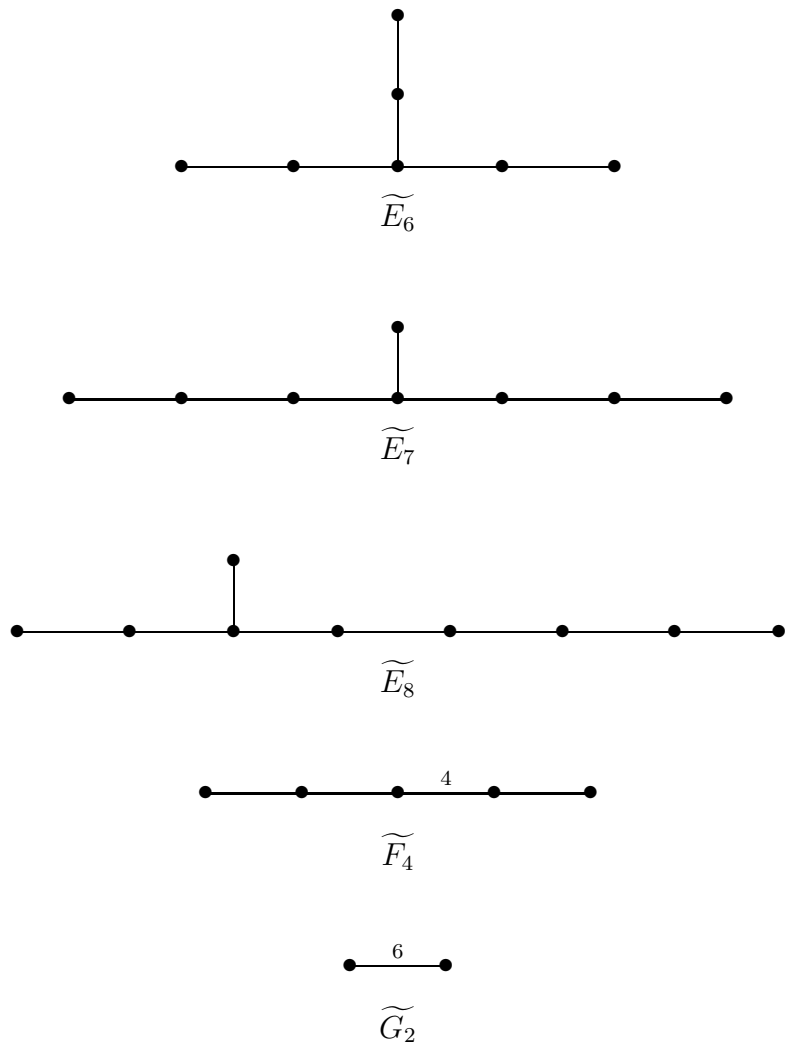
\begin{figure}[!hbtp]

\[
\xymatrix @C=3pc {              &                                & *=0{\bullet} \ar@{-}[d]^<{}    &               &               \\
                                &                                & *=0{\bullet} \ar@{-}[d]^<{}    &               &               \\
*=0{\bullet} \ar@{-}[r]^{}_<{}  & *=0{\bullet} \ar@{-}[r]^{}_<{} & *=0{\bullet} \ar@{-}[r]^{}_<{} & *=0{\bullet} \ar@{-}[r]^{}_<{}_>{}  & *=0{\bullet} }
\]
$$\widetilde{E_6}$$

\[
\xymatrix @C=3pc {              &                                &                                & *=0{\bullet} \ar@{-}[d]^<{}    &     &     &               \\
*=0{\bullet} \ar@{-}[r]^{}_<{}  & *=0{\bullet} \ar@{-}[r]^{}_<{} & *=0{\bullet} \ar@{-}[r]^{}_<{} & *=0{\bullet} \ar@{-}[r]^{}_<{} & *=0{\bullet} \ar@{-}[r]^{}_<{} & *=0{\bullet} \ar@{-}[r]^{}_<{}_>{}  & *=0{\bullet} }
\]
$$\widetilde{E_7}$$

\[
\xymatrix @C=3pc {              &                                & *=0{\bullet} \ar@{-}[d]^<{}    &     &     &    &           \\
*=0{\bullet} \ar@{-}[r]^{}_<{}  & *=0{\bullet} \ar@{-}[r]^{}_<{} & *=0{\bullet} \ar@{-}[r]^{}_<{} & *=0{\bullet} \ar@{-}[r]^{}_<{}& *=0{\bullet} \ar@{-}[r]^{}_<{} & *=0{\bullet} \ar@{-}[r]^{}_<{} & *=0{\bullet} \ar@{-}[r]^{}_<{}_>{}  & *=0{\bullet} }
\]
$$\widetilde{E_8}$$

\[
\xymatrix @C=3pc {*=0{\bullet} \ar@{-}[r]^{}_<{}  & *=0{\bullet} \ar@{-}[r]^{}_<{} & *=0{\bullet} \ar@{-}[r]^{4}_<{} &*=0{\bullet} \ar@{-}[r]^{}_<{}_>{}  & *=0{\bullet} }
\]
$$\widetilde{F_4}$$

\[
\xymatrix @C=3pc { *=0{\bullet} \ar@{-}[r]^{6}_<{}_>{}  & *=0{\bullet} }
\]
$$\widetilde{G_2}$$

\caption{Affine Coxeter systems (part II).}
\end{figure}

\listoffigures

\addcontentsline{toc}{chapter}{Bibliography}
\bibliographystyle{alpha}
\nocite{*}
\bibliography{bibliography}

\def\cprime{$'$}
\begin{thebibliography}{Mar02b}

\bibitem[Abr08]{abra}
Samson Abramsky.
\newblock Temperley-{L}ieb algebra: from knot theory to logic and computation
  via quantum mechanics.
\newblock In {\em Mathematics of quantum computation and quantum technology},
  Chapman \& Hall/CRC Appl. Math. Nonlinear Sci. Ser., pages 515--558. Chapman
  \& Hall/CRC, Boca Raton, FL, 2008.

\bibitem[BB05]{bb-ccg}
Anders Bj{\"o}rner and Francesco Brenti.
\newblock {\em Combinatorics of {C}oxeter groups}, volume 231 of {\em Graduate
  Texts in Mathematics}.
\newblock Springer, New York, 2005.

\bibitem[BCM06]{bcm-sm}
Francesco Brenti, Fabrizio Caselli, and Mario Marietti.
\newblock Special matchings and {K}azhdan-{L}usztig polynomials.
\newblock {\em Adv. Math.}, 202(2):555--601, 2006.

\bibitem[BI06]{bi-lp}
Francesco Brenti and Federico Incitti.
\newblock Lattice paths, lexicographic correspondence and {K}azhdan-{L}usztig
  polynomials.
\newblock {\em J. Algebra}, 303(2):742--762, 2006.

\bibitem[BJS93]{bjs-cps}
Sara~C. Billey, William Jockusch, and Richard~P. Stanley.
\newblock Some combinatorial properties of {S}chubert polynomials.
\newblock {\em J. Algebraic Combin.}, 2(4):345--374, 1993.

\bibitem[Boe88a]{bb-klp}
Brian~D. Boe.
\newblock Kazhdan-{L}usztig polynomials for {H}ermitian symmetric spaces.
\newblock {\em Trans. Amer. Math. Soc.}, 309(1):279--294, 1988.

\bibitem[Boe88b]{Boe88}
Brian~D. Boe.
\newblock Kazhdan-{L}usztig polynomials for {H}ermitian symmetric spaces.
\newblock {\em Trans. Amer. Math. Soc.}, 309(1):279--294, 1988.

\bibitem[Bre]{bre-algcomb}
Francesco Brenti.
\newblock Graduate course \emph{{C}oxeter {G}roups} held by {P}rofessor
  {F}rancesco {B}renti at {U}niversity of {R}ome ``{S}apienza'', {S}pring 2009.

\bibitem[Bre04]{bf-klp}
Francesco Brenti.
\newblock Kazhdan-{L}usztig polynomials: history problems, and combinatorial
  invariance.
\newblock {\em S\'em. Lothar. Combin.}, 49:Art. B49b, 30 pp. (electronic),
  2002/04.

\bibitem[Bre94a]{Bre94}
Francesco Brenti.
\newblock A combinatorial formula for {K}azhdan-{L}usztig polynomials.
\newblock {\em Invent. Math.}, 118(2):371--394, 1994.

\bibitem[Bre94b]{bre-qep}
Francesco Brenti.
\newblock {$q$}-{E}ulerian polynomials arising from {C}oxeter groups.
\newblock {\em European J. Combin.}, 15(5):417--441, 1994.

\bibitem[Bre97a]{bf-cek}
Francesco Brenti.
\newblock Combinatorial expansions of {K}azhdan-{L}usztig polynomials.
\newblock {\em J. London Math. Soc. (2)}, 55(3):448--472, 1997.

\bibitem[Bre97b]{bre-cpkl}
Francesco Brenti.
\newblock Combinatorial properties of the {K}azhdan-{L}usztig {$R$}-polynomials
  for {$S_n$}.
\newblock {\em Adv. Math.}, 126(1):21--51, 1997.

\bibitem[Bre97c]{Bre97}
Francesco Brenti.
\newblock Combinatorial properties of the {K}azhdan-{L}usztig {$R$}-polynomials
  for {$S_n$}.
\newblock {\em Adv. Math.}, 126(1):21--51, 1997.

\bibitem[Bre98a]{bf-lpk}
Francesco Brenti.
\newblock Lattice paths and {K}azhdan-{L}usztig polynomials.
\newblock {\em J. Amer. Math. Soc.}, 11(2):229--259, 1998.

\bibitem[Bre98b]{Bre98}
Francesco Brenti.
\newblock Lattice paths and {K}azhdan-{L}usztig polynomials.
\newblock {\em J. Amer. Math. Soc.}, 11(2):229--259, 1998.

\bibitem[Bre98c]{bf-ulb}
Francesco Brenti.
\newblock Upper and lower bounds for {K}azhdan-{L}usztig polynomials.
\newblock {\em European J. Combin.}, 19(3):283--297, 1998.

\bibitem[Bre00]{bf-ark}
Francesco Brenti.
\newblock Approximation results for {K}azhdan-{L}usztig polynomials.
\newblock In {\em Combinatorial methods in representation theory ({K}yoto,
  1998)}, volume~28 of {\em Adv. Stud. Pure Math.}, pages 55--79. Kinokuniya,
  Tokyo, 2000.

\bibitem[Bre02a]{bf-klr}
Francesco Brenti.
\newblock Kazhdan-{L}usztig and {$R$}-polynomials, {Y}oung's lattice, and
  {D}yck partitions.
\newblock {\em Pacific J. Math.}, 207(2):257--286, 2002.

\bibitem[Bre02b]{Bre02}
Francesco Brenti.
\newblock Kazhdan-{L}usztig and {$R$}-polynomials, {Y}oung's lattice, and
  {D}yck partitions.
\newblock {\em Pacific J. Math.}, 207(2):257--286, 2002.

\bibitem[Bre03]{bf-pkic}
Francesco Brenti.
\newblock {$P$}-kernels, {IC} bases and {K}azhdan-{L}usztig polynomials.
\newblock {\em J. Algebra}, 259(2):613--627, 2003.

\bibitem[Bre09a]{bf-pkl}
Francesco Brenti.
\newblock Parabolic {K}azhdan-{L}usztig polynomials for {H}ermitian symmetric
  pairs.
\newblock {\em Trans. Amer. Math. Soc.}, 361(4):1703--1729, 2009.

\bibitem[Bre09b]{bre-pkl}
Francesco Brenti.
\newblock Parabolic {K}azhdan-{L}usztig polynomials for {H}ermitian symmetric
  pairs.
\newblock {\em Trans. Amer. Math. Soc.}, 361(4):1703--1729, 2009.

\bibitem[BS00]{BreSim}
Francesco Brenti and Rodica Simion.
\newblock Explicit formulae for some {K}azhdan-{L}usztig polynomials.
\newblock {\em J. Algebraic Combin.}, 11(3):187--196, 2000.

\bibitem[BW01]{BW01}
Sara~C. Billey and Gregory~S. Warrington.
\newblock Kazhdan-{L}usztig polynomials for 321-hexagon-avoiding permutations.
\newblock {\em J. Algebraic Combin.}, 13(2):111--136, 2001.

\bibitem[BW03]{BW02}
Sara~C. Billey and Gregory~S. Warrington.
\newblock Maximal singular loci of {S}chubert varieties in {${\rm SL}(n)/B$}.
\newblock {\em Trans. Amer. Math. Soc.}, 355(10):3915--3945 (electronic), 2003.

\bibitem[Car94]{Car94}
James~B. Carrell.
\newblock The {B}ruhat graph of a {C}oxeter group, a conjecture of {D}eodhar,
  and rational smoothness of {S}chubert varieties.
\newblock In {\em Algebraic groups and their generalizations: classical methods
  ({U}niversity {P}ark, {PA}, 1991)}, volume~56 of {\em Proc. Sympos. Pure
  Math.}, pages 53--61. Amer. Math. Soc., Providence, RI, 1994.

\bibitem[Cas03]{Cas}
Fabrizio Caselli.
\newblock Proof of two conjectures of {B}renti and {S}imion on
  {K}azhdan-{L}usztig polynomials.
\newblock {\em J. Algebraic Combin.}, 18(3):171--187, 2003.

\bibitem[CJ03]{caut}
S.~Cautis and D.~M. Jackson.
\newblock The matrix of chromatic joins and the {T}emperley-{L}ieb algebra.
\newblock {\em J. Combin. Theory Ser. B}, 89(1):109--155, 2003.

\bibitem[dC96]{dc-sac}
Fokko du~Cloux.
\newblock The state of the art in the computation of {K}azhdan-{L}usztig
  polynomials.
\newblock {\em Appl. Algebra Engrg. Comm. Comput.}, 7(3):211--219, 1996.
\newblock Computational methods in Lie theory (Essen, 1994).

\bibitem[dC99]{dc-sop}
Fokko du~Cloux.
\newblock Some open problems in the theory of {K}azhdan-{L}usztig polynomials
  and {C}oxeter groups.
\newblock In {\em Computational methods for representations of groups and
  algebras ({E}ssen, 1997)}, volume 173 of {\em Progr. Math.}, pages 201--210.
  Birkh\"auser, Basel, 1999.

\bibitem[dC02]{dc-ckl}
Fokko du~Cloux.
\newblock Computing {K}azhdan-{L}usztig polynomials for arbitrary {C}oxeter
  groups.
\newblock {\em Experiment. Math.}, 11(3):371--381, 2002.

\bibitem[Deg91]{degu}
Tetsuo Deguchi.
\newblock Multivariable vertex models associated with the {T}emperley-{L}ieb
  algebra.
\newblock {\em Phys. Lett. A}, 159(3):163--169, 1991.

\bibitem[Del06]{de-cik}
Ewan Delanoy.
\newblock Combinatorial invariance of {K}azhdan-{L}usztig polynomials on
  intervals starting from the identity.
\newblock {\em J. Algebraic Combin.}, 24(4):437--463, 2006.

\bibitem[Deo85]{Deo2}
Vinay~V. Deodhar.
\newblock On some geometric aspects of {B}ruhat orderings. {I}. {A} finer
  decomposition of {B}ruhat cells.
\newblock {\em Invent. Math.}, 79(3):499--511, 1985.

\bibitem[Deo87]{deo-sga}
Vinay~V. Deodhar.
\newblock On some geometric aspects of {B}ruhat orderings. {II}. {T}he
  parabolic analogue of {K}azhdan-{L}usztig polynomials.
\newblock {\em J. Algebra}, 111(2):483--506, 1987.

\bibitem[Deo90]{Deo3}
Vinay~V. Deodhar.
\newblock A combinatorial setting for questions in {K}azhdan-{L}usztig theory.
\newblock {\em Geom. Dedicata}, 36(1):95--119, 1990.

\bibitem[Du96]{du-icb}
Jie Du.
\newblock Global {IC} bases for quantum linear groups.
\newblock {\em J. Pure Appl. Algebra}, 114(1):25--37, 1996.

\bibitem[Dye93]{Dy}
M.~J. Dyer.
\newblock Hecke algebras and shellings of {B}ruhat intervals.
\newblock {\em Compositio Math.}, 89(1):91--115, 1993.

\bibitem[Dye97]{Dye97}
M.~J. Dyer.
\newblock On coefficients of {$q$} in {K}azhdan-{L}usztig polynomials.
\newblock In {\em Algebraic groups and {L}ie groups}, volume~9 of {\em Austral.
  Math. Soc. Lect. Ser.}, pages 189--194. Cambridge Univ. Press, Cambridge,
  1997.

\bibitem[Fan95]{fan-phd}
C.~Kenneth Fan.
\newblock {\em A {H}ecke algebra quotient and properties of commutative
  elements of a {W}eyl group}.
\newblock ProQuest LLC, Ann Arbor, MI, 1995.
\newblock Thesis (Ph.D.)--Massachusetts Institute of Technology.

\bibitem[Fan96]{fan-haq}
C.~K. Fan.
\newblock A {H}ecke algebra quotient and some combinatorial applications.
\newblock {\em J. Algebraic Combin.}, 5(3):175--189, 1996.

\bibitem[FG97]{fg-mtl}
C.~K. Fan and R.~M. Green.
\newblock Monomials and {T}emperley-{L}ieb algebras.
\newblock {\em J. Algebra}, 190(2):498--517, 1997.

\bibitem[FK09]{fend}
Paul Fendley and Vyacheslav Krushkal.
\newblock Tutte chromatic identities from the {T}emperley-{L}ieb algebra.
\newblock {\em Geom. Topol.}, 13(2):709--741, 2009.

\bibitem[GL99]{gl-cbh}
R.~M. Green and J.~Losonczy.
\newblock Canonical bases for {H}ecke algebra quotients.
\newblock {\em Math. Res. Lett.}, 6(2):213--222, 1999.

\bibitem[GL00]{gl-ppk}
R.~M. Green and J.~Losonczy.
\newblock A projection property for {K}azhdan-{L}usztig bases.
\newblock {\em Internat. Math. Res. Notices}, (1):23--34, 2000.

\bibitem[GL01]{gl-fck}
R.~M. Green and J.~Losonczy.
\newblock Fully commutative {K}azhdan-{L}usztig cells.
\newblock {\em Ann. Inst. Fourier (Grenoble)}, 51(4):1025--1045, 2001.

\bibitem[Gra]{gr-phd}
J.~J. Graham.
\newblock Modular representations of {H}ecke algebras and related algebras.
\newblock {\em Ph.D. thesis}.

\bibitem[Gre98]{gr-gtld}
R.~M. Green.
\newblock Generalized {T}emperley-{L}ieb algebras and decorated tangles.
\newblock {\em J. Knot Theory Ramifications}, 7(2):155--171, 1998.

\bibitem[Gre07]{gre-gjt}
R.~M. Green.
\newblock Generalized {J}ones traces and {K}azhdan-{L}usztig bases.
\newblock {\em J. Pure Appl. Algebra}, 211(3):744--772, 2007.

\bibitem[Gre09]{g-lck}
R.~M. Green.
\newblock Leading coefficients of {K}azhdan-{L}usztig polynomials and fully
  commutative elements.
\newblock {\em J. Algebraic Combin.}, 30(2):165--171, 2009.

\bibitem[Hum90]{Hum}
James~E. Humphreys.
\newblock {\em Reflection groups and {C}oxeter groups}, volume~29 of {\em
  Cambridge Studies in Advanced Mathematics}.
\newblock Cambridge University Press, Cambridge, 1990.

\bibitem[Inc06]{if-cik}
Federico Incitti.
\newblock On the combinatorial invariance of {K}azhdan-{L}usztig polynomials.
\newblock {\em J. Combin. Theory Ser. A}, 113(7):1332--1350, 2006.

\bibitem[Inc07]{if-mci}
Federico Incitti.
\newblock More on the combinatorial invariance of {K}azhdan-{L}usztig
  polynomials.
\newblock {\em J. Combin. Theory Ser. A}, 114(3):461--482, 2007.

\bibitem[Jon85]{jo-pik}
Vaughan F.~R. Jones.
\newblock A polynomial invariant for knots via von {N}eumann algebras.
\newblock {\em Bull. Amer. Math. Soc. (N.S.)}, 12(1):103--111, 1985.

\bibitem[Jon87]{jo-har}
V.~F.~R. Jones.
\newblock Hecke algebra representations of braid groups and link polynomials.
\newblock {\em Ann. of Math. (2)}, 126(2):335--388, 1987.

\bibitem[Jon09a]{jb-klp}
Brant~C. Jones.
\newblock Kazhdan-{L}usztig polynomials for maximally-clustered
  hexagon-avoiding permutations.
\newblock {\em J. Algebra}, 322(10):3459--3477, 2009.

\bibitem[Jon09b]{jb-lck}
Brant~C. Jones.
\newblock Leading coefficients of {K}azhdan-{L}usztig polynomials for {D}eodhar
  elements.
\newblock {\em J. Algebraic Combin.}, 29(2):229--260, 2009.

\bibitem[KL79]{kl-rcg}
David Kazhdan and George Lusztig.
\newblock Representations of {C}oxeter groups and {H}ecke algebras.
\newblock {\em Invent. Math.}, 53(2):165--184, 1979.

\bibitem[KL00]{Ki-La}
Alexander Kirillov, Jr. and Alain Lascoux.
\newblock Factorization of {K}azhdan-{L}usztig elements for {G}rassmanians.
\newblock In {\em Combinatorial methods in representation theory ({K}yoto,
  1998)}, volume~28 of {\em Adv. Stud. Pure Math.}, pages 143--154. Kinokuniya,
  Tokyo, 2000.

\bibitem[KL08]{loui}
Louis~H. Kauffman and Samuel~J. Lomonaco, Jr.
\newblock The {F}ibonacci model and the {T}emperley-{L}ieb algebra.
\newblock {\em Internat. J. Modern Phys. B}, 22(29):5065--5080, 2008.

\bibitem[KT03]{kauff}
Louis Kauffman and Robin Thomas.
\newblock Temperley-{L}ieb algebras and the four-color theorem.
\newblock {\em Combinatorica}, 23(4):653--667, 2003.

\bibitem[Las95]{la-pkl}
Alain Lascoux.
\newblock Polyn\^omes de {K}azhdan-{L}usztig pour les vari\'et\'es de
  {S}chubert vexillaires.
\newblock {\em C. R. Acad. Sci. Paris S\'er. I Math.}, 321(6):667--670, 1995.

\bibitem[Lev90]{levy}
Dan Levy.
\newblock Structure of {T}emperley-{L}ieb algebras and its application to
  {$2$}{D} statistical models.
\newblock {\em Phys. Rev. Lett.}, 64(5):499--502, 1990.

\bibitem[Lic92]{libr}
W.~B.~R. Lickorish.
\newblock Calculations with the {T}emperley-{L}ieb algebra.
\newblock {\em Comment. Math. Helv.}, 67(4):571--591, 1992.

\bibitem[Lin96]{link}
Jon Links.
\newblock Temperley-{L}ieb algebra and a new integrable electronic model.
\newblock {\em J. Phys. A}, 29(4):L69--L73, 1996.

\bibitem[Los00]{los-klb}
Jozsef Losonczy.
\newblock The {K}azhdan-{L}usztig basis and the {T}emperley-{L}ieb quotient in
  type {D}.
\newblock {\em J. Algebra}, 233(1):1--15, 2000.

\bibitem[LS81]{L-S}
Alain Lascoux and Marcel-Paul Sch{\"u}tzenberger.
\newblock Polyn\^omes de {K}azhdan \& {L}usztig pour les grassmanniennes.
\newblock In {\em Young tableaux and {S}chur functors in algebra and geometry
  ({T}oru\'n, 1980)}, volume~87 of {\em Ast\'erisque}, pages 249--266. Soc.
  Math. France, Paris, 1981.

\bibitem[LS96]{ls-tbg}
Alain Lascoux and Marcel-Paul Sch{\"u}tzenberger.
\newblock Treillis et bases des groupes de {C}oxeter.
\newblock {\em Electron. J. Combin.}, 3(2):Research paper 27, approx. 35 pp.
  (electronic), 1996.

\bibitem[LT00]{Le-Th}
Bernard Leclerc and Jean-Yves Thibon.
\newblock Littlewood-{R}ichardson coefficients and {K}azhdan-{L}usztig
  polynomials.
\newblock In {\em Combinatorial methods in representation theory ({K}yoto,
  1998)}, volume~28 of {\em Adv. Stud. Pure Math.}, pages 155--220. Kinokuniya,
  Tokyo, 2000.

\bibitem[Mar88]{map}
P.~P. Martin.
\newblock Temperley-{L}ieb algebra, group theory and the {P}otts model.
\newblock {\em J. Phys. A}, 21(3):577--591, 1988.

\bibitem[Mar90]{mapp}
P.~P. Martin.
\newblock Temperley-{L}ieb algebras and the long distance properties of
  statistical mechanical models.
\newblock {\em J. Phys. A}, 23(1):7--30, 1990.

\bibitem[Mar02a]{Mar}
Mario Marietti.
\newblock Closed product formulas for certain {$R$}-polynomials.
\newblock {\em European J. Combin.}, 23(1):57--62, 2002.

\bibitem[Mar02b]{m-cpf}
Mario Marietti.
\newblock Closed product formulas for certain {$R$}-polynomials.
\newblock {\em European J. Combin.}, 23(1):57--62, 2002.

\bibitem[Mar06]{m-bek}
Mario Marietti.
\newblock Boolean elements in {K}azhdan-{L}usztig theory.
\newblock {\em J. Algebra}, 295(1):1--26, 2006.

\bibitem[Nic06]{nich}
A.~Nichols.
\newblock The {T}emperley-{L}ieb algebra and its generalizations in the {P}otts
  and {$XXZ$} models.
\newblock {\em J. Stat. Mech. Theory Exp.}, (1):P01003, 46 pp. (electronic),
  2006.

\bibitem[Pol99a]{pp-cak}
Patrick Polo.
\newblock Construction of arbitrary {K}azhdan-{L}usztig polynomials in
  symmetric groups.
\newblock {\em Represent. Theory}, 3:90--104 (electronic), 1999.

\bibitem[Pol99b]{Pol99}
Patrick Polo.
\newblock Construction of arbitrary {K}azhdan-{L}usztig polynomials in
  symmetric groups.
\newblock {\em Represent. Theory}, 3:90--104 (electronic), 1999.

\bibitem[Shi97]{shi-ece}
Jian-Yi Shi.
\newblock The enumeration of {C}oxeter elements.
\newblock {\em J. Algebraic Combin.}, 6(2):161--171, 1997.

\bibitem[Shi03]{shi-fce}
Jian-Yi Shi.
\newblock Fully commutative elements and {K}azhdan-{L}usztig cells in the
  finite and affine {C}oxeter groups.
\newblock {\em Proc. Amer. Math. Soc.}, 131(11):3371--3378 (electronic), 2003.

\bibitem[Shi05]{shi-fceii}
Jian-Yi Shi.
\newblock Fully commutative elements and {K}azhdan-{L}usztig cells in the
  finite and affine {C}oxeter groups. {II}.
\newblock {\em Proc. Amer. Math. Soc.}, 133(9):2525--2531, 2005.

\bibitem[SSV98]{SSV}
B.~Shapiro, M.~Shapiro, and A.~Vainshtein.
\newblock Kazhdan-{L}usztig polynomials for certain varieties of incomplete
  flags.
\newblock In {\em Proceedings of the 7th {C}onference on {F}ormal {P}ower
  {S}eries and {A}lgebraic {C}ombinatorics ({N}oisy-le-{G}rand, 1995)}, volume
  180, pages 345--355, 1998.

\bibitem[Sta97]{sta-ec1}
Richard~P. Stanley.
\newblock {\em Enumerative combinatorics. {V}ol. 1}, volume~49 of {\em
  Cambridge Studies in Advanced Mathematics}.
\newblock Cambridge University Press, Cambridge, 1997.
\newblock With a foreword by Gian-Carlo Rota, Corrected reprint of the 1986
  original.

\bibitem[Ste96]{ste-fce}
John~R. Stembridge.
\newblock On the fully commutative elements of {C}oxeter groups.
\newblock {\em J. Algebraic Combin.}, 5(4):353--385, 1996.

\bibitem[Ste97]{ste-sca}
John~R. Stembridge.
\newblock Some combinatorial aspects of reduced words in finite {C}oxeter
  groups.
\newblock {\em Trans. Amer. Math. Soc.}, 349(4):1285--1332, 1997.

\bibitem[Tag95a]{th-nnf}
Hiroyuki Tagawa.
\newblock On the non-negativity of the first coefficient of {K}azhdan-{L}usztig
  polynomials.
\newblock {\em J. Algebra}, 177(3):698--707, 1995.

\bibitem[Tag95b]{Tag94a}
Hiroyuki Tagawa.
\newblock On the non-negativity of the first coefficient of {K}azhdan-{L}usztig
  polynomials.
\newblock {\em J. Algebra}, 177(3):698--707, 1995.

\bibitem[Tem93]{nere}
H.~N.~V. Temperley.
\newblock New representations of the {T}emperley-{L}ieb algebra with
  applications.
\newblock In {\em Low-dimensional topology and quantum field theory
  ({C}ambridge, 1992)}, volume 315 of {\em NATO Adv. Sci. Inst. Ser. B Phys.},
  pages 203--212. Plenum, New York, 1993.

\bibitem[TL71]{tl-pcp}
H.~N.~V. Temperley and E.~H. Lieb.
\newblock Relations between the ``percolation'' and ``colouring'' problem and
  other graph-theoretical problems associated with regular planar lattices:
  some exact results for the ``percolation'' problem.
\newblock {\em Proc. Roy. Soc. London Ser. A}, 322(1549):251--280, 1971.

\bibitem[Vin06]{clai}
Claire Vincenti.
\newblock Alg\`ebre de {T}emperley-{L}ieb de type {$B$}.
\newblock {\em C. R. Math. Acad. Sci. Paris}, 342(4):233--236, 2006.

\bibitem[War01]{wg-klp}
Gregory~Saunders Warrington.
\newblock {\em Kazhdan-{L}usztig polynomials, pattern avoidance and singular
  loci of {S}chubert varieties}.
\newblock ProQuest LLC, Ann Arbor, MI, 2001.
\newblock Thesis (Ph.D.)--Harvard University.

\bibitem[Wes95]{we-rttl}
B.~W. Westbury.
\newblock The representation theory of the {T}emperley-{L}ieb algebras.
\newblock {\em Math. Z.}, 219(4):539--565, 1995.

\bibitem[Woo09]{wa-pkl}
Alexander Woo.
\newblock Permutations with {K}azhdan-{L}usztig polynomial
  {$P_{id,w}(q)=1+q^h$}.
\newblock {\em Electron. J. Combin.}, 16(2, Special volume in honor of Anders
  Bjorner):Research Paper 10, 32, 2009.
\newblock With an appendix by Sara Billey and Jonathan Weed.

\bibitem[Xi05]{xi-lcc}
Nanhua Xi.
\newblock The leading coefficient of certain {K}azhdan-{L}usztig polynomials of
  the permutation group $s_n$.
\newblock {\em J. Algebra}, 285(1):136--145, 2005.

\bibitem[Zel83]{Zel}
A.~V. Zelevinski{\u\i}.
\newblock Small resolutions of singularities of {S}chubert varieties.
\newblock {\em Funktsional. Anal. i Prilozhen.}, 17(2):75--77, 1983.

\bibitem[Zha06]{yong}
Yong Zhang.
\newblock Teleportation, braid group and {T}emperley-{L}ieb algebra.
\newblock {\em J. Phys. A}, 39(37):11599--11622, 2006.

\bibitem[Zha09]{zhan}
Yong Zhang.
\newblock Braid group, {T}emperley-{L}ieb algebra, and quantum information and
  computation.
\newblock In {\em Advances in quantum computation}, volume 482 of {\em Contemp.
  Math.}, pages 49--89. Amer. Math. Soc., Providence, RI, 2009.

\end{thebibliography}

\end{document}